\documentclass[a4paper,10pt]{amsart}
\pdfoutput=1
\usepackage{etex}

\usepackage[T1]{fontenc}
\usepackage{lmodern}
\usepackage[utf8]{inputenc}
\usepackage{amssymb}
\usepackage{enumitem}
\usepackage{mathrsfs,dsfont,mathtools}
\usepackage{nicefrac}
\usepackage{graphicx}

\usepackage{tikz}
\usetikzlibrary{matrix,arrows,calc,cd}
\tikzset{cd/.style=matrix of math nodes,row sep=2em,column sep=2em, text height=1.5ex, text depth=0.5ex}
\tikzset{cdar/.style=->,auto}
\tikzset{mid/.style={anchor=mid}} 
\tikzset{dar/.style={double,double equal sign distance,-implies}}
\tikzset{narrowfill/.style={inner sep=1pt, fill=white}}

\setlist[enumerate,1]{label=\textup{(\arabic*)}}
\setlist[enumerate,2]{label=\textup{(\alph*)}}

\usepackage[pdftitle={A more general method to classify up to equivariant KK-equivalence II: Computing obstruction classes},
  pdfauthor={Ralf Meyer},
  pdfsubject={Mathematics}
]{hyperref}
\usepackage[lite]{amsrefs}

\renewcommand{\PrintDOI}[1]{\href{http://dx.doi.org/\detokenize{#1}}{doi: \detokenize{#1}}}

\BibSpec{book}{%
  +{}  {\PrintPrimary}                {transition}
  +{,} { \textit}                     {title}
  +{.} { }                            {part}
  +{:} { \textit}                     {subtitle}
  +{,} { \PrintEdition}               {edition}
  +{}  { \PrintEditorsB}              {editor}
  +{,} { \PrintTranslatorsC}          {translator}
  +{,} { \PrintContributions}         {contribution}
  +{,} { }                            {series}
  +{,} { \voltext}                    {volume}
  +{,} { }                            {publisher}
  +{,} { }                            {organization}
  +{,} { }                            {address}
  +{,} { \PrintDateB}                 {date}
  +{,} { }                            {status}
  +{}  { \parenthesize}               {language}
  +{}  { \PrintTranslation}           {translation}
  +{;} { \PrintReprint}               {reprint}
  +{.} { }                            {note}
  +{.} {}                             {transition}
  +{} { \PrintDOI}      {doi}
  +{} { available at \url}            {eprint}
  +{}  {\SentenceSpace \PrintReviews} {review}
}

\usepackage{microtype}

\numberwithin{equation}{section}
\theoremstyle{plain}
\newtheorem{theorem}[equation]{Theorem}
\newtheorem{lemma}[equation]{Lemma}
\newtheorem{proposition}[equation]{Proposition}
\newtheorem{corollary}[equation]{Corollary}
\theoremstyle{definition}

\theoremstyle{remark}
\newtheorem{remark}[equation]{Remark}

\newcommand*{\C}{\mathbb C}
\newcommand*{\Z}{\mathbb Z}
\newcommand*{\N}{\mathbb N}

\newcommand*{\T}{\mathbb T}

\newcommand*{\id}{\textup{id}}

\newcommand*{\KK}{\textup{KK}}

\newcommand*{\K}{\textup{K}}

\newcommand*{\7}[1]{[#1]}
\newcommand*{\tautr}[2]{\tau_{[#1]}^{[#2]}}

\newcommand*{\ev}{\textup{ev}}

\newcommand*{\Cst}{\textup C^*}

\newcommand*{\Star}{\texorpdfstring{$^*$\nb-}{*-}}

\newcommand*{\Cont}{\textup C} 

\newcommand*{\Ab}{\mathfrak{Ab}} 
\newcommand*{\Abel}[1][A]{\mathfrak #1} 
\newcommand*{\Tri}[1][T]{\mathfrak #1}  
\newcommand*{\KKcat}{\mathfrak{KK}}
\newcommand*{\Boot}{\mathfrak{B}}

\newcommand*{\Ideal}{\mathfrak I}
\newcommand*{\Proj}{\mathfrak P} 
\newcommand*{\lad}{\vdash}

\newcommand*{\nb}{\nobreakdash}  

\newcommand*{\congto}{\xrightarrow\sim}
\newcommand{\idealin}{\mathrel{\triangleleft}} 

\newcommand*{\matrixintikz}{\left( \begin{array}{@{}ll@{}}
    \scriptstyle\varphi_{\7{1}}^{\7{1}}&\scriptstyle0\\
    \scriptstyle\varphi_{\7{1}}^{\7{1,2}}&\scriptstyle\varphi_{\7{1,2}}^{\7{1,2}}
  \end{array}\right)}

\mathtoolsset{mathic}

\DeclarePairedDelimiter{\gen}{\langle}{\rangle}
\DeclarePairedDelimiterX{\braket}[2]{\langle}{\rangle}{#1\,\delimsize\vert\,\mathopen{}#2}
\DeclarePairedDelimiterX{\braketop}[3]{\langle}{\rangle}{#1\,\delimsize\vert #2\delimsize\vert\,\mathopen{}#3}
\DeclarePairedDelimiterX{\BRAKET}[2]{\langle}{\rangle}{\!\delimsize\langle#1\,\delimsize\vert\,\mathopen{}#2\delimsize\rangle\!}
\DeclarePairedDelimiterX{\setgiven}[2]{\{}{\}}{#1\,{:}\,\mathopen{}#2}

\DeclareMathOperator{\Ext}{Ext}
\DeclareMathOperator{\Tor}{Tor}
\DeclareMathOperator{\Hom}{Hom}
\DeclareMathOperator{\Aut}{Aut}

\DeclareMathOperator{\coker}{coker}

\newcommand*{\inOb}{\mathrel{\in\in}}
\newcommand*{\defeq}{\mathrel{\vcentcolon=}}

\newcommand*{\into}{\rightarrowtail}
\newcommand*{\prto}{\twoheadrightarrow}

\makeatletter
\newbox\xrat@below
\newbox\xrat@above
\newcommand{\xinto}[1]{
  \setbox\xrat@above=\hbox{\ensuremath{\scriptstyle #1}}%
  \pgfmathsetlengthmacro{\xrat@len}{(\wd\xrat@above)+.6em}%
  \mathrel{\tikz [>->,baseline=-.75ex]
                 \draw (0,0) -- node[above=-2pt] {\box\xrat@above}
                       (\xrat@len,0) ;}}
\newcommand{\xprto}[1]{
  \setbox\xrat@above=\hbox{\ensuremath{\scriptstyle #1}}%
  \pgfmathsetlengthmacro{\xrat@len}{(\wd\xrat@above)+.6em}%
  \mathrel{\tikz [->>,baseline=-.75ex]
                 \draw (0,0) -- node[above=-2pt] {\box\xrat@above}
                       (\xrat@len,0) ;}}                       
\makeatother

\begin{document}

\title[Computing obstruction classes]{A more general method to classify up to equivariant KK-equivalence II: Computing obstruction classes}

\author{Ralf Meyer}
\email{rmeyer2@uni-goettingen.de}

\address{Mathematisches Institut\\
  Georg-August Universit\"at G\"ottingen\\
  Bunsenstra\ss{}e 3--5\\
  37073 G\"ottingen\\
  Germany}

\begin{abstract}
  We describe Universal Coefficient Theorems for the equivariant
  Kasparov theory for \(\Cst\)\nb-algebras
  with an action of the group of integers or over a unique path space,
  using KK-valued invariants.  We compare the resulting classification
  up to equivariant KK-equivalence with the recent classification
  theorem involving a K\nb-theoretic invariant together with an
  obstruction class in a certain \(\Ext^2\)-group
  and with the classification by filtrated \(\K\)\nb-theory.
  This is based on a general theorem that computes these obstruction
  classes.
\end{abstract}


\thanks{Supported by the DFG grant
  ``Classification of non-simple purely infinite C*-algebras.''}

\maketitle

\section{Introduction}
\label{sec:intro}

Objects in a triangulated category, such as equivariant KK-categories,
may be classified up to isomorphism using a primary homological
invariant and a secondary ``obstruction class'' provided they have a
projective resolution of length~\(2\)
in a suitable sense.  This method was applied
in~\cite{Bentmann-Meyer:More_general} to objects in the
circle-equivariant KK-category \(\KKcat^\T\),
the equivariant KK-category \(\KKcat^X\)
for \(\Cst\)\nb-algebras
over a finite, unique path space~\(X\),
and graph \(\Cst\)\nb-algebras
with finite ideal lattice.  Here we compute the obstruction classes
that occur in this classification.  This makes the classification for
objects of \(\KKcat^\T\) and~\(\KKcat^X\) more explicit.

The result suggests, in fact, a different invariant for objects in
these categories that is fine enough to admit a Universal Coefficient
Theorem.  The price to pay is that the invariant uses bivariant
\(\K\)\nb-theory
instead of ordinary \(\K\)\nb-theory.
We explain this for the equivariant bootstrap class
in the category~\(\KKcat^\Z\);
the latter is equivalent to~\(\KKcat^\T\) by Baaj--Skandalis duality.

An object of~\(\KKcat^\Z\)
is a \(\Cst\)\nb-algebra~\(A\)
with a \(\Z\)\nb-action,
which is generated by a single automorphism~\(\alpha\).
The most obvious homological invariant on~\(\KKcat^\Z\)
maps this to the \(\K\)\nb-theory
\(\K_*(A)\)
with the module structure over the group ring \(\Z[x,x^{-1}]\)
of~\(\Z\)
that is induced by~\(\alpha\).
Let~\(\Abel\)
be the category of countable, \(\Z/2\)\nb-graded
modules over \(\Z[x,x^{-1}]\).
This is a stable Abelian category.  The \(\K\)\nb-theory
described above defines a stable homological functor
\(F^\Z_\K\colon \KKcat^\Z \to \Abel\).
The category~\(\Abel\)
has cohomological dimension~\(2\),
that is, any object has a projective resolution of length~\(2\).
Let~\(\Abel\delta\)
be the additive category of pairs~\((A,\delta)\)
with \(A\inOb\Abel\)
and \(\delta\in\Ext^2_{\Abel}(\Sigma A, A)\);
morphisms from \((A,\delta)\)
to \((A',\delta')\)
in~\(\Abel\delta\)
are morphisms~\(f\)
from~\(A\)
to~\(A'\)
in~\(\Abel\)
with \(\delta' f=f\delta\).
It is shown in~\cite{Bentmann-Meyer:More_general} that isomorphism
classes of objects in the bootstrap class in~\(\KKcat^\Z\)
are in bijection with isomorphism classes of objects
in~\(\Abel\delta\).
In particular, any object~\(A\)
of~\(\Abel\)
lifts to an object in the bootstrap class in~\(\KKcat^\Z\).
The lifting is, however, not unique: different liftings of~\(A\)
are in bijection with \(\Ext^2_{\Abel}(\Sigma A, A)\);
here two liftings \(B_1,B_2\)
are identified if there is an isomorphism \(B_1\cong B_2\)
that induces the identity map on \(A=F^\Z_\K(B_1)=F^\Z_\K(B_2)\).

In this article, we compute the obstruction class of an object
of~\(\KKcat^\Z\)
explicitly.  This question remained open
in~\cite{Bentmann-Meyer:More_general}.  Let \((A,\alpha)\)
be a \(\Cst\)\nb-algebra
with an automorphism~\(\alpha\)
as above.  The exact sequence in the Universal Coefficient Theorem
splits, so that
\[
\KK_0(A,A) \cong
\Hom\bigl(\K_*(A),\K_*(A)\bigr)
\oplus \Ext^1\bigl(\K_{*+1}(A),\K_*(A)\bigr).
\]
While the splitting above is not natural, it is canonical enough to
associate well defined elements
\(\alpha_0\in\Hom\bigl(\K_*(A),\K_*(A)\bigr)\)
and \(\alpha_1\in \Ext^1\bigl(\K_{*+1}(A),\K_*(A)\bigr)\)
to~\(\alpha\).
Here~\(\alpha_0\)
is the action on \(\K_*(A)\)
induced by~\(\alpha\),
which is part of the \(\Z[x,x^{-1}]\)-module
structure on~\(F^\Z_\K(A)\).
We show how to compute the obstruction class from~\(\alpha_1\).
Namely, it is~\(\mu^*(\alpha_1)\) for a canonical map
\[
\mu^*\colon \Ext^1\bigl(\K_{*+1}(A),\K_*(A)\bigr) \to
\Ext^2_{\Z[x,x^{-1}]}\bigl(\K_{*+1}(A),\K_*(A)\bigr).
\]
Thus an object \((A,\alpha)\)
of \(\KKcat^\Z\)
is determined uniquely up to isomorphism by~\(A\)
as an object of~\(\KKcat\)
together with the class~\([\alpha]\)
of~\(\alpha\)
in \(\KK_0(A,A)\).
We treat the pair \((A,[\alpha])\)
as an object of a certain exact category~\(\KKcat[\Z]\)
with a suspension automorphism.  The forgetful functor
\(F^\Z\colon \KKcat^\Z \to \KKcat[\Z]\)
is a stable homological functor.  The fact that it gives a complete
invariant for objects of~\(\KKcat^\Z\)
follows from a Universal Coefficient Theorem computing
\(\KK^\Z_*(A,B)\).
This Universal Coefficient Theorem is already proven
in~\cite{Meyer-Nest:Homology_in_KK}, and it is shown there to be
equivalent to a Pimsner--Voiculescu like exact sequence for
\(\KK^\Z_*(A,B)\).

The Universal Coefficient Theorem for the invariant
\(F^\Z\colon \KKcat^\Z \to \KKcat[\Z]\)
is based on projective resolutions of length~\(1\)
for the stable homological ideal \(\ker F^\Z\).
We use these resolutions to compute the obstruction class related to
the homological invariant \(F^\Z_\K(A) \defeq \K_*(A)\).
Namely, let \(0\to B_1 \to B_0\to A \to 0\)
be a projective resolution of length~\(1\)
for the ideal \(\ker F^\Z\)
in \(\KKcat^\Z\).
This remains exact with respect to the larger stable homological ideal
\(\ker F^\Z_\K\).
The objects \(B_1\)
and~\(B_0\)
are no longer projective for \(\ker F^\Z_\K\),
but they have projective resolutions of length~\(1\).
Thus \(\KK^\Z_0(B_1,B_0)\)
is computed by a Universal Coefficient Theorem, which splits
non-naturally.  So the arrow \(B_1\to B_0\)
gives a class in \(\Ext^1_{\Z[x,x^{-1}]}\bigl(\K_{*+1}(B_1),\K_*(B_0)\bigr)\).
It must determine the obstruction class of~\(A\)
because~\(A\)
is \(\KK^\Z\)-equivalent
to the cone of the map \(B_1\to B_0\).
The result is the formula for the obstruction class of~\(A\)
asserted above.  This computation of obstruction classes is carried
out in Section~\ref{sec:triangulated} in the same generality as
in~\cite{Bentmann-Meyer:More_general}, using a triangulated category
with a stable homological ideal with enough projectives.

The Universal Coefficient Theorem for \(\KKcat^\Z\)
is developed in Section~\ref{sec:UCT_Z}.  This is followed by an
analogous treatment for the category of \(\Cst\)\nb-algebras
over a unique path space~\(X\)
in Section~\ref{sec:UCT_X}.  In
Section~\ref{sec:compare_classification}, we compare the
classification theorems that our two invariants give for objects in
the bootstrap classes in \(\KKcat^\Z\)
and \(\KKcat^X\).
In Section~\ref{sec:filtrated_K}, we study \(\KKcat^X\)
in the special case where~\(X\)
is totally ordered.  Then filtrated \(\K\)\nb-theory
is a third complete invariant for the same objects
(see~\cite{Meyer-Nest:Filtrated_K}), and we compare this invariant
with the new ones.  This is a rather long homological computation.  In
Section~\ref{sec:extensions}, the result of
Section~\ref{sec:filtrated_K} is formulated more nicely if~\(X\)
has only two points, so that \(\Cst\)\nb-algebras
over~\(X\) are extensions of \(\Cst\)\nb-algebras.

\section{Universal Coefficient Theorems with KK-valued functors}
\label{sec:KK-valued_UCT}

This section develops Universal Coefficient Theorems for the
equivariant Kasparov categories \(\KKcat^\Z\)
and \(\KKcat^X\)
for a unique path space~\(X\),
which are based on forgetful functors to~\(\KKcat\).
We shall use the general framework developed
in~\cite{Meyer-Nest:Homology_in_KK} for doing relative homological
algebra in triangulated categories.  The starting point for this is a
stable homological ideal~\(\Ideal\)
in a triangulated category~\(\Tri\).
Being stable and homological means that there is a stable homological
functor \(F\colon \Tri \to \Abel\)
to some stable Abelian category~\(\Abel\) with
\[
\Ideal(A,B) = \ker F(A,B)
\defeq \setgiven{g\in\Tri(A,B)}{F(g)=0}.
\]
Here stability of \(\Abel\) and~\(F\)
means that~\(\Abel\)
is equipped with a suspension automorphism and that
\(F\circ \Sigma_{\Tri} = \Sigma_{\Abel} \circ F\).

A stable homological ideal~\(\Ideal\)
allows to carry over various notions from homological algebra
to~\(\Tri\).
In particular, there are \(\Ideal\)\nb-exact
chain complexes, \(\Ideal\)\nb-projective
objects and \(\Ideal\)\nb-projective
resolutions in~\(\Tri\),
which then allow to define \(\Ideal\)\nb-derived
functors.  We shall only be interested in cases where there are enough
\(\Ideal\)\nb-projective
objects.  The thick subcategory~\(\gen{\Proj_\Ideal}\)
generated by the \(\Ideal\)\nb-projective
objects is an analogue of the ``bootstrap class'' in Kasparov theory.
If \(A\inOb \gen{\Proj_\Ideal}\)
has an \(\Ideal\)\nb-projective
resolution of length~\(1\),
then the graded group \(\Tri_*(A,B)\)
for \(B\inOb\Tri\)
may be computed by a Universal Coefficient Theorem (UCT).  The Hom and
Ext groups in this UCT are those in a certain Abelian category,
namely, the target category of the universal \(\Ideal\)\nb-exact
stable homological functor~\(F^u\).
The functor~\(F^u\)
strongly classifies objects of~\(\gen{\Proj_\Ideal}\)
with an \(\Ideal\)\nb-projective
resolution of length~\(1\).
That is, any isomorphism \(F^u(A) \cong F^u(B)\)
between such objects is induced by an isomorphism \(A\cong B\)
in~\(\Tri\).
In many cases of interest, the universal functor~\(F^u\)
is quite explicit.  The following theorem allows to recognise it:

\begin{theorem}[\cite{Meyer-Nest:Homology_in_KK}*{Theorem~57}]
  \label{the:characterise_universal_homological}
  Let~\(\Tri\)
  be a triangulated category, let \(\Ideal \subseteq \Tri\)
  be a stable homological ideal, and let \(F\colon \Tri\to\Abel\)
  be an \(\Ideal\)\nb-exact
  stable homological functor into a stable Abelian category~\(\Abel\).
  Let~\(\Proj\Abel\)
  be the class of projective objects in~\(\Abel\).
  Suppose that idempotent morphisms in \(\Tri\)
  split.  The functor~\(F\)
  is the universal \(\Ideal\)\nb-exact
  stable homological functor to a stable Abelian category and there
  are enough \(\Ideal\)\nb-projective
  objects in~\(\Tri\) if and only if
  \begin{enumerate}
  \item \(\Abel\) has enough projective objects;
  \item the adjoint functor~\(F^\lad\)
    of~\(F\) is defined on~\(\Proj\Abel\);
  \item \(F\circ F^\lad(A) \cong A\) for all \(A\inOb\Proj\Abel\).
  \end{enumerate}
\end{theorem}

The ideals to be treated in this section are defined as the kernel on
morphisms of a triangulated functor \(F\colon \Tri\to\Tri[S]\)
to another triangulated category~\(\Tri[S]\).
Ideals of this form are already treated
in~\cite{Meyer-Nest:Homology_in_KK}, using a construction of Freyd to
embed~\(\Tri[S]\)
into an Abelian category.  Here we unify these two cases by allowing
the functor~\(F\)
to take values in an \emph{exact category} (see \cites{Buehler:Exact,
  Keller:Handbook}).  An exact category is an additive category with
a chosen class of ``admissible'' extensions.  A typical example is a
full, additive subcategory of an Abelian category that is closed
under extensions.  An Abelian category~\(\Abel\)
is exact where all extensions are admissible.  And a triangulated
category~\(\Tri[S]\)
is exact where only split extensions are admissible.  So exact
categories contain the two cases treated
in~\cite{Meyer-Nest:Homology_in_KK}.  In many examples of
classification results using a UCT, the range of the universal
\(\Ideal\)\nb-exact
functor~\(F^u_\mathrm{old}\)
is a certain exact subcategory of a module category, and only objects
in this exact subcategory have projective resolutions of length~\(1\).
This holds for the UCTs in \cites{Bentmann:Thesis,
  Bentmann-Koehler:UCT, Koehler:Thesis, Meyer-Nest:Filtrated_K} (see
Remark~\ref{rem:exact_NT} for one of these cases).  The phenomenon is
understood better in~\cite{Ambrogio-Stevenson-Stovicek:Gorenstein}.

Quillen's Embedding Theorem allows to embed any exact category into an
Abelian one in a fully faithful, fully exact way.  So all results in
the general theory carry over to the case where~\(F\)
takes values in an exact category.  In particular, we may modify the
definition of the universal \(\Ideal\)\nb-exact
stable homological functor by allowing exact categories as target.
Let \(F^u_\mathrm{old}\colon \Tri\to\Abel\)
be the universal \(\Ideal\)\nb-exact
functor into an Abelian category as in
Theorem~\ref{the:characterise_universal_homological}.  Let
\(\Abel[E]\subseteq\Abel\)
be the exact subcategory of~\(\Abel\)
generated by the range of~\(F^u_\mathrm{old}\).
Let \(F^u\colon \Tri\to\Abel[E]\)
be~\(F^u_\mathrm{old}\)
viewed as a functor to~\(\Abel[E]\).
This is the universal \(\Ideal\)\nb-exact
functor in the new sense.  Its universal property follows from the one
of \(F^u_\mathrm{old}\)
and Quillen's Embedding Theorem.  If~\(\Ideal\)
is defined by a forgetful functor \(F\colon \Tri\to\Tri[S]\)
to another triangulated category, then the universal homological
functor to an Abelian category typically uses Freyd's embedding
of~\(\Tri[S]\)
into an Abelian category.  In contrast, we shall see that in many
examples the universal \(\ker F\)\nb-exact
functor to an exact category gives an exact category that is very
closely related to~\(\Tri[S]\).

\begin{remark}
  \label{rem:universal_to_exact}
  Assume that all objects in~\(\Tri\)
  have an \(\Ideal\)\nb-projective
  resolution of finite length.  Then the image of~\(F^u_\mathrm{old}\)
  is contained in the subcategory of objects in~\(\Abel\)
  with a finite-length projective resolution.  This subcategory is
  exact.  And it is already generated as an exact category by the
  projective objects in~\(\Abel\):
  this is proved by induction on the length of a resolution.  Hence
  the image of~\(F^u_\mathrm{new}\)
  must be equal to the subcategory of objects with a finite-length
  projective resolution.  So \(F^u_\mathrm{new} = F^u_\mathrm{old}\)
  if and only if \emph{all} objects of~\(\Abel\)
  have a finite-length projective resolution.
\end{remark}

\subsection{Actions of the group of integers}
\label{sec:UCT_Z}

We now treat a concrete example, namely, the case
\(\Tri\defeq \KKcat^\Z\).
This is a triangulated category.  Its objects are pairs \((A,\alpha)\)
with a separable \(\Cst\)\nb-algebra~\(A\)
and \(\alpha\in\Aut(A)\);
the latter generates an action of~\(\Z\)
on~\(A\)
by automorphisms.  The arrows in \(\KKcat^\Z\)
are the Kasparov groups \(\KK^\Z_0(A,B)\),
and the composition is the Kasparov product; we have dropped the
automorphisms from our notation as usual to avoid clutter.  The
triangulated category structure on~\(\KKcat^\Z\)
is described in \cite{Meyer-Nest:BC}*{Appendix}.  The relative
homological algebra in~\(\KKcat^\Z\)
is already studied in~\cite{Meyer-Nest:Homology_in_KK}.  The main
result is the following variant of the Pimsner--Voiculescu sequence
for \(\KK^\Z_*(A,B)\).
Let \(A\)
and~\(B\)
be \(\Cst\)\nb-algebras
with automorphisms \(\alpha\)
and~\(\beta\), respectively.  Then there is an exact sequence
\begin{equation}
  \label{eq:PV-sequence}
  \begin{tikzcd}[column sep=huge]
    \KK_1(A,B) \arrow[r] &
    \KK_0^\Z(A,B) \arrow[r, "\mathrm{forget}"] & 
    \KK_0(A,B) \arrow[d, "\alpha^*\beta_*^{-1}-1"] \\
    \KK_1(A,B) \arrow[u, "\alpha^*\beta_*^{-1}-1"] &
    \KK_1^\Z(A,B) \arrow[l, "\mathrm{forget}"] &
    \KK_0(A,B) \arrow[l]
  \end{tikzcd}
\end{equation}
(see \cite{Meyer-Nest:Homology_in_KK}*{Section~5.1}); here
\((\alpha^*\beta_*^{-1}-1)(x) = \beta^{-1}\circ x\circ\alpha - x\)
for all \(x\in \KK_*(A,B)\).
We may rewrite this as a pair of short exact sequences
\begin{multline*}
  \coker \bigl(\alpha^*\beta_*^{-1}-1\colon
  \KK_{1+i}(A,B) \to \KK_{1+i}(A,B)\bigr)
  \into \KK_i^\Z(A,B)
  \\ \prto
  \ker \bigl(\alpha^*\beta_*^{-1}-1\colon \KK_i(A,B) \to \KK_i(A,B)\bigr)
\end{multline*}
for \(i=0,1\).
We shall explain that the long exact sequence~\eqref{eq:PV-sequence}
is an instance of a Universal Coefficient Theorem as in
\cite{Meyer-Nest:Homology_in_KK}*{Theorem~66}.  In particular, the
cokernel and kernel in it are the Ext and Hom groups in a certain
exact category.  This is already proven
in~\cite{Meyer-Nest:Homology_in_KK}, but the consequences for
classification are not explored there.  We shall also treat other
examples by the same method later.

The forgetful functor
\[
R^\Z\colon \KKcat^\Z\to\KKcat,\qquad
(A,\alpha)\mapsto A,
\]
is triangulated.  So its kernel on morphisms
\[
\Ideal^\Z(A,B) \defeq \setgiven{f\in\KK_0^\Z(A,B)}{R^\Z(f)=0
  \text{ in }\KK_0(A,B)}
\]
is a stable homological ideal in~\(\KKcat^\Z\).
We now describe the universal \(\Ideal^\Z\)\nb-exact
functor.  The main point here is to describe its target category.  Let
\(\KKcat[\Z]\)
be the additive category of functors \(\Z\to\KKcat\)
with natural transformations as arrows.  Equivalently, an object
of~\(\KKcat[\Z]\)
is an object \(A\inOb\KKcat\)
with a group homomorphism \(a\colon \Z\to\KK_0(A,A)^\times\),
\(n\mapsto a_n\),
where \(\KK_0(A,A)^\times\)
denotes the multiplicative group of invertible elements in the ring
\(\KK_0(A,A)\).
And an arrow \((A,a) \to (B,b)\)
in \(\KKcat[\Z]\)
is an arrow \(f\in \KK_0(A,B)\)
that satisfies \(b_n\circ f = f\circ a_n\)
for all \(n\in\Z\).
The homomorphism~\(a\)
is determined by its value at \(1\in\Z\),
which may be any element \(a_1\in \KK_0(A,A)^\times\).
So we also denote objects of~\(\KKcat[\Z]\)
as \((A,a_1)\).
An element \(f\in \KK_0(A,B)\)
is an arrow \((A,a) \to (B,b)\)
in \(\KKcat[\Z]\)
if and only if \(b_1\circ f = f\circ a_1\)
or, equivalently, \(b_1^{-1}\circ f \circ a_1 - f=0\).  Thus
\[
\ker \bigl(\alpha^*\beta_*^{-1}-1\colon \KK_0(A,B) \to \KK_0(A,B)\bigr)
\cong \Hom_{\KKcat[\Z]}\bigl((A,[\alpha]),(B,[\beta])\bigr).
\]
Here we have implicitly used the functor
\[
F^\Z\colon \KKcat^\Z\to \KKcat[\Z],\qquad
(A,\alpha) \mapsto (A,[\alpha]),
\]
where \([\alpha]\in \KK_0(A,A)^\times\)
is the \(\KKcat\)\nb-class
of the automorphism~\(\alpha\).
It has the same kernel on morphisms~\(\Ideal^\Z\)
as the forgetful functor~\(R^\Z\).

We call a kernel--cokernel pair \(K \into E \prto Q\)
in \(\KKcat[\Z]\)
\emph{admissible} if it splits in~\(\KKcat\),
that is, \(E\cong K\oplus Q\)
as objects of~\(\KKcat\).
This turns \(\KKcat[\Z]\)
into an exact category (see \cite{Buehler:Exact}*{Exercise~5.3}).  We
use this exact structure to define projective objects and projective
resolutions in \(\KKcat[\Z]\).

Given \(A\inOb\KKcat\)
and a free Abelian group \(G= \Z[I]\)
on a countable set~\(I\), we define \(A\otimes G\inOb\KKcat\) by
\begin{equation}
  \label{def:tensor_with_Cstar}
  A \otimes G \defeq \bigoplus_{i\in I} A.
\end{equation}
In particular, \(A\otimes \Z \defeq A\).
This construction is an additive functor in~\(G\).
Namely, let \(G\)
and~\(H\)
be free Abelian groups and let \(f\colon G\to H\)
be a group homomorphism.  The Universal Coefficient Theorem implies
\(\KK_0(\C\otimes G,\C\otimes H) \cong \Hom(G,H)\).
So we get a functorial \(f_\C\in \KK_0(\C\otimes G,\C\otimes H)\).
Identifying \(A\otimes G = A \otimes (\C\otimes G)\)
(with the minimal tensor product of \(\Cst\)\nb-algebras),
we get a functorial
\[
f_A\defeq \id_A \otimes f_\C\in\KK_0(A\otimes G,A\otimes H).
\]

Let \(A\inOb\KKcat\).
Then \(A\otimes \Z[x,x^{-1}]\) with the invertible element
\[
x_A = \id_A \otimes x
\in \KK_0\bigl(A\otimes \Z[x,x^{-1}],A\otimes \Z[x,x^{-1}]\bigr)
\]
induced by multiplication with the invertible element \(x\in\Z[x,x^{-1}]\)
is an object of \(\KKcat[\Z]\).
It behaves like a free module over~\(A\) because
\[
\Hom_{\KKcat[\Z]}\bigl((A\otimes \Z[x,x^{-1}],x_A), (B,b) \bigr)
\cong \KK_0(A,B).
\]
In particular, \((A\otimes \Z[x,x^{-1}],x_A)\)
is projective in the exact category \(\KKcat[\Z]\).
The invertible element~\(x_A\)
in the \(\KK\)\nb-endomorphism
ring of \((A\otimes \Z[x,x^{-1}],x_A)\)
lifts to the shift automorphism
\[
\tau\in\Aut(\Cont_0(\Z,A)),\qquad
(\tau f)(n) \defeq f(n-1).
\]
That is,
\(F^\Z(\Cont_0(\Z,A),\tau) \cong (A\otimes \Z[x,x^{-1}],x_A)\).  And
\begin{equation}
  \label{eq:FZ_adjoint}
  \KK^\Z_0(\Cont_0(\Z,A), B)
  \cong \KK_0(A,B)
  \cong \Hom_{\KKcat[\Z]}\bigl((A\otimes \Z[x,x^{-1}],x_A),B\bigr).
\end{equation}
The first isomorphism here is \cite{Meyer-Nest:BC}*{Equation~(20)}.
It applies the forgetful functor
\[
\KK^\Z_0\bigl((\Cont_0(\Z,A),\tau), (B,\beta) \bigr) \to
\KK_0(\Cont_0(\Z,A),B)
\]
and then composes with the inclusion \Star{}homomorphism
\(A\to\Cont_0(\Z,A)\)
that maps~\(A\)
identically onto the summand at \(0\in\Z\).
Equation~\eqref{eq:FZ_adjoint} says that the partial adjoint of the
functor~\(F^\Z\)
is defined on \((A\otimes \Z[x,x^{-1}],x_A)\)
and maps it to \((\Cont_0(\Z,A),\tau)\).

\begin{lemma}
  \label{lem:PV-coker_is_Ext}
  Any object of~\(\KKcat[\Z]\) has a free resolution of length~\(1\).
\end{lemma}

\begin{proof}
  The trivial representation of the group~\(\Z\)
  on~\(\Z\)
  corresponds to~\(\Z\)
  made a module over~\(\Z[x,x^{-1}]\)
  by letting \(x\in \Z[x,x^{-1}]\)
  act by~\(1\).
  This module has the following free resolution of length~\(1\):
  \begin{equation}
    \label{eq:ZZ_resolution}
    0\to \Z[x,x^{-1}] \xrightarrow{\mathrm{mult}(x-1)} \Z[x,x^{-1}]
    \xrightarrow{\ev_1} \Z \to 0,
  \end{equation}
  where \(\ev_1\colon \Z[x,x^{-1}] \to \Z\)
  is the homomorpism of evaluation at~\(1\),
  so \(\ev_1(x^n)=1\)
  for all \(n\in\Z\).
  The extension in~\eqref{eq:ZZ_resolution} splits as an extension
  of Abelian groups because~\(\Z\)
  is free as an Abelian group.

  The resolution~\eqref{eq:ZZ_resolution} induces a chain complex
  \begin{equation}
    \label{eq:AZ_resolution}
    0 \to (A \otimes \Z[x,x^{-1}],[\alpha]\otimes x)
    \to (A \otimes \Z[x,x^{-1}],[\alpha]\otimes x)
    \to (A,[\alpha])
    \to 0
  \end{equation}
  in the additive category~\(\KKcat[\Z]\).
  That is, the maps are arrows in~\(\KKcat[\Z]\).
  The extension~\eqref{eq:ZZ_resolution} splits by a group
  homomorphism, and the tensor product construction is additive.
  Hence~\eqref{eq:AZ_resolution} is exact, that is, it splits
  in~\(\KKcat\).  The arrow
  \[
  A \otimes \Z[x,x^{-1}]
  = \bigoplus_{n\in\Z} A
  \xrightarrow{\bigoplus_{n\in\Z} [\alpha^n]}
  \bigoplus_{n\in\Z} A
  = A \otimes \Z[x,x^{-1}]
  \]
  in~\(\KKcat\) is an isomorphism
  \[
  (A \otimes \Z[x,x^{-1}],\id_A\otimes x) \cong
  (A \otimes \Z[x,x^{-1}],[\alpha]\otimes x).
  \]
  Thus~\eqref{eq:AZ_resolution} is isomorphic to a free resolution
  \begin{equation}
    \label{eq:AZ_resolution_free}
    0 \to (A \otimes \Z[x,x^{-1}],1\otimes x)
    \to (A \otimes \Z[x,x^{-1}],1\otimes x)
    \to (A,[\alpha])
    \to 0,
  \end{equation}
  in~\(\KKcat[\Z]\),
  where the boundary map on \((A \otimes \Z[x,x^{-1}],1\otimes x)\)
  has changed to
  \([\alpha]^{-1} \otimes \mathrm{mult}(x) - \id_{A \otimes
    \Z[x,x^{-1}]}\).
\end{proof}

Lemma~\ref{lem:PV-coker_is_Ext} implies that the exact category
\(\KKcat[\Z]\)
has enough projective objects and that an object of~\(\KKcat[\Z]\)
is projective if and only if it is a direct summand of a free object.
Since the partial adjoint of the functor~\(F^\Z\)
is defined on all free objects, it is defined also on all projective
objects of \(\KKcat[\Z]\),
and it is inverse to~\(F^\Z\)
on this subcategory.  Idempotents in the category \(\KKcat^\Z\)
split because it is triangulated and has countable direct sums.

\begin{proposition}
  \label{pro:FZ_universal_exact}
  The functor \(F^\Z\colon \KKcat^\Z \to \KKcat[\Z]\)
  is the universal \(\Ideal^\Z\)\nb-exact
  stable homological functor from~\(\KKcat^\Z\) to an exact category.
\end{proposition}

\begin{proof}
  To prove this, we first embed \(\KKcat[\Z]\)
  into an Abelian category by Quillen's Embedding Theorem.  One way to
  do this is to embed~\(\KKcat\)
  into an Abelian category~\(\Abel\)
  by Freyd's Embedding Theorem and then form \(\Abel{}[\Z]\).
  The Abelian category \(\Abel{}[\Z]\)
  has enough projective objects, and all its projective objects
  already belong to \(\KKcat[\Z]\).
  So Theorem~\ref{the:characterise_universal_homological} applies and
  shows that the functor to~\(\Abel{}[\Z]\)
  is the universal \(\Ideal\)\nb-exact
  functor to an Abelian category.  The subcategory
  \(\KKcat[\Z] \subseteq \Abel{}[\Z]\)
  is exact and contains the range of the functor.  In fact, we shall
  show soon that the functor \(\KKcat^\Z \to \KKcat[\Z]\)
  is surjective on objects (see Theorem~\ref{the:Z_classify}).  Taking
  this for granted, Remark~\ref{rem:universal_to_exact} shows that the
  functor \(\KKcat^\Z \to \KKcat[\Z]\)
  is the universal \(\Ideal\)\nb-exact functor to an exact category.
\end{proof}

Equip \(\Cont_0(\Z,A)\)
with the \(\Z\)\nb-action
generated by the shift automorphism~\(\tau\).
We use~\eqref{eq:ZZ_resolution} to lift the
resolution~\eqref{eq:AZ_resolution_free} in \(\KKcat[\Z]\)
to the following \(\Ideal^\Z\)\nb-projective
resolution of length~\(1\) in~\(\KKcat^\Z\):
\begin{equation}
  \label{eq:lift_AZ_resolution}
  0 \to \Cont_0(\Z) \otimes A
  \xrightarrow{\varphi} \Cont_0(\Z) \otimes A
  \xrightarrow{p} (A,[\alpha])
  \to 0;
\end{equation}
here \(\varphi = [\tau]\otimes [\alpha]^{-1} - 1\),
and~\(p\)
is the counit of the adjunction~\eqref{eq:FZ_adjoint}; that is, the
first isomorphism in~\eqref{eq:FZ_adjoint} maps~\(p\)
to the identity element in \(\KK_0(A,A)\).
In other words, when we forget the \(\Z\)\nb-actions,
then~\(p\)
restricts to the identity map at the \(0\)th
summand in \(\Cont_0(\Z) \otimes A = \bigoplus_{n\in\Z} A\).
When we use the resolution~\eqref{eq:lift_AZ_resolution} to compute
derived functors, we get
\[
\Ext^1_{\KKcat[\Z],{\Ideal^\Z}}\bigl((A,[\alpha]),(B,[\beta])\bigr)
\cong
\coker \bigl((\alpha^*)^{-1}\beta_*-1\colon \KK_0(A,B) \to \KK_0(A,B)\bigr).
\]
The map \((\alpha^*)^{-1}\beta_*-1\),
has the same cokernel as \(\alpha^*\beta_*^{-1}-1\).

A rather deep result says that the \(\Ideal^\Z\)\nb-projective
objects generate \(\KKcat^\Z\).
Equivalently, if \(R^\Z(A) \cong 0\)
in \(\KKcat\),
then already \(A\cong 0\)
in \(\KKcat^\Z\).
This is related to the proof of the Baum--Connes conjecture for the
group~\(\Z\).
It follows immediately from the Pimsner--Voiculescu
sequence~\eqref{eq:PV-sequence}.

\begin{theorem}
  \label{the:Z_classify}
  Let \((A,\alpha)\)
  and \((B,\beta)\)
  be objects of \(\KKcat^\Z\).
  Any isomorphism \(F^\Z(A) \cong F^\Z(B)\)
  in \(\KKcat[\Z]\)
  lifts to an isomorphism \(A\cong B\)
  in \(\KKcat^\Z\).
  And any object of \(\KKcat[\Z]\)
  is isomorphic to \(F^\Z(A)\)
  for some \(A\inOb\KKcat^\Z\), which is unique up to isomorphism.

  If \(t\in \KK_0(A,B)^\times\)
  is a \(\KK\)\nb-equivalence
  with \(t\circ [\alpha] = [\beta]\circ t\),
  then there is an isomorphism in \(\KK^\Z_0(A,B)\)
  that is mapped to~\(t\)
  by the forgetful functor \(\KKcat^\Z\to\KKcat\).
\end{theorem}

\begin{proof}
  The Universal Coefficient Theorem for \(\KK^\Z_0(A,B)\)
  and the ideal~\(\Ideal^\Z\)
  allows to lift any map \(t\colon (A,[\alpha])\to (B,[\beta])\)
  in~\(\KKcat[\Z]\)
  to an element
  \(\hat{t}\in \KK^\Z_0\bigl((A,\alpha), (B,\beta)\bigr)\).
  The naturality of the Universal Coefficient Theorem implies that the
  composition vanishes for two elements of the \(\Ext^1\)-part
  of \(\KK^\Z_0(A,B)\).
  Hence~\(\hat{t}\)
  is invertible if~\(t\)
  is invertible.  Any object of \(\KKcat[\Z]\)
  has a projective resolution of length~\(1\).
  This allows to lift it to an object of~\(\KKcat^\Z\)
  (see \cite{Bentmann-Meyer:More_general}*{Proposition~2.3}).  This
  proves both assertions in the first paragraph.  The second paragraph
  only describes isomorphisms in \(\KKcat[\Z]\) more concretely.
\end{proof}

We make the isomorphism criterion in Theorem~\ref{the:Z_classify} more
explicit under the assumption that the \(\Cst\)\nb-algebras
\(A\)
and~\(B\)
belong to the bootstrap class.  Let~\(A^\pm\)
be \(\Cst\)\nb-algebras in the bootstrap class with
\[
\K_0(A^+) = \K_0(A),\quad
\K_1(A^+) = 0,\qquad
\K_0(A^-) = 0,\quad
\K_1(A^-) = \K_1(A).
\]
Then \(\K_*(A^+ \oplus A^-) \cong \K_*(A)\).
So \(A\cong A^+ \oplus A^-\)
by the Universal Coefficient Theorem for \(\KK\).
We use such a \(\KK\)\nb-equivalence
to map~\(\alpha\)
to an element of \(\KK_0(A^+\oplus A^-, A^+\oplus A^-)\).
We rewrite this as a \(2\times2\)-matrix
\[
\begin{pmatrix}
  \alpha^{++}& \alpha^{+-}\\
  \alpha^{-+}& \alpha^{--}
\end{pmatrix},\qquad
\begin{aligned}
  \alpha^{++} &\in\KK_0(A^+,A^+) \cong \Hom\bigl(\K_0(A),\K_0(A)\bigr),\\
  \alpha^{+-} &\in\KK_0(A^-,A^+) \cong \Ext\bigl(\K_1(A),\K_0(A)\bigr),\\
  \alpha^{-+} &\in\KK_0(A^+,A^-) \cong \Ext\bigl(\K_0(A),\K_1(A)\bigr),\\
  \alpha^{--} &\in\KK_0(A^-,A^-) \cong \Hom\bigl(\K_1(A),\K_1(A)\bigr).
\end{aligned}
\]
Here we have used the Universal Coefficient Theorem for the \(\Cst\)\nb-algebras
\(A^\pm\).
A similar decomposition \(B\cong B^+ \oplus B^-\)
allows us to map~\(\beta\)
to an element of \(\KK_0(B^+\oplus B^-, B^+\oplus B^-)\),
which we then rewrite as a \(2\times2\)-matrix
\[
\begin{pmatrix}
  \beta^{++}& \beta^{+-}\\
  \beta^{-+}& \beta^{--}
\end{pmatrix},\qquad
\begin{aligned}
  \beta^{++} &\in\KK_0(B^+,B^+) \cong \Hom\bigl(\K_0(B),\K_0(B)\bigr),\\
  \beta^{+-} &\in\KK_0(B^-,B^+) \cong \Ext\bigl(\K_1(B),\K_0(B)\bigr),\\
  \beta^{-+} &\in\KK_0(B^+,B^-) \cong \Ext\bigl(\K_0(B),\K_1(B)\bigr),\\
  \beta^{--} &\in\KK_0(B^-,B^-) \cong \Hom\bigl(\K_1(B),\K_1(B)\bigr).
\end{aligned}
\]
And we may also transfer an element \(t\in \KK_0(A,B)\)
to such a \(2\times2\)-matrix
\[
\begin{pmatrix}
  t^{++}& t^{+-}\\
  t^{-+}& t^{--}
\end{pmatrix},\qquad
\begin{aligned}
  t^{++} &\in\KK_0(A^+,B^+) \cong \Hom\bigl(\K_0(A),\K_0(B)\bigr),\\
  t^{+-} &\in\KK_0(A^-,B^+) \cong \Ext\bigl(\K_1(A),\K_0(B)\bigr),\\
  t^{-+} &\in\KK_0(A^+,B^-) \cong \Ext\bigl(\K_0(A),\K_1(B)\bigr),\\
  t^{--} &\in\KK_0(A^-,B^-) \cong \Hom\bigl(\K_1(A),\K_1(B)\bigr).
\end{aligned}
\]
The naturality of the exact sequence in the Universal Coefficient Theorem implies that the
Kasparov product of an element of \(\Ext\bigl(\K_{1+*}(A),\K_*(B)\bigr)\)
with an element of \(\KK_0(B,C)\)
depends only on the image of the latter in \(\Hom\bigl(\K_*(B),\K_*(C)\bigr)\).
Thus the product of two \(\Ext\)-terms
always vanishes.  Hence the condition \(t [\alpha]= [\beta] t\)
in Theorem~\ref{the:Z_classify} is equivalent to four equations
\begin{alignat*}{2}
  t^{++} \alpha^{++} &= \beta^{++} t^{++},&\qquad
  t^{+-} \alpha^{--} +   t^{++} \alpha^{+-}
  &= \beta^{+-} t^{--} +  \beta^{++} t^{+-}, \\
  t^{--} \alpha^{--} &=  \beta^{--} t^{--},&\qquad
  t^{-+} \alpha^{++} +   t^{--} \alpha^{-+}
  &=  \beta^{-+} t^{++} + \beta^{--} t^{-+}.
\end{alignat*}
The equations on the left assert that the two diagonal entries of
\(t [\alpha]\)
and~\([\beta] t\)
are equal; those on the right assert equality of the two off-diagonal
entries.  Theorem~\ref{the:Z_classify} says that \((A,\alpha)\)
and~\((B,\beta)\)
are \(\KK^\Z\)\nb-equivalent
if and only if there is an invertible element~\(t\)
for which these equations hold.  The equations on the left mean that
the isomorphism \(t_* = (t^{++}, t^{--}) \in \Hom\bigl(\K_*(A),\K_*(B)\bigr)\)
induced by~\(t\)
intertwines the automorphisms \(\alpha_*\in \Aut(\K_*(A))\)
and \(\beta_*\in \Aut(\K_*(B))\).
Equivalently, \(t_*\)
is an isomorphism of \(\Z/2\)\nb-graded
\(\Z[x,x^{-1}]\)-modules.
The equations on the right mean that there is
\(t^1 \defeq (t^{+-},t^{-+}) \in \Ext\bigl(\K_{1+*}(A),\K_*(B)\bigr)\) with
\[
t^1 \alpha_* - \beta_* t^1 = \beta^1 t_* - t_* \alpha^1,
\]
where we abbreviate \(\alpha_* = (\alpha^{++},\alpha^{--})\)
and \(\alpha^1 = (\alpha^{+-},\alpha^{-+})\),
and similarly for~\(\beta\).
The choice of~\(t^1\)
has no effect on the invertibility of~\(t\).
So the criterion in Theorem~\ref{the:Z_classify} is whether the image
of \(\beta^1 t_* - t_* \alpha^1\) vanishes in the cokernel of
\begin{equation}
  \label{eq:cokernel_Ext_Z}
  \Ext\bigl(\K_{1+*}(A),\K_*(B)\bigr) \to \Ext\bigl(\K_{1+*}(A),\K_*(B)\bigr),\qquad
  t^1 \mapsto t^1 \alpha_* - \beta_* t^1.
\end{equation}
We shall later identify this cokernel with the group
\(\Ext^2_{\Z[x,x^{-1}]}\bigl(\K_{1+*}(A),\K_*(B)\bigr)\)
and show that the image of \(\beta^1 t_* - t_* \alpha^1\)
in this cokernel is the relative obstruction class for \((A,\alpha)\)
and \((B,\beta)\)
and a \(\Z[x,x^{-1}]\)-module
isomorphisms \(t_*\colon \K_*(A) \cong \K_*(B)\).
This gives the rule to translate between the classifying invariants in
Theorem~\ref{the:Z_classify} and
in~\cite{Bentmann-Meyer:More_general}.

\begin{remark}
  If \(A=B\)
  is a unital Kirchberg algebra (separable, nuclear, unital, purely infinite
  and simple), then a much finer classification theorem for
  automorphisms is proved by Nakamura~\cite{Nakamura:Aperiodic}.
\end{remark}

\subsection{\texorpdfstring{$\Cst$}{C*}-Algebras over unique path spaces}
\label{sec:UCT_X}

Now we prove a Universal Coefficient Theorem for \(\KKcat^X\)
for a unique path space~\(X\).
Let~\(X\)
be a countable set and let~\(\to\)
be a relation on~\(X\),
which says for which points \(x,y\in X\)
there is an edge \(x\to y\).
Equip~\(X\)
with the partial order generated by~\(\leftarrow\),
that is, \(x \preceq y\)
if and only if there is a chain of edges
\(x=x_0 \leftarrow x_1 \leftarrow \dotsb \leftarrow x_\ell = y\)
with some \(\ell\ge0\)
and \(x_1,\dotsc,x_{\ell-1}\in X\).
Equip~\(X\)
with the Alexandrov topology generated by this partial order.  We
assume~\((X,\to)\)
to be a \emph{unique path space}, that is, there is at most one chain
of edges
\(x=x_0 \leftarrow x_1 \leftarrow \dotsb \leftarrow x_\ell = y\)
between any two points \(x,y\in X\).

For \(x\in X\),
the subset \(U_x\defeq \setgiven{y\in X}{x\preceq y}\)
is the minimal open subset containing~\(x\).
A \emph{\(\Cst\)\nb-algebra
  over~\(X\)}
is equivalent to a \(\Cst\)\nb-algebra~\(A\)
with fixed ideals \(A(U_x)\idealin A\)
for all \(x\in X\),
such that \(A(U_x) \subseteq A(U_y)\)
for all \(x,y\in X\)
with \(x\to y\)
or, equivalently, with \(U_x \subseteq U_y\).
The equivariant Kasparov category~\(\KKcat^X\)
for \(\Cst\)\nb-algebras
over~\(X\)
has separable \(\Cst\)\nb-algebras
over~\(X\)
as objects and the KK-groups \(\KK_0^X(A,B)\)
as arrows (see also~\cite{Meyer-Nest:Bootstrap}, where this category
is denoted \(\mathfrak{KK}(X)\)).  The forgetful functor
\[
R^X\colon \KKcat^X \to \prod_{x\in X} \KKcat,\qquad
A\mapsto (A(U_x))_{x\in X},
\]
is a triangulated functor between triangulated categories.  Its kernel on
morphisms
\[
\Ideal^X(A,B) \defeq \setgiven{f\in \KK^X_0(A,B)}
{f(U_x)=0 \text{ in } \KK_0(A(U_x),B(U_x)) \text{ for all } x\in X}
\]
is a stable homological ideal.  We now describe the universal
\(\Ideal^X\)\nb-exact
stable homological functor as in Section~\ref{sec:UCT_Z}.

Let \(\KKcat[X]\)
be the category of functors \((X,\succeq) \to \KKcat\)
with natural transformations as arrows.  Since~\(X\) is a unique
path space, the category associated to the partially ordered set
\((X,\succeq)\)
is the path category of the directed graph \((X,\to)\).
By the universal property of the path category, an object of
\(\KKcat[X]\)
is given by \(A_x\inOb\KKcat\)
for \(x\in X\)
and \(\alpha_{y,x}\in \KK_0(A_x,A_y)\)
for \(x,y\in X\)
with \(x\to y\),
without any relations on the~\(\alpha_{y,x}\).
This uniquely determines \(\KK\)-classes
\(\alpha_{y,x}\in \KK_0(A_x,A_y)\)
for \(x,y\in X\)
with \(x\succeq y\)
such that \(\alpha_{x,x} = \id_{A_x}\)
and \(\alpha_{z,y}\circ\alpha_{y,x}=\alpha_{z,x}\)
for all \(x,y,z\in X\)
with \(x\succeq y\succeq z\).
An arrow \((A_x,\alpha_{y,x})\to (B_x,\beta_{y,x})\)
is a family of arrows \(f_x\in \KK_0(A_x,B_x)\)
for \(x\in X\)
with \(f_y \alpha_{y,x} = \beta_{y,x} f_x\)
in \(\KK_0(A_x,B_y)\)
for all \(x,y\in X\)
with \(x\to y\);
then \(f_y \alpha_{y,x} = \beta_{y,x} f_x\)
holds for all \(x,y\in X\) with \(x\succeq y\).  Define
\[
F^X\colon \KKcat^X\to \KKcat[X]
\]
by mapping a \(\Cst\)\nb-algebra~\(A\)
over~\(X\)
to the object of \(\KKcat[X]\)
where \(A_x \defeq A(U_x)\)
and where \(\alpha_{y,x}\in\KK_0(A_x,A_y)\)
for \(x,y\in X\)
with \(x\to y\)
is the \(\KK\)-class
of the inclusion map \(A(U_x) \hookrightarrow A(U_y)\).
Then \(\alpha_{y,x}\in\KK_0(A_x,A_y)\)
for \(x,y\in X\)
with \(x\succeq y\)
is the \(\KK\)-class
of the inclusion map as well.  A kernel--cokernel pair \(K\to E\to Q\)
in \(\KKcat[X]\)
is called \emph{admissible} if it splits pointwise, that is,
\(K_x \to E_x \to Q_x\)
is a split extension in \(\KKcat\)
for all \(x\in X\);
so \(E_x \cong K_x \oplus Q_x\)
in \(\KKcat\)
for all \(x\in X\),
but the sections \(Q_x \to E_x\)
are not compatible with the structure maps \(E_x \to E_y\)
and \(Q_x \to Q_y\)
for \(x\to y\).
This turns~\(\KKcat[X]\)
into an exact category (see \cite{Buehler:Exact}*{Exercise~5.3}).

Let \(z\in X\)
and \(A\inOb\KKcat\).
As in~\cite{Meyer-Nest:Bootstrap}, let \(i_z(A)\inOb\KKcat^X\)
be the \(\Cst\)\nb-algebra~\(A\)
with
\[
i_z(A)(U_x) \defeq
\begin{cases}
  A&\text{if }z\in U_x\text{, that is, }x\preceq z,\\
  0&\text{otherwise.}
\end{cases}
\]
Then
\begin{equation}
  \label{eq:KK_out_ix}
  \KK^X_0(i_z(A),B) \cong \KK_0\bigl(A,B(U_z)\bigr)
\end{equation}
for all \(B\inOb\KKcat^X\)
by \cite{Meyer-Nest:Bootstrap}*{Proposition 3.13}.  This isomorphism
applies the restriction map
\(\KK^X_0(i_z(A),B) \to \KK_0(i_z(A)(U_z),B(U_z))\)
and then identifies \(i_z(A)(U_z) = A\).  The object
\[
j_z(A) \defeq F^X(i_z(A)) \inOb\KKcat[X]
\]
has \(j_z(A)_x = A\)
for \(x\preceq z\)
and \(j_z(A)_x = 0\)
otherwise,
and the map \(j_z(A)_x \to j_z(A)_y\)
for \(x\succeq y\)
is the identity map in \(\KK_0(A,A)\)
if \(z\succeq x\)
and the zero map in \(\KK_0(0,j_z(A)_y)\) otherwise.  We compute
\begin{equation}
  \label{eq:KKX_out_Fix}
  \Hom_{\KKcat[X]}\bigl(j_z(A),(B_x,\beta_{y,x})\bigr) \cong \KK_0(A,B_z).  
\end{equation}
Equations \eqref{eq:KK_out_ix} and~\eqref{eq:KKX_out_Fix} imply
\begin{equation}
  \label{eq:KKX_out_Fix2}
  \KK^X_0(i_z(A),B)
  \cong \KK_0(A,B(U_z))
  \cong \Hom_{\KKcat[X]}\bigl(j_z(A),F^X(B)\bigr)
\end{equation}
for all \(B\inOb\KKcat^X\).

\begin{theorem}
  \label{the:projectives_ups}
  Let \((X,\to)\)
  be a countable unique path space.  The objects of the exact
  category~\(\KKcat[X]\)
  of the form \(\bigoplus_{z\in X} j_z(A_z)\)
  for \(A_z\inOb\KKcat\)
  for \(z\in X\)
  are projective.  For any object \(A = (A_x,\alpha_{x,y})\)
  of \(\KKcat[X]\),
  there is an admissible extension
  \begin{equation}
    \label{eq:projective_resolution_ups}
    \bigoplus_{x\to y} j_y(A_x)
    \overset{\iota}\into \bigoplus_{x\in X} j_x(A_x)
    \overset{\pi}\prto A,
  \end{equation}
  which is a projective resolution of length~\(1\).
  The exact category~\(\KKcat[X]\)
  has enough projective objects, and the partial adjoint of~\(F^X\)
  is defined on all projective objects of~\(\KKcat[X]\)
  and is a section for~\(F^X\)
  there.  The functor \(F^X\colon \KKcat^X \to \KKcat[X]\)
  is the universal \(\Ideal^X\)\nb-exact
  stable homological functor to an exact category.
\end{theorem}

\begin{proof}
  The objects~\(j_z(A_z)\)
  are projective by~\eqref{eq:KKX_out_Fix}.  This is inherited by the
  direct sum \(\bigoplus_{z\in X} j_z(A_z)\).
  The identity map on~\(A_x\)
  has an adjunct \(a_x\colon j_x(A_x) \to A\)
  by~\eqref{eq:KKX_out_Fix}.  These maps induce the map
  \(\pi= (a_x)_{x\in X}\colon \bigoplus_{x\in X} j_x(A_x) \to A\).
  For each edge \(x\to y\)
  in the directed graph \((X,\to)\),
  the map \((\id_{A_x},-\alpha_{y,x})\colon A_x \to A_x \oplus A_y\)
  is adjunct to a map
  \[
  a_{x\to y}\colon j_y(A_x) \to j_x(A_x) \oplus j_y(A_y)
  \subseteq \bigoplus_{x\in X} j_x(A_x)
  \]
  by~\eqref{eq:KKX_out_Fix}.  It satisfies
  \(\pi\circ a_{x\to y} = 0\).
  The maps~\(a_{x\to y}\) combine to a map
  \[
  \iota\colon \bigoplus_{x\to y} j_y(A_x) \to \bigoplus_{x} j_x(A_x)
  \]
  with \(\pi\circ\iota=0\).

  Now we prove that the maps \(\pi\)
  and~\(\iota\)
  mapped to \(\prod_{x\in X} \KKcat\)
  by the forgetful functor form a split exact sequence.  We consider
  the entries at a fixed \(z\in X\).
  The entry of~\(A\)
  at \(z\in X\)
  is simply~\(A_z\).
  The entry of \(\bigoplus_{x\in X} j_x(A_x)\)
  at \(z\in X\)
  is the direct sum of~\(A_x\)
  over all \(x\in X\)
  with \(x\succeq z\).
  The entry of \(\bigoplus_{x\to y} j_y(A_x)\)
  at \(z\in X\)
  is the direct sum of~\(A_x\)
  for all edges \(x\to y\)
  in~\(X\) with \(y\succeq z\).  The entry of~\(\pi\) at~\(z\) is
  \[
  (\alpha_{z,x})_{x\succeq z} \colon
  \bigoplus_{x\succeq z} A_x \to A_z.
  \]
  This is split surjective with the canonical section that
  maps~\(A_z\)
  identically onto the summand~\(A_z\)
  for \(x=z\).
  The entry of~\(\iota\)
  at~\(z\)
  maps the summand~\(A_x\)
  for \(x\to y\succeq z\)
  to \(A_x \oplus A_y\subseteq \bigoplus_{t\succeq z} A_t\)
  using \((\id_{A_x},-\alpha_{y,x})\).  We are going to define a map
  \[
  s_1\colon \bigoplus_{t\succeq z} A_t \to
  \bigoplus_{x\to y \succeq z} A_x
  \]
  which together with~\(s_0\)
  forms a contracting homotopy for the short chain complex formed by
  \(\iota|_z\)
  and~\(\pi|_z\).
  By assumption, if \(t\succeq z\)
  then there is a unique chain
  \(t=x_0 \to x_1 \to \dotsb \to x_\ell = z\).
  We let \((s_1)_z|_{A_t}\)
  for \(t\succeq z\)
  map the summand \(A_t\)
  in \(\bigoplus_{t\succeq z} A_t\)
  to the direct sum of~\(A_{x_i}\)
  for the edges \(x_i \to x_{i+1}\)
  for \(i=0,\dotsc,\ell-1\),
  where we use~\(\alpha_{x_i,t}\)
  to map \(A_t\)
  to~\(A_{x_i}\).
  If \(x\to y \succeq z\),
  then \(y=x_1\)
  in the above chain.  Therefore, \(s_1\circ\iota|_{A_x}\)
  is a map to \(\bigoplus_{j=0}^\ell A_{x_j}\),
  where the entry at~\(A_{x_j}\)
  is \(\alpha_{x_j,x} - \alpha_{x_j,y}\circ \alpha_{y,x} = 0\)
  for \(j=1,\dotsc,\ell\),
  and the identity map for \(j=0\).
  So \(s_1\circ \iota\)
  is the identity map.  Finally, we claim that
  \(s_0\circ \pi + \iota\circ s_1\)
  is the identity map on \(\bigoplus_{t\succeq z} A_t\).
  This is checked on each summand~\(A_t\)
  separately.  Let \(t=x_0 \to x_1 \to \dotsb \to x_\ell = z\)
  be the unique chain as above.  Then \(\iota\circ s_1\)
  is a telescoping sum of \(\pm\alpha_{x_j,t}\)
  for \(j=0,\dotsc,\ell\),
  where \(\alpha_{t,t}\)
  occurs only with sign~\(+\)
  and~\(\alpha_{z,t}\)
  only with sign~\(-\).
  And \(s_0\circ\pi\)
  is the map~\(\alpha_{z,t}\).
  Thus \(s_0\circ \pi + \iota\circ s_1\)
  is the identity map on~\(A_t\).
  This finishes the proof that~\eqref{eq:projective_resolution_ups} is
  \(\Ideal^X\)\nb-exact.
  Its entries are \(\Ideal^X\)\nb-projective.
  So it is an \(\Ideal^X\)\nb-projective resolution.

  The projective resolution~\eqref{eq:projective_resolution_ups}
  implies that~\(\KKcat[X]\)
  has enough projective objects and that an object is projective if
  and only if it is a direct summand of an object of the form
  \(\bigoplus_{z\in X} j_z(B_z)\)
  for some separable \(\Cst\)\nb-algebras~\(B_z\)
  for \(z\in X\).
  Equation~\eqref{eq:KKX_out_Fix2} says that the adjoint functor to
  \(F^X\colon \KKcat^X\to\KKcat[X]\)
  is defined on \(j_z(A)\)
  and maps it to~\(i_z(A)\).
  Since the partial adjoint commutes with direct sums, it is also
  defined on \(\bigoplus_{z\in X} j_z(B_z)\)
  for any \(B_z\inOb\KKcat\)
  and maps it to \(\bigoplus_{z\in X} i_z(B_z)\).
  Idempotents in the category~\(\KKcat^X\)
  split because it is triangulated and has countable direct sums.
  Therefore, the partial adjoint of~\(F^X\)
  is defined on all projective objects of~\(\KKcat[X]\)
  and is a section for~\(F^X\)
  there.  An argument as in the proof of
  Proposition~\ref{pro:FZ_universal_exact} shows that~\(F^X\)
  is the universal \(\Ideal^X\)\nb-exact
  functor to an exact category.  The target category is not smaller
  because the functor \(\KKcat^X \to \KKcat[X]\)
  is surjective on objects by Corollary~\ref{cor:classify_KKX}.
\end{proof}

Unlike in Section~\ref{sec:UCT_Z}, the \(\Ideal^X\)\nb-projective
objects in~\(\KKcat^X\)
do not generate~\(\KKcat^X\).
The localising subcategory \(\gen{\Proj_\Ideal}\)
generated by them is equal to the localising subcategory generated by
objects of the form \(i_x(A)\)
for \(x\in X\),
\(A\inOb\KKcat\).
If~\(X\)
is finite, then it is described in several equivalent ways in
\cite{Meyer-Nest:Bootstrap}*{Definition~4.7}.  It contains all
nuclear \(\Cst\)\nb-algebras over~\(X\).  This is the subcategory
on which Theorem~\ref{the:projectives_ups} implies a Universal
Coefficient Theorem, using \(\Hom\)
and \(\Ext\)
groups in the category \(\KKcat[X]\).
This implies the following classification theorem:

\begin{corollary}
  \label{cor:classify_KKX}
  Let \((X,\to)\)
  be a countable unique path space and let \(A\)
  and~\(B\)
  be \(\Cst\)\nb-algebras
  over~\(X\)
  that belong to the localising subcategory \(\gen{\Proj_\Ideal}\)
  generated by objects of the form~\(i_x(B)\)
  for \(x\in X\),
  \(B\inOb\KKcat\).
  Any isomorphism between \(F^X(A)\)
  and~\(F^X(B)\)
  in~\(\KKcat[X]\)
  lifts to an isomorphism in~\(\KKcat^X\).
  And any object of~\(\KKcat[X]\)
  lifts to an object of~\(\gen{\Proj_\Ideal}\).
  An isomorphism \(F^X(A) \congto F^X(B)\)
  is a family of invertible elements \(t_x\in \KK_0(A(U_x),B(U_x))\)
  for \(x\in X\) for which the diagrams
  \[
  \begin{tikzcd}
    A(U_x) \arrow[d, "t_x"] \arrow[r, hookrightarrow] &
    A(U_y) \arrow[d, "t_y"] \\
    B(U_x) \arrow[r, hookrightarrow] &
    B(U_y) 
  \end{tikzcd}
  \]
  commute in~\(\KKcat\) for all \(x,y\in X\) with \(x\to y\).
\end{corollary}

The criterion above may be made more explicit if, in addition,
\(A(U_x)\)
and \(B(U_x)\)
belong to the bootstrap class for all \(x\in X\).
As in Section~\ref{sec:UCT_Z}, we identify
\(A(U_x) \cong A(U_x)^+ \oplus A(U_x)^-\)
and rewrite the classes of the inclusion maps
\(\alpha_{y,x}\in \KK_0(A(U_x),A(U_y))\)
and \(\beta_{y,x}\in \KK_0(B(U_x),B(U_y))\)
for \(x\to y\)
as \(2\times2\)\nb-matrices,
and similarly for the arrows \(t_x \in \KK_0(A(U_x),B(U_x))\).
As in the case of~\(\KKcat^\Z\),
the equality \(\beta_{y,x} t_x = t_y \alpha_{y,x}\)
for \(x\to y\)
in Corollary~\ref{cor:classify_KKX} may be rewritten as four
equalities of matrix coefficients.  The equality of the diagonal terms
says that the diagrams \(\bigl(\K_*(A(U_x)),\K_*(\alpha_{y,x})\bigr)\)
and \(\bigl(\K_*(B(U_x)),\K_*(\beta_{y,x})\bigr)\)
of countable \(\Z/2\)\nb-graded
Abelian groups are isomorphic.  That is, for all \(x,y\in X\)
with \(x\to y\), the following diagram commutes:
\[
\begin{tikzcd}[column sep=large]
  \K_*(A(U_x)) \arrow[d, "\K_*(t_x)"] \arrow[r, "\K_*(\alpha_{y,x})"] &
  \K_*(A(U_y)) \arrow[d, "\K_*(t_y)"] \\
  \K_*(B(U_x)) \arrow[r, "\K_*(\beta_{y,x})"] &
  \K_*(B(U_y))
\end{tikzcd}
\]
The equality of the off-diagonal terms will be studied in
Section~\ref{sec:compare_ups}.

\section{Computation of the obstruction class}
\label{sec:triangulated}

We recall the setup of~\cite{Bentmann-Meyer:More_general}.
Let~\(\Tri\)
be a triangulated category with countable direct sums.  Let~\(\Ideal\)
be a stable homological ideal in~\(\Tri\)
with enough projective objects.  Let \(F\colon \Tri\to\Abel\)
be the universal \(\Ideal\)\nb-exact
stable homological functor (in this article, we allow~\(\Abel\)
to be exact).  We assume that~\(\Abel\)
is paired as in \cite{Bentmann-Meyer:More_general}*{Definition~2.14},
that is, \(\Abel = \Abel_+ \times \Abel_-\)
with \(\Sigma \Abel_+ = \Abel_-\)
and \(\Sigma \Abel_- = \Abel_+\).
For instance, \(\Abel\)
could be the category of \(\Z\)\nb-graded
or \(\Z/2\)\nb-graded
modules over some ring, with the suspension automorphism shifting the
grading, and \(\Abel_+\)
and~\(\Abel_-\)
may be taken to be the graded modules concentrated in even or odd
degrees, respectively.  We want to compute the obstruction class of an
object \(A\inOb\Tri\).
This is only meaningful if~\(A\)
is constructed from simpler ingredients.

We assume that there is an exact, \(\Ideal\)\nb-exact triangle
\begin{equation}
  \label{eq:triangle_B1B0A}
  B_1 \xrightarrow{\varphi} B_0 \xrightarrow{p} A \xrightarrow{i}
  \Sigma B_1,
\end{equation}
where \(B_0\)
and~\(B_1\)
are objects of~\(\gen{\Proj_\Ideal}\)
with projective resolutions of length~\(1\).
Then \(A\inOb\gen{\Proj_\Ideal}\)
as well.  The \(\Ideal\)\nb-exactness
assumption says that \(i\in\Ideal\).
Equivalently, \(F(i)=0\),
\(F(\varphi)\)
is monic, and \(F(p)\)
is epic.  The objects \(B_0\)
and~\(B_1\)
are uniquely determined up to isomorphism by \(F(B_0)\)
and~\(F(B_1)\)
because they have projective resolutions of length~\(1\)
(compare \cite{Bentmann-Meyer:More_general}*{Proposition~2.3}).  We
may split \(B_i \cong B_i^+ \oplus B_i^-\)
with \(F(B_i^\pm)\in \Abel_\pm\)
for \(i=0,1\).  Then \(\varphi\) becomes a \(2\times2\)-matrix
\[
\varphi =
\begin{pmatrix}
  \varphi_{++}& \varphi_{+-}\\
  \varphi_{-+}& \varphi_{--}
\end{pmatrix}
\]
with \(\varphi_{-+}\colon B_1^+ \to B_0^-\),
and so on.  The two diagonal entries give an element of
\[
\Tri(B_1^+,B_0^+) \oplus \Tri(B_1^-,B_0^-) \cong
\Hom_{\Abel}\bigl(F(B_1),F(B_0)\bigr),
\]
and the two off-diagonal entries give an element of
\[
\Tri(B_1^+,B_0^-) \oplus \Tri(B_1^-,B_0^+) \cong
\Ext^1_{\Abel}\bigl(\Sigma F(B_1),F(B_0)\bigr);
\]
here we have used the Universal Coefficient Theorem to compute
\(\Tri(B_1^\pm,B_0^\pm)\).
The results for these four groups only have a single Hom or a single
Ext group because of the parity assumptions.  Thus the splitting of
\(\varphi\)
into \(\varphi^0\defeq\varphi_{++}+\varphi_{--}\)
and \(\varphi^1\defeq\varphi_{+-}+\varphi_{-+}\)
splits the exact sequence
\[
\Ext^1_{\Abel}\bigl(\Sigma F(B_1),F(B_0)\bigr) \into
\Tri_*(B_1,B_0) \prto
\Hom_{\Abel}\bigl(F(B_1),F(B_0)\bigr).
\]
We shall compute the obstruction class of~\(A\)
in terms of~\(\varphi^1\).

By assumption, there is a short exact sequence
\[
\begin{tikzcd}
  F(B_1) \arrow[r, rightarrowtail, "F(\varphi)"] &
  F(B_0) \arrow[r, twoheadrightarrow, "F(p)"] &
  F(A).
\end{tikzcd}
\]
This induces a long exact sequence
\begin{multline*}
  0 \leftarrow \Ext^2_{\Abel}\bigl(\Sigma F(A),F(A)\bigr)
  \xleftarrow{\partial} \Ext^1_{\Abel}\bigl(\Sigma F(B_1),F(A)\bigr)
  \xleftarrow{F(\varphi)^*}\\
  \Ext^1_{\Abel}\bigl(\Sigma F(B_0),F(A)\bigr)
  \xleftarrow{F(p)^*} \Ext^1_{\Abel}\bigl(\Sigma F(A),F(A)\bigr)
  \leftarrow \dotsb
\end{multline*}
because \(\Ext^k_{\Abel}\bigl(F(B_i),F(A)\bigr)=0\) for \(i=0,1\), \(k\ge2\).

\begin{theorem}
  \label{the:compute_obstruction}
  The obstruction class of~\(A\) is
  \[
  \partial(p_*\varphi^1)
  = \partial(F(p)\circ \varphi^1)\in
  \Ext^2_{\Abel}\bigl(\Sigma F(A),F(A)\bigr),
  \]
  where \(F(p)\in \Hom_{\Abel}\bigl(\Sigma F(B_0),\Sigma F(A)\bigr)\)
  and \(\varphi^1\in \Ext^1_{\Abel}\bigl(\Sigma F(B_1),F(B_0)\bigr)\).
\end{theorem}

\begin{proof}
  We shall recall the construction of obstruction classes
  in~\cite{Bentmann-Meyer:More_general} along the way.  It starts with
  an \(\Ideal\)\nb-projective
  resolution of~\(A\)
  of length~\(2\).
  So first we have to construct this.  We use the projective
  resolutions of~\(F(B_i)\)
  of length~\(1\),
  which exist by assumption.  They lift canonically to
  \(\Ideal\)\nb-projective resolutions
  \begin{equation}
    \label{eq:projective_resolutions_Bi}
    0 \to P_{i1} \xrightarrow{d_{i1}} P_{i0} \xrightarrow{d_{i0}} B_i
  \end{equation}
  in~\(\Tri\)
  for \(i=0,1\)
  (see \cite{Meyer-Nest:Homology_in_KK}*{Theorem~59}).
  Since~\eqref{eq:projective_resolutions_Bi} is a resolution,
  \(d_{i1}\)
  is \(\Ideal\)\nb-monic
  and \(d_{i0}\)
  is \(\Ideal\)\nb-epic.
  The arrow \(\varphi\in\Tri(B_1,B_0)\) lifts to a chain map
  \begin{equation}
    \label{eq:lift_varphi_resolutions}
    \begin{tikzcd}[baseline=(current bounding box.west)]
      P_{11} \arrow[r, rightarrowtail, "d_{11}"]
      \arrow[d, "\Phi_1"] &
      P_{10} \arrow[r, twoheadrightarrow, "d_{10}"]
      \arrow[d, "\Phi_0"] & B_1
      \arrow[d, rightarrowtail, "\varphi"] \\
      P_{01} \arrow[r, rightarrowtail, "d_{01}"] &
      P_{00} \arrow[r, twoheadrightarrow, "d_{00}"] & B_0
    \end{tikzcd}
  \end{equation}
  between the \(\Ideal\)\nb-projective
  resolutions~\eqref{eq:projective_resolutions_Bi} (see
  \cite{Meyer-Nest:Homology_in_KK}*{Proposition~44}).  We
  write~\(\into\)
  for \(\Ideal\)\nb-monic
  and~\(\prto\) for \(\Ideal\)\nb-epic maps.  We claim that
  \begin{equation}
    \label{eq:resolution_A}
    0 \to
    P_{11} \xrightarrow{(-d_{11},\Phi_1)} P_{10} \oplus P_{01}
    \xrightarrow{(\Phi_0,d_{01})} P_{00}
    \xrightarrow{p\circ d_{00}} A
    \to 0
  \end{equation}
  is an \(\Ideal\)\nb-projective
  resolution of~\(A\)
  of length~\(2\).
  The entries are \(\Ideal\)\nb-projective
  by construction.  Next we prove that~\eqref{eq:resolution_A} is a
  resolution, that is, it becomes an exact chain complex when we
  apply~\(F\) to it.

  When we apply~\(F\)
  to the diagram~\eqref{eq:lift_varphi_resolutions}, the two rows
  become short exact sequences in~\(\Abel\),
  and the vertical maps become a chain map between them.  The mapping
  cone of this chain map is again an exact chain complex in~\(\Abel\).
  It has the form
  \[
  0 \to F(P_{11})
  \to F(P_{10}) \oplus F(P_{01})
  \to F(P_{00}) \oplus F(B_1)
  \to F(B_0) \to 0.
  \]
  The map \(F(\varphi)\colon F(B_1) \to F(B_0)\)
  is monic with cokernel~\(F(A)\).
  So the direct summand \(F(B_1)\)
  and its image in \(F(B_0)\)
  together form a contractible subcomplex.  The quotient by it is
  again an exact chain complex in~\(\Abel\).
  This is what we get by applying~\(F\)
  to~\eqref{eq:resolution_A}.  So this is a resolution as asserted.

  An axiom for triangulated categories provides an exact triangle
  \[
  \begin{tikzcd}[column sep=large]
    P_{11} \arrow[r, rightarrowtail, "{(-d_{11},\Phi_1)}"] &
    P_{10} \oplus P_{01}  \arrow[r, twoheadrightarrow] &
    D \arrow[r] & \Sigma P_{11}
  \end{tikzcd}
  \]
  containing \((-d_{11},\Phi_1)\).
  Similarly, the map \(p\circ d_{00}\colon P_{00}\to A\)
  in~\eqref{eq:resolution_A} is part of an exact triangle
  \begin{equation}
    \label{eq:DP00A}
    \begin{tikzcd}[column sep=large]
      D' \arrow[r, rightarrowtail, "\gamma"] &
      P_{00}  \arrow[r, twoheadrightarrow, "p\circ d_{00}"] &
      A \arrow[r] & \Sigma D'.
    \end{tikzcd}
  \end{equation}
  The long exact sequences for~\(F\)
  applied to these two exact triangles show that \(F(D)\)
  is the cokernel of the monomorphism \(F(-d_{11},\Phi_1)\),
  and that \(F(D')\)
  is the kernel of the epimorphism \(F(p\circ d_{00})\).
  The exactness of~\eqref{eq:resolution_A} implies
  \(F(D)\cong F(D')\).
  Since \(B_0\)
  and~\(B_1\)
  belong to \(\gen{\Proj_\Ideal}\),
  so do \(D\)
  and~\(D'\).
  And \(F(D)\)
  has a projective resolution of length~\(1\)
  by construction.  Hence the Universal Coefficient Theorem applies to \(D\)
  and~\(D'\).
  Thus the isomorphism \(F(D)\cong F(D')\)
  lifts to an isomorphism \(D\cong D'\).
  We shall identify \(D=D'\).

  The Universal Coefficient Theorem for~\(D\) gives a short exact sequence
  \begin{equation}
    \label{eq:UCT_D_P00}
    \begin{tikzcd}[column sep=scriptsize]
      \Ext^1_{\Abel}\bigl(\Sigma F(D), F(P_{00})\bigr) \arrow[r, rightarrowtail] &
      \Tri(D,P_{00}) \arrow[r, twoheadrightarrow, "F"] &
      \Hom_{\Abel}\bigl(F(D),F(P_{00})\bigr).
    \end{tikzcd}
  \end{equation}
  We split \(D=D^+\oplus D^-\)
  and \(P_{00} = P_{00}^+ \oplus P_{00}^-\)
  into objects of even and odd parity as in the construction
  of~\(\varphi^1\)
  above the theorem.  Then \(\Tri(D,P_{00})\)
  splits accordingly as a \(2\times2\)-matrix.
  The sum of the diagonal terms
  \(\Tri(D^+,P_{00}^+) \oplus\Tri(D^-,P_{00}^-)\)
  is isomorphic to \(\Hom_{\Abel}\bigl(F(D),F(P_{00})\bigr)\),
  whereas the sum of the off-diagonal terms
  \(\Tri(D^-,P_{00}^+) \oplus\Tri(D^+,P_{00}^-)\)
  is isomorphic to \(\Ext^1_{\Abel}\bigl(\Sigma F(D),F(P_{00})\bigr)\).
  This is the (unnatural) splitting of the Universal Coefficient Theorem exact
  sequence~\eqref{eq:UCT_D_P00} that follows because~\(\Abel\)
  is paired.  In particular, we decompose
  \(\gamma = \gamma^0 + \gamma^1\)
  into its parity-preserving and parity-reversing parts.

  Let~\(A'\)
  be another object of~\(\gen{\Proj_\Ideal}\)
  with an isomorphism \(F(A') \cong F(A)\).
  The argument above shows that both \(A\)
  and~\(A'\)
  are cones of some \(\gamma,\gamma'\in\Tri(D, P_{00})\)
  as in~\eqref{eq:DP00A}, which lift the inclusion map
  \(F(D) \into F(P_{00})\)
  that we get from the resolution~\eqref{eq:resolution_A}.  The
  \emph{relative obstruction class} is defined as follows: compose
  \[
  \gamma-\gamma'\in \Ext^1_{\Abel}\bigl(\Sigma F(D), F(P_{00})\bigr) \subseteq
  \Tri(D,P_{00})
  \]
  with the map \(F(p\circ d_{00})\colon F(P_{00}) \to F(A)\)
  and apply the boundary map for the extension
  \(F(D) \into F(P_{00}) \prto F(A)\).
  That is, plug \(\gamma-\gamma'\) into the map
  \begin{multline}
    \label{eq:transfer_to_Ext2}
    \Ext^1_{\Abel}\bigl(\Sigma F(D),F(P_{00})\bigr) \xrightarrow{F(p\circ d_{00})_*}
    \Ext^1_{\Abel}\bigl(\Sigma F(D),F(A)\bigr)
    \\ \xrightarrow{\partial_{DP_{00}A}} \Ext^2_{\Abel}\bigl(\Sigma F(A),F(A)\bigr).
  \end{multline}
  Let \(\gamma^0\colon D\to P_{00}\)
  be the unique parity-preserving arrow that lifts the inclusion map
  \(F(D) \into F(P_{00})\)
  and let~\(A^0\)
  be its cone.  The \emph{obstruction class} of~\(A\)
  is the relative obstruction class for \(A\)
  and~\(A^0\).
  That is, we plug \(\gamma^1 \defeq \gamma- \gamma^0\)
  into~\eqref{eq:transfer_to_Ext2}.

  Finally, we relate the obstruction class defined above
  to~\(\varphi^1\).
  The solid square in the following diagram commutes:
  \begin{equation}
    \label{eq:compare_triangles_BBA_DPA}
    \begin{tikzcd}[baseline=(current bounding box.west)]
      B_1  \arrow[r, rightarrowtail, "\varphi"] &
      B_0  \arrow[r, twoheadrightarrow, "p"] &
      A \arrow[r] &
      \Sigma B_1\\
      D \arrow[r, rightarrowtail, "\gamma"]
      \arrow[u, dotted, "\varepsilon"] &
      P_{00} \arrow[r, twoheadrightarrow, "p\circ d_{00}"]
      \arrow[u, twoheadrightarrow, "d_{00}"]&
      A \arrow[r]
      \arrow[u, equal] &
      \Sigma D \arrow[u, dotted, "\Sigma \varepsilon"]
    \end{tikzcd}
  \end{equation}
  By the third axiom of triangulated categories, there is an
  arrow~\(\varepsilon\)
  making all three squares commute.  We shall only use the left
  square.  Since \(D\),
  \(B_1\),
  \(B_0\)
  and~\(P_{00}\)
  have projective resolutions of length~\(1\)
  and~\(\Abel\)
  is paired, the Universal Coefficient Theorem allows us to split them
  into even and odd parts.  Hence each of the arrows \(\gamma\),
  \(\varphi\),
  \(\varepsilon\)
  and~\(d_{00}\)
  splits into a parity-preserving and a parity-reversing part.  We
  have already used the splittings \(\varphi=\varphi^0+\varphi^1\)
  and \(\gamma=\gamma^0+\gamma^1\),
  and now we also split \(\varepsilon=\varepsilon^0+\varepsilon^1\).
  The arrow \(d_{00}\)
  is parity-preserving because
  \(\Tri(P_{00},B_0) \cong \Hom_{\Abel}\bigl(F(P_{00}),F(B_0)\bigr)\)
  has no parity-reversing part.  The left commuting square
  in~\eqref{eq:compare_triangles_BBA_DPA} implies
  \begin{equation}
    \label{eq:odd_part_relation}
    d_{00}\circ \gamma^1
    = \varphi^0\circ \varepsilon^1 + \varphi^1\circ\varepsilon^0
    = \varphi\circ \varepsilon^1 + \varphi^1\circ\varepsilon.
  \end{equation}
  The second step uses that the composite of two odd terms always
  vanishes, that is, the \(\Ext^1\)-term in the Universal Coefficient Theorem is nilpotent.

  Composing \(\gamma^1\in\Ext^1_{\Abel}\bigl(\Sigma F(D),F(P_{00})\bigr)\)
  with \(F(p\circ d_{00})\)
  as in~\eqref{eq:transfer_to_Ext2} has the same effect as composing
  with \(p\circ d_{00}\)
  because the exact sequence in the Universal Coefficient Theorem is natural.  Thus the
  obstruction class of~\(A\)
  is the image of
  \(p\circ d_{00}\circ \gamma^1 \in\Ext^1_{\Abel}\bigl(\Sigma F(D),F(A)\bigr)\)
  under the boundary map
  \[
  \partial_{DP_{00}A}\colon
  \Ext^1_{\Abel}\bigl(\Sigma F(D),F(A)\bigr) \to \Ext^2_{\Abel}\bigl(\Sigma F(A),F(A)\bigr)
  \]
  Equation~\eqref{eq:odd_part_relation} and \(p\circ\varphi=0\) imply
  \[
  p\circ d_{00}\circ\gamma^1
  = p\circ \varphi\circ \varepsilon^1 + p\circ \varphi^1\circ\varepsilon
  = p\circ \varphi^1\circ\varepsilon.
  \]
  The composite
  \(p\circ \varphi^1 = p_*(\varphi^1) \in
  \Ext^1_{\Abel}\bigl(\Sigma F(B_1),F(A)\bigr)\)
  also appears in the statement of the theorem.  Composing
  with~\(\varepsilon\)
  in~\(\Tri\)
  has the same effect as composing with \(F(\varepsilon)\)
  in the graded category \(\Ext^*_{\Abel}\).
  When we apply~\(F\)
  to the morphism of exact
  triangles~\eqref{eq:compare_triangles_BBA_DPA}, we get the following
  morphism of extensions in~\(\Abel\):
  \[
  \begin{tikzcd}[column sep=huge]
    F(B_1)  \arrow[r, rightarrowtail, "F(\varphi)"] &
    F(B_0)  \arrow[r, twoheadrightarrow, "F(p)"] &
    F(A)\\
    F(D) \arrow[r, rightarrowtail, "F(\gamma)"]
    \arrow[u, dotted, "F(\varepsilon)"] &
    F(P_{00}) \arrow[r, twoheadrightarrow, "F(p)\circ F(d_{00})"]
    \arrow[u, twoheadrightarrow, "F(d_{00})"]&
    F(A) \arrow[u, equal]
  \end{tikzcd}
  \]
  Since boundary maps in Ext-theory are natural, the boundary map
  \[
  \partial\colon \Ext^1_{\Abel}\bigl(\Sigma F(B_1),F(A)\bigr) \to
  \Ext^2_{\Abel}\bigl(\Sigma F(A),F(A)\bigr)
  \]
  for the top row is equal to the composite of \(F(\varepsilon)\)
  and the boundary map~\(\partial_{DP_{00}A}\).
  Thus the obstruction class of~\(A\)
  is \(\partial(F(p)\circ\varphi^1)\) as asserted.
\end{proof}

\section{Comparison of classification theorems}
\label{sec:compare_classification}

In this section, we apply Theorem~\ref{the:compute_obstruction} in
several cases to relate the classification theorem involving the
obstruction class to other classification theorems.  We first compare
the classification for \(\Z\)\nb-actions
in Theorem~\ref{the:Z_classify} with the one obtained
in~\cite{Bentmann-Meyer:More_general}.  Then we compare the
classification for \(\Cst\)\nb-algebras
over a unique path space~\((X,\to)\)
in Corollary~\ref{cor:classify_KKX} with the one
in~\cite{Bentmann-Meyer:More_general}.  In both cases, the invariants
can be translated into each other rather directly.

\subsection{Actions of the integers}
\label{sec:Z-action}

The Universal Coefficient Theorem for \(\Z\)\nb-actions
in Section~\ref{sec:UCT_Z} is based on the stable homological
ideal~\(\Ideal^\Z\)
defined by the forgetful functor \(\KKcat^\Z\to\KKcat\).
Now we use another ideal~\(\Ideal^\Z_\K\).
Let~\(\Abel^\Z\)
be the category of countable \(\Z/2\)-graded
\(\Z[x,x^{-1}]\)\nb-modules.  Let
\[
F^\Z_\K\colon \KKcat^\Z\to\Abel^\Z,\qquad
(A,\alpha) \mapsto \bigl(\K_*(A),\K_*(\alpha)\bigr),
\]
that is, we map \((A,\alpha)\inOb \KKcat^\Z\)
to the \(\Z/2\)\nb-graded
Abelian group \(\K_*(A)\)
with the \(\Z[x,x^{-1}]\)\nb-module
structure given by the automorphism \(\K_*(\alpha)\)
of \(\K_*(A)\).  Let
\[
\Ideal^\Z_\K(A,B) \defeq
\setgiven{\varphi\in \KK^\Z_0(A,B)}{F^\Z_\K(\varphi)=0}
\]
be the kernel of~\(F^\Z_\K\)
on morphisms.  This example is treated
in~\cite{Bentmann-Meyer:More_general}, but there \(\KKcat^\Z\)
is disguised as \(\KKcat^\T\)
for the circle group~\(\T\).
These two categories are equivalent by Baaj--Skandalis duality
(see~\cite{Baaj-Skandalis:Hopf_KK}).  The functor~\(F^\Z_\K\)
corresponds to the functor on \(\KKcat^\T\)
that maps a \(\Cst\)\nb-algebra
with a continuous \(\T\)\nb-action
to \(\K_*^\T(A)\)
with the canonical module structure over the representation ring
\(\Z[x,x^{-1}]\)
of~\(\T\), which is used in~\cite{Bentmann-Meyer:More_general}.

If \(B\inOb\KK^\Z\), then~\eqref{eq:FZ_adjoint} implies
\[
\KK^\Z_0(\Cont_0(\Z), B)
\cong \KK_0(\C,B)
\cong \K_0(B)
\cong \Hom_{\Abel^\Z}\bigl(\Z[x,x^{-1}], F^\Z_\K(B)\bigr).
\]
Thus the partial adjoint of~\(F^\Z_\K\)
is defined on the rank-\(1\)
free module \(\Z[x,x^{-1}]\)
and maps it to \(\Cont_0(\Z)\).
Then it is defined on all free modules, also in odd parity.
Since idempotents in~\(\KKcat^\Z\)
split, the partial adjoint~\((F^\Z_\K)^\lad\)
of~\(F^\Z_\K\)
is defined on all projective objects of~\(\Abel^\Z\).
Since \(F^\Z_\K(\Cont_0(\Z))\)
is the rank-\(1\)
free module again, we get
\((F^\Z_\K) \circ (F^\Z_\K)^\lad(P) \cong P\)
for all projective objects of~\(\Abel^\Z\).
Any object in~\(\Abel^\Z\)
has a projective resolution of length~\(2\).
Hence it belongs to the image of~\(F^\Z_\K\)
by \cite{Bentmann-Meyer:More_general}*{Lemma~2.4}.
Remark~\ref{rem:universal_to_exact} and
Theorem~\ref{the:characterise_universal_homological} show
that~\(F^\Z_\K\)
is the universal \(\Ideal^\Z_\K\)\nb-exact
stable homological functor both to an exact and to an Abelian
category.

The category~\(\Abel^\Z\)
is paired in an obvious way, using the subcategories~\(\Abel^\Z_\pm\)
of countable \(\Z[x,x^{-1}]\)\nb-modules
concentrated in degree \(0\)
and~\(1\),
respectively.  So the obstruction class is defined for any object
of~\(\KKcat^\Z\).
Assume that \((A,\alpha)\inOb \KKcat^\Z\)
is such that~\(A\)
belongs to the bootstrap class in~\(\KKcat\).
Equivalently, \((A,\alpha)\)
belongs to the localising subcategory of~\(\KKcat^\Z\)
generated by the \(\Ideal^\Z_\K\)\nb-projective
object \(\Cont_0(\Z)\).
Then the main result of~\cite{Bentmann-Meyer:More_general} shows
that~\(A\)
is determined uniquely up to isomorphism by
\(\bigl(\K_*(A),\K_*(\alpha)\bigr)\inOb \Abel^\Z\)
and the obstruction class.

A chain complex in~\(\KKcat^\Z\)
that is \(\Ideal^\Z\)\nb-exact
is also \(\Ideal^\Z_\K\)\nb-exact
because \(\Ideal^\Z \subseteq \Ideal^\Z_\K\).
So the \(\Ideal^\Z\)\nb-projective
resolution in~\eqref{eq:lift_AZ_resolution} is
\(\Ideal^\Z_\K\)\nb-exact.
Its entries \(\Cont_0(\Z,A)\)
are no longer \(\Ideal^\Z_\K\)\nb-projective.
We claim, however, that they have \(\Ideal^\Z_\K\)\nb-projective
resolutions of length~\(1\).
Let \(P_1 \into P_0\prto \K_*(A)\)
be a resolution of the \(\Z/2\)\nb-graded
Abelian group~\(\K_*(A)\) by countable free Abelian groups.  Then
\[
\Z[x,x^{-1}] \otimes P_1 \into \Z[x,x^{-1}] \otimes P_0
\prto \Z[x,x^{-1}] \otimes \K_*(A)
\]
is a projective resolution of
\[
F^\Z_\K(\Cont_0(\Z,A)) \cong \Z[x,x^{-1}] \otimes \K_*(A).
\]
Here the tensor products with \(\Z[x,x^{-1}]\)
carry the module structure defined by multiplication with~\(x\)
in the tensor factor~\(\Z[x,x^{-1}]\).
So we are in the situation of~\eqref{eq:triangle_B1B0A}.
Theorem~\ref{the:compute_obstruction} implies the following formula
for the obstruction class of~\(A\):

\begin{theorem}
  \label{the:obstruction_class_Z}
  Let~\(A\)
  be a \(\Cst\)\nb-algebra
  in the bootstrap class in~\(\KKcat\)
  and \(\alpha\in\Aut(A)\).
  Split \([\alpha]\in \KK_0(A,A)\)
  into parity-preserving and parity-reversing parts
  \(\alpha^0\in \Hom\bigl(\K_*(A),\K_*(A)\bigr)\)
  and \(\alpha^1\in \Ext\bigl(\K_{1+*}(A),\K_*(A)\bigr)\).  Define
  \[
  \gamma\colon \Ext^1_\Z\bigl(\K_{1+*}(A),\K_*(A)\bigr) \to
  \Ext^1_\Z\bigl(\K_{1+*}(A),\K_*(A)\bigr),\qquad
  x\mapsto \alpha\circ x - x\circ\alpha.
  \]
  Then
  \(\Ext^2_{\Z[x,x^{-1}]}\bigl(\Sigma F^\Z_\K(A),F^\Z_\K(A)\bigr)
  \cong \coker \gamma\)
  and the obstruction class of~\(A\)
  is the image of \(-\alpha^1 (\alpha^0)^{-1}\) in this cokernel.
\end{theorem}

\begin{proof}
  In our case, the map~\(\varphi\)
  in~\eqref{eq:triangle_B1B0A} is the map
  \(\Cont_0(\Z) \otimes A \to \Cont_0(\Z) \otimes A\)
  in~\eqref{eq:lift_AZ_resolution}.  Split \(A=A^+ \oplus A^-\)
  into its even and odd parts as before.  Then
  \[
  \Cont_0(\Z) \otimes A \cong \Cont_0(\Z) \otimes A^+ \oplus \Cont_0(\Z)\otimes A^-
  \]
  is the parity decomposition of~\(\Cont_0(\Z) \otimes A\).
  The translation~\(\tau\)
  and the identity are parity-preserving.  Decompose \([\alpha]\)
  and~\([\alpha^{-1}]\)
  into their even and odd parts.  The parity-reversing part of the map
  on \(\Cont_0(\Z) \otimes A\)
  is \(\varphi^1 = [\tau]\otimes [\alpha^{-1}]^1\).
  The map \(p\in \KK^\Z_0(\Cont_0(\Z) \otimes A,A)\)
  satisfies \(p (\tau\otimes \id_A) = [\alpha] p\)
  because it is \(\Z\)\nb-equivariant,
  and it restricts to the identity map on the \(0\)th
  summand of \(\Cont_0(\Z) \otimes A\).
  The isomorphism
  \(\KK^\Z_0\bigl((\Cont_0(\Z,A),\tau), (B,\beta) \bigr) \cong
  \KK_0(A,B)\)
  in~\eqref{eq:FZ_adjoint} forgets the \(\Z\)\nb-action
  and then evaluates at~\(0\).
  Therefore, the image of~\(\varphi^1\)
  in \(\KK_0(A,A)\)
  is the restriction of \(p\circ \varphi^1\)
  to the \(0th\)
  summand.  And this is \([\alpha] [\alpha^{-1}]^1 \in \KK_0(A,A)\).
  Since \(\alpha \alpha^{-1} = \id_A\)
  and products of two parity-reversing \(\KK\)\nb-classes
  always vanish,
  \([\alpha] [\alpha^{-1}]^1 + [\alpha]^1 [\alpha^{-1}] = [\id_A]^1 =
  0\).
  Theorem~\ref{the:compute_obstruction} now says that the obstruction
  class for~\(A\)
  is the image of
  \[
  [\alpha] [\alpha^{-1}]^1 = - [\alpha]^1 [\alpha^{-1}] = -\alpha^1 (\alpha^0)^{-1}
  \in \Ext_\Z\bigl(\K_{1+*}(A),\K_*(A))\bigr)
  \]
  under the boundary map to
  \(\Ext^2_{\Z[x,x^{-1}]}\bigl(\Sigma F^\Z_\K(A), F^\Z_\K(A)\bigr)\).
  The description of the latter Ext group in the theorem follows when
  we compute it with the projective resolution
  in~\eqref{eq:resolution_A} defined by the length-\(1\)
  resolutions of \(\Cont_0(\Z) \otimes A\)
  above.  This computation has already been done in
  \cite{Bentmann-Meyer:More_general}*{Section~3.2}.  When the Ext
  groups are described in this way, the boundary map in
  Theorem~\ref{the:compute_obstruction} becomes a trivial map, mapping
  an element of \(\Ext_\Z\bigl(\K_{1+*}(A),\K_*(A)\bigr)\)
  to its image in \(\coker \gamma\).
\end{proof}

The formula for the obstruction class depends, of course, on the
isomorphism
\(\Ext^2_{\Z[x,x^{-1}]}\bigl(\Sigma F^\Z_\K(A),F^\Z_\K(A)\bigr) \cong
\coker \gamma\),
and this depends on the \(\Ideal^\Z\)\nb-projective
resolution from which it is obtained.  Our theorem uses the most
obvious resolution.  One may apply the automorphism of
\(\Ext^2_{\Z[x,x^{-1}]}\bigl(\Sigma F^\Z_\K(A),F^\Z_\K(A)\bigr)\)
that composes with the automorphism~\(-\Sigma\alpha^0\)
on \(\Sigma F^\Z_\K(A)\).
This replaces the obstruction class~\(-\alpha^1 (\alpha^0)^{-1}\)
by~\(\alpha^1\).
So the images of \(-\alpha^1 (\alpha^0)^{-1}\)
and~\(\alpha^1\)
in \(\Ext^2_{\Z[x,x^{-1}]}\bigl(\Sigma F^\Z_\K(A),F^\Z_\K(A)\bigr)\)
contain the same information.

Theorem~\ref{the:obstruction_class_Z} and the computations after
Theorem~\ref{the:Z_classify} allow to deduce the classification
by~\(F^\Z_\K\)
and the obstruction class from the classification by the
invariant~\(F^\Z\)
in Section~\ref{sec:UCT_Z}.  Let \((A,\alpha)\)
and \((B,\beta)\)
belong to the bootstrap class in~\(\KKcat^\Z\).
Assume that there is an isomorphism
\(t^0\colon F^\Z_\K(A) \congto F^\Z_\K(B)\)
in~\(\Abel^\Z\),
that is, a grading-preserving \(\Z[x,x^{-1}]\)-module
isomorphism \(\K_*(A) \cong \K_*(B)\).
We use this isomorphism to identify \(\K_*(A) = \K_*(B)\).
By Theorem~\ref{the:obstruction_class_Z}, the relative obstruction
class vanishes if and only if
\(\alpha^1 - \beta^1 \in \Ext\bigl(\K_{*+1}(A),\K_*(A)\bigr)\)
vanishes in \(\coker(\gamma)\).
Equivalently, there is \(t^1\in\Ext^1\bigl(\K_{1+*}(A),\K_*(B)\bigr)\)
for which \(t=t^0+t^1\in \KK_0(A,B)\)
satisfies \([\beta]\circ t = t\circ [\alpha]\).
Then~\(t\)
is an isomorphism \(F^\Z(A) \cong F^\Z(B)\)
in~\(\KKcat[\Z]\),
and Theorem~\ref{the:Z_classify} shows that such an isomorphism lifts
to an isomorphism in \(\KK^\Z_0(A,B)\).

Recall that Baaj--Skandalis duality is an equivalence of triangulated
categories \(\KKcat^\T \cong \KKcat^\Z\).
Hence everything said above about \(\Z\)\nb-actions
carries over to \(\T\)\nb-actions.
The functor \(A\mapsto \K_*(A)\)
becomes \(B\mapsto \K_*^\T(B)\)
on \(\KKcat^\T\),
equipped with the natural module structure over the representation
ring~\(\Z[x,x^{-1}]\)
of~\(\T\).
The automorphism~\(\alpha\)
is replaced by an automorphism of \(B\rtimes \T\),
namely, the generator~\(\beta\)
of the dual action of~\(\Z\).
Since \(\K_*(B\rtimes\T) \cong \K_*^\T(B)\),
the Universal Coefficient Theorem splits
\(\KK_0(B\rtimes\T,B\rtimes\T)\)
into the parity-preserving part
\(\Hom\bigl(\K_*^\T(B),\K_*^\T(B)\bigr)\)
and the parity-reversing part
\(\Ext\bigl(\K_{1+*}^\T(B),\K_*^\T(B)\bigr)\).
The obstruction class of \(B\rtimes\T\)
is the class of
\(-\beta^1 (\beta^0)^{-1} \in
\Ext\bigl(\K_{1+*}^\T(B),\K_*^\T(B)\bigr)\) in the cokernel of the map
\[
\gamma\colon \Ext\bigl(\K_{1+*}^\T(B),\K_*^\T(B)\bigr) \to
\Ext\bigl(\K_{1+*}^\T(B),\K_*^\T(B)\bigr),\qquad
x\mapsto \beta^0 x - x\beta^0.
\]
As above, a grading-preserving \(\Z[x,x^{-1}]\)-module
isomorphism \(t\colon \K_*^\T(A) \to \K_*^\T(B)\)
is compatible with the obstruction classes if and only if it lifts to
an isomorphism between \(A\rtimes\T\)
and~\(B\rtimes\T\)
in \(\KKcat[\Z]\),
and such an isomorphism lifts further to an isomorphism between
\(A\rtimes\T\)
and~\(B\rtimes\T\)
in \(\KKcat^\Z\).
By Baaj--Skandalis duality, the latter is equivalent to an invertible
element in \(\KK^\T_0(A,B)\).

\subsection{\texorpdfstring{$\Cst$}{C*}-Algebras over unique path spaces}
\label{sec:compare_ups}

Now we return to the setup of Section~\ref{sec:UCT_X}.  So \((X,\to)\)
is a countable directed graph with the unique path property.  Let~\(\preceq\)
be the partial order generated by~\(\leftarrow\)
and equip~\(X\)
with the Alexandrov topology defined by~\(\preceq\).
Let \(\KKcat^X\)
be the category of \(\Cst\)\nb-algebras
over~\(X\).
Objects in the appropriate bootstrap class in \(\KKcat^X\)
are classified in~\cite{Bentmann-Meyer:More_general} under the extra
assumption that~\(X\)
be finite.  Actually, the arguments
in~\cite{Bentmann-Meyer:More_general} work in the same way if the
set~\(X\)
is countable.  Here we treat this more general case right away.

Let~\(\Abel^X\)
be the category of all functors from~\(X\)
to the category of countable \(\Z/2\)-graded
Abelian groups.  Equivalently, an object of~\(\Abel^X\)
is a family of countable \(\Z/2\)\nb-graded
Abelian groups~\(G_x\)
for \(x\in X\)
with grading-preserving group homomorphisms
\(\gamma_{y,x}\colon G_x \to G_y\)
for all \(x,y\in X\)
with \(x\to y\).
Morphisms \((G_x,\gamma_{y,x}) \to (H_x,\eta_{y,x})\)
in~\(\Abel^X\)
are families of grading-preserving group homomorphisms
\(t_x\colon G_x \to H_x\)
that satisfy \(t_y\circ \gamma_{y,x} = \eta_{y,x}\circ t_x\)
for all \(x,y\in X\) with \(x\to y\).  We define the functor
\[
F^X_\K\colon \KKcat^X \to \Abel^X
\]
by mapping a \(\Cst\)\nb-algebra
over~\(X\)
to the diagram of \(\Z/2\)\nb-graded
Abelian groups \(\K_*(A(U_x))\)
for \(x\in X\)
with the maps induced by the inclusion maps \(A(U_x) \idealin A(U_y)\)
for \(x\to y\).
The target category~\(\Abel^X\)
is a paired, stable Abelian category, where
\(\Abel^X_\pm \subseteq \Abel^X\)
are the subcategories of \(\Z/2\)\nb-graded
groups where the odd or even part vanishes, respectively.
And~\(F^X_\K\)
is a stable homological functor.  Let \(\Ideal^X_\K\)
be its kernel on morphisms.  Equation~\eqref{eq:KKX_out_Fix2} implies
\[
\KK^X_0(i_z(\C),B)
\cong \KK_0\bigl(\C,B(U_z)\bigr)
\cong \K_0\bigl(B(U_z)\bigr)
\cong \Hom_{\Abel^X}\bigl(j_z(\C),F^X_\K(B)\bigr).
\]
Hence the partial adjoint~\((F^X_\K)^\lad\)
of~\(F^X_\K\)
is defined on \(j_z(\C)\)
for all \(z\in X\).
As in Section~\ref{sec:UCT_X}, any object of~\(\Abel^X\)
is a quotient of a direct sum of objects of the form \(j_z(\C)\)
for \(z\in X\).
Hence~\((F^X_\K)^\lad\)
is defined on all projective objects of~\(\Ideal^X_\K\)
and \(F^X_\K \circ (F^X_\K)^\lad(P)=P\)
for all projective objects~\(P\)
of~\(\Abel^X\);
this is proved like the corresponding statement about~\(F^\Z_\K\)
in Section~\ref{sec:Z-action}.  Therefore, the functor
\(F^X_\K\colon \KKcat^X \to \Abel^X\)
is the universal \(\Ideal^X_\K\)\nb-exact
stable homological functor into an exact category or into an Abelian
category by Theorem~\ref{the:characterise_universal_homological} and
Remark~\ref{rem:universal_to_exact}.  And~\(F^X_\K\)
restricts to an equivalence of categories between the
\(\Ideal^X_\K\)\nb-projective
objects in~\(\KKcat^X\)
and the projective objects in~\(\Abel^X\).
Let \(\Boot^X \subseteq\KKcat^X\)
be the localising subcategory generated by the
\(\Ideal^X_\K\)\nb-projective
objects.  This is the analogue of the bootstrap class in~\(\KKcat^X\).

\begin{lemma}
  \label{lem:Ext_2_over_X}
  Let \(G=(G_x,\gamma_{y,x})\)
  and \(H=(H_x,\eta_{y,x})\) be objects of~\(\Abel^X\).
  \begin{enumerate}
  \item There is a projective resolution for~\(G\)
    of length~\(2\) in~\(\Abel^X\).
  \item There is \(A\in\Boot^X\) with \(F^X_\K(A) \cong G\).
  \item Write \(\Ext\)
    for the \(\Ext^1\)
    of Abelian groups.  The group \(\Ext^2_{\Abel^X}(G,H)\)
    is naturally isomorphic to the cokernel of the map
    \begin{multline}
      \label{eq:cokernel_Ext_2_X}
      \prod_{x\in X} \Ext(G_x,H_x)
      \to \prod_{x\to y} \Ext(G_x,H_y),\\
      (t_x)_{x\in X} \mapsto
      (\eta_{y,x}\circ t_x - t_y\circ \gamma_{y,x})_{x\to y},
    \end{multline}
  \end{enumerate}
\end{lemma}

\begin{proof}
  Since \(\Ideal^X \subseteq \Ideal^X_\K\),
  any \(\Ideal^X\)\nb-exact
  chain complex in~\(\KKcat^X\)
  is also \(\Ideal^X_\K\)\nb-exact.
  In particular, the \(\Ideal^X\)\nb-projective
  resolution in~\eqref{eq:projective_resolution_ups} is
  \(\Ideal^X_\K\)\nb-exact.
  We claim that its entries have \(\Ideal^X_\K\)\nb-projective
  resolutions of length~\(1\).
  Hence there is a projective resolution of~\(F^X_\K(A)\)
  of length~\(2\)
  as in~\eqref{eq:resolution_A}.  To prove the same for all
  objects~\(G\)
  of~\(\Abel^X\), we carry over~\eqref{eq:projective_resolution_ups}.

  Let \(J_z(B)\inOb\Abel^\Z\)
  for a countable \(\Z/2\)\nb-graded
  Abelian group~\(B\)
  be the diagram with \(J_z(B)_x = B\)
  if \(z\succeq x\)
  and \(J_z(B)_x = 0\)
  otherwise, with the identity map \(J_z(B)_x \to J_z(B)_y\)
  for \(z\succeq x\to y\) and the zero map otherwise.  Then
  \begin{equation}
    \label{eq:Abel_X_jz_adjunction}
    \Hom_{\Abel^X}\bigl(J_z(B),H\bigr) \cong \Hom\bigl(B,H_z\bigr), 
  \end{equation}
  where the isomorphism simply restricts a morphism in~\(\Abel^X\)
  to the object \(z\in X\).
  In particular, if \(y\succeq z\),
  then \(J_y(B)_z = B\)
  and so there is a canonical map
  \(J_{y\succeq z}(B)\colon J_z(B) \to J_y(B)\)
  in~\(\Abel^X\).
  Explicitly, this map is the identity map on \(J_z(B)_x\)
  if \(z\succeq x\)
  and the zero map on \(J_z(B)_x=0\)
  if \(z\nsucceq x\).
  The proof that~\eqref{eq:projective_resolution_ups} is exact also
  proves the exactness of
  \begin{equation}
    \label{eq:projective_resolution_ups_Abel}
    0 \to
    \bigoplus_{x\to y} J_y(G_x)
    \xrightarrow{\psi} \bigoplus_{x\in X} J_x(G_x)
    \xrightarrow{q} G
    \to 0;
  \end{equation}
  here~\(\psi\)
  restricted to the summand \(J_y(G_x)\)
  for \(x\to y\)
  is the map \((J_{x\to y}(G_x),J_y(\gamma_{y,x}))\)
  to \(J_x(G_x) \oplus J_y(G_y)\),
  and~\(q\)
  maps \(J_x(G_x)\)
  to~\(G\)
  by the adjunct of the identity map on~\(G_x\)
  under the adjunction in~\eqref{eq:Abel_X_jz_adjunction}.
  Explicitly, the entry of \(\bigoplus_{x\in X} J_x(G_x)\)
  at \(z\in X\)
  is \(\bigoplus_{x\succeq z} G_x\),
  which is mapped to~\(G_z\)
  by \((\gamma_{z,x})_{x\succeq z}\).
  The proof that~\eqref{eq:projective_resolution_ups_Abel} is exact
  shows that the chain complexes of \(\Z/2\)\nb-graded Abelian groups
  \[
  0 \to
  \bigoplus_{x\to y} J_y(G_x)_z
  \xrightarrow{\psi_z} \bigoplus_{x\in X} J_x(G_x)_z
  \xrightarrow{q_z} G_z
  \to 0
  \]
  are contractible for all \(z\in X\).

  For each \(x\in X\), there is a resolution
  \begin{equation}
    \label{eq:resolve_Gx}
    P_{x,1} \xinto{d_{x,1}} P_{x,0} \xprto{d_{x,0}} G_x    
  \end{equation}
  of~\(G_x\)
  of length~\(1\)
  by countable free \(\Z/2\)\nb-graded
  Abelian groups.  The homomorphism \(\gamma_{y,x}\colon G_x \to G_y\)
  for \(x \to y\) lifts to a morphism of extensions
  \begin{equation}
    \label{eq:resolve_Gx_Gy}
    \begin{tikzcd}
      P_{x,1} \arrow[r, rightarrowtail, "d_{x,1}"] \arrow[d, "\gamma^1_{y,x}"]&
      P_{x,0} \arrow[r, twoheadrightarrow, "d_{x,0}"] \arrow[d, "\gamma^0_{y,x}"]&
      G_x \arrow[d, "\gamma_{y,x}"]\\
      P_{y,1} \arrow[r, rightarrowtail, "d_{y,1}"] &
      P_{y,0} \arrow[r, twoheadrightarrow, "d_{y,0}"] &
      G_y
    \end{tikzcd}
  \end{equation}
  The construction~\(J_z\)
  above is an exact functor, and it maps free Abelian groups to
  projective objects of~\(\Abel^X\)
  by the adjunction~\eqref{eq:Abel_X_jz_adjunction}.  Hence
  \begin{align*}
    \bigoplus_{x\to y} J_y(P_{x,1})
    \into \bigoplus_{x\to y} J_y(P_{x,0})
    \prto \bigoplus_{x\to y} J_y(G_x),\\
    \bigoplus_{x\in X} J_x(P_{x,1})
    \into \bigoplus_{x\in X} J_x(P_{x,0})
    \prto \bigoplus_{x\in X} J_sx(G_x)
  \end{align*}
  are projective resolutions of length~\(1\)
  in~\(\Abel^X\).
  The resolution in~\eqref{eq:projective_resolution_ups_Abel}
  gives a projective resolution of~\(G\)
  of length~\(2\)
  as in~\eqref{eq:resolution_A}.  This proves the first assertion.

  Now \cite{Bentmann-Meyer:More_general}*{Lemma 2.4} shows that there
  is \(A\inOb\Boot^X\)
  with \(G=F^X_\K(A)\);
  the property that \(\KK^X_0(A,B)=0\)
  for all \(\Ideal^X_\K\)\nb-contractible
  \(B\inOb\KKcat^X\) is equivalent to \(A\inOb\Boot^X\).

  The projective resolution of~\(G\) built above has the form
  \[
  0 \to
  \bigoplus_{x\to y} J_y(P_{x,1})
  \to \bigoplus_{x\to y} J_y(P_{x,0}) \oplus \bigoplus_{x\in X} J_x(P_{x,1})
  \to \bigoplus_{x\in X} J_x(P_{x,0})
  \to G,
  \]
  where the maps
  \(\bigoplus_{x\to y} J_y(P_{x,i}) \to \bigoplus_{x\in X}
  J_x(P_{x,i})\)
  for \(i=0,1\)
  use \(\gamma_{y,x}^i \colon P_{x,i} \to P_{y,i}\)
  in~\eqref{eq:resolve_Gx_Gy}.  We use this resolution to compute the
  group \(\Ext^2_{\Abel^X}(G,H)\).
  An element of \(\Ext^2_{\Abel^X}(G,H)\)
  is represented by a map \(\bigoplus_{x\to y} J_y(P_{x,1}) \to H\).
  By the adjunction~\eqref{eq:Abel_X_jz_adjunction}, this corresponds
  to a family of maps \(f_{x\to y}\colon P_{x,1} \to H_y\).
  This family represents~\(0\)
  in \(\Ext^2_{\Abel^X}(G,H)\)
  if and only if the corresponding map
  \(\bigoplus_{x\to y} J_y(P_{x,1}) \to H\)
  factors through
  \(\bigoplus_{x\to y} J_y(P_{x,0}) \oplus \bigoplus_{x}
  J_x(P_{x,1})\).
  Using the adjunction~\eqref{eq:Abel_X_jz_adjunction} again, a map on
  this direct sum corresponds to families of maps
  \(g_{x\to y}\colon P_{x,0} \to H_y\)
  and \(h_x\colon P_{x,1} \to H_x\).
  The resulting map \(\bigoplus_{x\to y} J_y(P_{x,1}) \to H\)
  corresponds to the family of maps
  \[
  -g_{x\to y} \circ d_{x,1} + \eta_{y,x} h_x - h_y \gamma_{y,x}^1\colon
  P_{x,1} \to H_y.
  \]
  The resolutions~\eqref{eq:resolve_Gx} compute \(\Ext(G_x,D)\)
  for any \(\Z/2\)\nb-graded
  Abelian group~\(D\).
  So each \(f_{x\to y}\colon P_{x,1} \to H_y\)
  represents an element of \(\Ext(G_x,H_y)\),
  and it represents the zero element if and only if it is of the form
  \(g_{x\to y} \circ d_{x,1}\)
  for some \(g_{x\to y}\colon P_{x,0} \to H_y\).
  The elements~\(h_x\)
  above represent elements of \(\Ext(G_x,H_x)\).
  If they represent the zero element of \(\Ext(G_x,H_x)\),
  then the term \(\eta_{y,x} h_x - h_y \gamma_{y,x}^1\)
  above may be rewritten in the form \(-g_{x\to y} \circ d_{x,1}\).
  Now we get the formula for \(\Ext^2_{\Abel^X}(G,H)\)
  in the third statement in the lemma.
\end{proof}

\begin{theorem}
  \label{the:classification_X_obstruction}
  Let \(A\)
  and~\(B\)
  belong to~\(\Boot^X\).
  An isomorphism \(t\colon F^X_\K(A) \congto F^X_\K(B)\)
  in~\(\Abel^X\)
  lifts to an invertible element in \(\KK^X_0(A,B)\)
  if and only if \(t\delta_A = \delta_B t\)
  holds in \(\Ext^2_{\Abel^X}\bigl(\Sigma F^X_\K(A),F^X_\K(B)\bigr)\),
  where \(\delta_A\)
  and~\(\delta_B\) are the obstruction classes of \(A\) and~\(B\).
\end{theorem}

\begin{proof}
  The lemma verifies all the conditions to apply the classification
  method of~\cite{Bentmann-Meyer:More_general}.
\end{proof}

The proof of the lemma also gives all the ingredients needed in
Section~\ref{sec:triangulated}.  So we may now compute obstruction
classes:

\begin{theorem}
  \label{the:obstruction_class_X}
  Let~\(B\)
  be an object of~\(\Boot^X\).
  Let \(B_x \defeq B(U_x)\)
  for \(x\in X\).
  Let \(\beta_{y,x} \in \KK_0(B_x,B_y)\)
  be the \(\KK\)\nb-class
  of the inclusion map \(B_x \hookrightarrow B_y\).
  Split \(\beta_{y,x} = \beta_{y,x}^0 + \beta_{y,x}^1\)
  with a parity-preserving part
  \(\beta_{y,x}^0 \in \Hom\bigl(\K_*(B_x),\K_*(B_y)\bigr)\)
  and a parity-reversing part
  \(\beta_{y,x}^1 \in \Ext\bigl(\K_{1+*}(B_x),\K_*(B_y)\bigr)\).
  The obstruction class of~\(B\)
  is the class in the cokernel of the map
  in~\eqref{eq:cokernel_Ext_2_X} that is represented by
  \[
  (\beta_{y,x}^1)_{x\to y} \in \prod_{x\to y} \Ext\bigl(\K_{1+*}(B_x),\K_*(B_y)\bigr).
  \]
\end{theorem}

\begin{proof}
  The projective resolution~\eqref{eq:projective_resolution_ups}
  in~\(\KKcat[X]\) lifts to an exact triangle
  \[
  \bigoplus_{x\to y} i_y(B_x) \xrightarrow{\varphi}
  \bigoplus_{x\in X} i_x(B_x) \xrightarrow{p}
  B \to \Sigma \bigoplus_{x\to y} i_y(B_x)
  \]
  in \(\KKcat^X\)
  by \cite{Bentmann-Meyer:More_general}*{Proposition 2.3}.  This is
  how the Universal Coefficient Theorem in
  \cite{Meyer-Nest:Homology_in_KK}*{Theorem~66} is proved.  Here the
  map~\(\varphi\)
  restricted to the summand~\(i_y(B_x)\)
  is the difference of two maps: the map \(i_y(\beta_{y,x})\)
  to~\(i_y(B_y)\)
  and the canonical map \(i_{x,y}(B_x)\colon i_y(B_x) \to i_x(B_x)\)
  that is the adjunct of the identity map \(B_x \to i_x(B_x)_y = B_x\)
  under the adjunction~\eqref{eq:KK_out_ix}.  And the map~\(p\)
  restricted to~\(i_x(B_x)\)
  is the adjunct of the identity map \(B_x \to B_x\)
  under the adjunction~\eqref{eq:KK_out_ix}.

  Both \(\bigoplus_{x\to y} i_y(B_x)\)
  and \(\bigoplus_{x\in X} i_x(B_x)\)
  belong to~\(\Boot^X\)
  and have \(\Ideal^X_\K\)\nb-projective
  resolutions of length~\(1\)
  (see~\eqref{eq:resolve_Gx}).  So we are in the situation
  of~\eqref{eq:triangle_B1B0A}.  We split each \(B_x \defeq B(U_x)\)
  into its even and odd parity part \(B_x = B_x^+ \oplus B_x^-\).
  Then \(i_y(B_x^+)\)
  and \(i_y(B_x^-)\)
  are of even or odd parity, respectively.  Split
  \(\beta_{y,x} = \beta_{y,x}^0 + \beta_{y,x}^1\)
  into a parity-preserving and a parity-reversing part.  So
  \(\beta_{y,x}^0 \in \Hom\bigl(\K_*(B_x),\K_*(B_y)\bigr)\)
  and \(\beta_{y,x}^1 \in \Ext\bigl(\K_{1+*}(B_x),\K_*(B_y)\bigr)\)
  by the Universal Coefficient Theorem for \(\KK\),
  see the discussion after Theorem~\ref{the:Z_classify}.  The induced
  maps \(i_y(\beta_{y,x}^0)\)
  and \(i_y(\beta_{y,x}^1)\)
  are parity-preserving and parity-reversing, respectively.  And the
  map \(i_{x,y}(B_x)\)
  is parity-preserving.  So the parity-reversing part~\(\varphi^1\)
  of~\(\varphi\)
  is the map that restricts to
  \(i_y(\beta_{y,x}^1)\colon i_y(B_x) \to i_y(B_y)\)
  on the summand for \(x\to y\).
  The composite \(p\circ \varphi^1\)
  is the map \(\bigoplus_{x\to y} i_y(B_x) \to B\)
  whose restriction to the summand \(i_y(B_x)\)
  is adjunct to \(\beta_{y,x}^1\colon B_x \to B_y\).
  These maps define an element of
  \(\prod_{x\to y} \Ext\bigl(\K_{1+*}(B_x),\K_*(B_y)\bigr)\).
  The obstruction class of~\(B\)
  is its image under the boundary map to
  \(\Ext^2_{\Abel^X}\bigl( \Sigma F^X_\K(B),F^X_\K(B)\bigr)\)
  by Theorem~\ref{the:compute_obstruction}.  We have described
  \(\Ext^2_{\Abel^X}(G,H)\)
  in Lemma~\ref{lem:Ext_2_over_X} in such a way that this boundary map
  becomes tautological: it simply maps an element of
  \(\prod_{x\to y} \Ext(G_x,H_y)\)
  to its class in the cokernel of the map
  in~\eqref{eq:cokernel_Ext_2_X}.  In particular, the obstruction
  class of~\(B\)
  is represented by
  \((\beta_{y,x}^1)_{x\to y} \in \prod_{x\to y}
  \Ext\bigl(\K_{1+*}(B_x),\K_*(B_y)\bigr)\).
\end{proof}

Theorem~\ref{the:obstruction_class_X} allows us to compare the
classification theorems for \(\Cst\)\nb-algebras
in~\(\Boot^X\)
that use the invariant \(F^X(B)\)
or \(F^X_\K(B)\)
with the obstruction class.  Let \(A,B\inOb\Boot^X\).
An isomorphism
\[
t^0\colon F^X_\K(A) \congto F^X_\K(B)
\]
is equivalent to a family of isomorphisms
\[
t^0_x\colon \K_*(A(U_x)) \congto \K_*(B(U_x))
\]
that make the diagrams
\[
\begin{tikzcd}
  \K_*(A(U_x)) \arrow[d, "\K_*(\alpha_{y,x})"]
  \arrow[r, "t^0_x", "\cong"'] &
  \K_*(B(U_x)) \arrow[d, "\K_*(\beta_{y,x})"] \\
  \K_*(A(U_y)) \arrow[r, "t^0_y", "\cong"'] &
  \K_*(B(U_y))
\end{tikzcd}
\]
commute for all edges \(x\to y\).
Here the maps \(\K_*(\alpha_{y,x})\)
and \(\K_*(\beta_{y,x})\)
are induced by the inclusion maps of our \(\Cst\)\nb-algebras
over~\(X\).
The obstruction class for the isomorphism~\((t^0_x)_{x\in X}\)
vanishes if and only if there are
\(t^1_x\in \Ext\bigl(\K_{1+*}(A(U_x)),\K_*(B(U_x))\bigr)\)
for \(x\in X\) such that
\begin{equation}
  \label{eq:t_condition_X}
  t^1_y \circ \K_*(\alpha_{y,x}) - \K_*(\beta_{y,x}) \circ t^1_x
  = \beta_{y,x}^1 \circ t^0_x - t^0_y \circ \alpha_{y,x}^1
\end{equation}
for all edges \(x\to y\).
Here
\[
\alpha_{y,x}^1 \in \Ext\bigl(\K_{1+*}(A(U_x)),\K_*(A(U_y))\bigr),\qquad
\beta_{y,x}^1\in \Ext\bigl(\K_{1+*}(B(U_x)),\K_*(B(U_y))\bigr)
\]
are the parity-reversing parts of the \(\KK\)\nb-classes
of the inclusion maps.  The elements \(t^0_x\)
and~\(t^1_x\)
together form \(t_x\in \KK_0(A(U_x),B(U_x))\).
And~\eqref{eq:t_condition_X} is equivalent to
\(t_y\circ \alpha_{y,x} = \beta_{y,x} \circ t_x\)
in \(\KK_0(A(U_x),B(U_y))\).

Corollary~\ref{cor:classify_KKX} says that any family of
\(\KK\)-equivalences
\(t_x\in \KK_0(A(U_x),B(U_x))\)
with \(t_y \circ \alpha_{y,x} = \beta_{y,x} \circ t_x\)
for all edges \(x\to y\)
lifts to an invertible element in \(\KK^X_0(A,B)\).
The classification using the obstruction class says that a family of
isomorphisms \(t^0_x\colon \K_*(A(U_x)) \congto \K_*(B(U_x))\)
lifts to an invertible element in \(\KK^X_0(A,B)\)
if and only if there exist~\(t^1_x\)
satisfying~\eqref{eq:t_condition_X}.  Corollary~\ref{cor:classify_KKX}
makes a slightly stronger assertion because it says that any choice of
elements~\(t^1_x\)
satisfying~\eqref{eq:t_condition_X} may be realised by an invertible
element in \(\KK^X_0(A,B)\).

The result in Corollary~\ref{cor:classify_KKX} about the existence of
liftings is also slightly stronger than the corresponding result using
the invariant~\(F^X_\K\)
and the obstruction class because it does not require the
\(\Cst\)\nb-algebras~\(A(U_x)\) to belong to the bootstrap class.

\section{Filtrated K-theory for totally ordered spaces}
\label{sec:filtrated_K}

Now we consider the special unique path space
\[
X = \{1 \leftarrow 2 \leftarrow \dotsb \leftarrow n\}
\]
for some \(n\in\N_{\ge2}\).
We are going to compare the classification for \(\Cst\)\nb-algebras
over~\(X\)
that follows from the Universal Coefficient Theorem
in~\cite{Meyer-Nest:Filtrated_K} to the classifications in
Section~\ref{sec:compare_ups}.  The invariant used
in~\cite{Meyer-Nest:Filtrated_K} is filtrated \(\K\)\nb-theory.
This is the diagram of \(\K\)\nb-theory
groups formed by \(\K_*(A(S))\)
for all locally closed subsets \(S\subseteq X\).
Here a subset is locally closed if and only if it is of the form
\[
[a,b] \defeq \{a,a+1,\dotsc,b-1,b\},\qquad
1\le a \le b \le n.
\]
The maps in the filtrated \(\K\)\nb-theory
diagram are those that come from natural transformations.  We are
going to describe these below.

Let~\(\Ideal^X_\mathrm{fil}\)
be the kernel on morphisms of the filtrated \(\K\)\nb-theory
functor, that is, \(f\in\KK^X_0(A,B)\)
belongs to \(\Ideal^X_\mathrm{fil}(A,B)\)
if and only if it induces the zero map
\(\K_*(A[a,b]) \to \K_*(B[a,b])\)
for all \(1 \le a \le b \le n\).
This is a stable homological ideal.  It is shown
in~\cite{Meyer-Nest:Homology_in_KK} that any object \(A\inOb\Boot^X\)
has an \(\Ideal^X_\mathrm{fil}\)\nb-projective resolution
\begin{equation}
  \label{eq:fil_K_resolution}
  0 \to P_1 \xrightarrow{\varphi} P_0 \xrightarrow{f} A
\end{equation}
of length~\(1\).
Then it follows that~\(A\)
is isomorphic to the cone of~\(\varphi\).
And there is a Universal Coefficient Theorem for objects
of~\(\Boot^X\)
based on Hom and Ext groups between their filtrated \(\K\)\nb-theory
diagrams.  Thus any isomorphism between the filtrated \(\K\)\nb-theory
diagrams of \(A,B\inOb\Boot^X\)
lifts to an equivalence in \(\KK^X_0(A,B)\).

The minimal open subset~\(U_x\)
containing~\(x\)
is \([x,n]\)
for each \(x\in X\).
Since \(\K_*(A(U_x))\)
is part of the filtrated \(\K\)\nb-theory
of~\(A\),
the ideal~\(\Ideal^X_\mathrm{fil}\)
is contained in the ideal~\(\Ideal^X_\K\)
that is used above for a general unique path space.  So any
\(\Ideal^X_\mathrm{fil}\)\nb-exact
chain complex is also \(\Ideal^X_\K\)\nb-exact.
If \(P\inOb\Boot^X\)
is \(\Ideal^X_\mathrm{fil}\)\nb-projective,
then the Abelian groups \(\K_*(P(S))\)
are free for all locally closed subsets \(S\subseteq X\)
(this follows from \cite{Meyer-Nest:Filtrated_K}*{Theorem~3.12} and
will become manifest below).  The computation of \(\Ext^2_{\Abel^X}\)
in Lemma~\ref{lem:Ext_2_over_X} shows that
\(\Ext^2_{\Abel^X}(F^X_\K(P),D)=0\)
for all \(D\inOb\Abel^X\)
if \(P(U_x)\)
is free for each \(x\in X\).
Hence \(P_1\)
and~\(P_0\)
have \(\Ideal^X_\K\)\nb-projective
resolutions of length~\(1\).
So we are in the situation of~\eqref{eq:triangle_B1B0A} and
Theorem~\ref{the:compute_obstruction} computes the obstruction class
from the parity-reversing part of the map~\(\varphi\)
in~\eqref{eq:fil_K_resolution}.  This computation is, however, quite
non-trivial.  We must first recall how the natural transformations in
the filtrated \(\K\)\nb-theory
diagram look like.  Going beyond the results
of~\cite{Meyer-Nest:Filtrated_K}, we then build an explicit
\(\Ideal^X_\mathrm{fil}\)\nb-projective
resolution.  Next, we observe which parts of the map~\(\varphi\)
are parity-reversing.  This gives a class in \(\Ext^2_{\Abel^X}\).
To translate it into the setting of
Theorem~\ref{the:obstruction_class_X}, we still have to compare the
resolution used there with the one coming
from~\eqref{eq:fil_K_resolution}.  This requires most of the work.


We first recall the description of the \(\Z/2\)\nb-graded
Abelian groups \(\mathcal{NT}_*([a,b],[c,d])\)
of natural transformations \(\K_*(A([a,b])) \to \K_*(A([c,d]))\)
in~\cite{Meyer-Nest:Filtrated_K}.
Let \(1\le a \le b \le n\).
The functor \(\KKcat^X\to\Ab^{\Z/2}\),
\(A\mapsto \K_*(A([a,b]))\),
is represented by an object~\(\mathcal{R}_{[a,b]}\)
of~\(\KKcat^X\),
which is described in~\cite{Meyer-Nest:Filtrated_K}.  Let
\(1\le a\le b \le n\)
and \(1\le c\le d \le n\).
By the Yoneda Lemma,
\[
\mathcal{NT}_*([a,b],[c,d])
\cong \KK^X_*(\mathcal{R}_{[c,d]},\mathcal{R}_{[a,b]})
\cong \K_*(\mathcal{R}_{[a,b]}([c,d])).
\]
These groups are computed in
\cite{Meyer-Nest:Filtrated_K}*{Equation~(3.1)}:
\begin{equation}
  \label{eq:NT_explicit}
  \mathcal{NT}_*([a,b],[c,d]) \cong
  \begin{cases}
    \Z_+&\text{if }c\le a \le d \le b,\\
    \Z_-&\text{if }a+1 \le c \le b+1 \le d,\\
    0&\text{otherwise.}
  \end{cases}
\end{equation}
We write \([a,b] \to [c,d]\)
in the first two cases, that is, when there is a non-zero natural
transformation in \(\mathcal{NT}_*([a,b],[c,d])\).
The groups \(\mathcal{NT}_*([a,b],[c,d])\)
form a \(\Z/2\)\nb-graded
ring~\(\mathcal{NT}\).
The filtrated \(\K\)\nb-theory
of a separable \(\Cst\)\nb-algebra
over~\(X\)
is a \(\Z/2\)\nb-graded,
countable module over it, which we denote by \(FK(A)\).

The computations in~\cite{Meyer-Nest:Filtrated_K} interpret the
elements of \(\mathcal{NT}_*([a,b],[c,d])\)
as follows.  Let~\(A\)
be a separable \(\Cst\)\nb-algebra
over~\(X\)
and let \(M \defeq FK(A)\).
If \(a< b \le c\),
then \([b,c]\)
is relatively open in \([a,c]\)
with complement \([a,b-1]\).
This induces the following natural six-term exact sequence in
\(\K\)\nb-theory:
\begin{equation}
  \label{eq:exact_module_NT}
  \dotsb \to M[b,c] \xrightarrow{i} M[a,c] \xrightarrow{r} M[a,b-1]
  \xrightarrow[\mathrm{odd}]{\delta} M[b,c] \to \dotsb
\end{equation}
where the maps \(i,r\)
preserve the \(\Z/2\)\nb-grading
and~\(\delta\)
reverses it.  Any natural transformation \(M[a,b] \to M[c,d]\)
is an integer multiple of a product of the maps \(i,r,\delta\)
above.  More precisely, if \(c\le a \le d \le b\),
then there is a commuting square
\begin{equation}
  \label{def:tautr_even}
  \begin{tikzcd}
    M[a,b] \arrow[r,"i"] \arrow[d, "r"] \arrow[dr, "\tautr{a,b}{c,d}" description] &
    M[c,b] \arrow[d, "r"] \\
    M[a,d] \arrow[r, "i"'] & M[c,d],
  \end{tikzcd}
\end{equation}
and its diagonal map~\(\tautr{a,b}{c,d}\)
generates \(\mathcal{NT}_0([a,b],[c,d]) \cong\Z\).
And if \(a+1 \le c \le b+1 \le d\), then there is a commuting square
\begin{equation}
  \label{def:tautr_odd}
  \begin{tikzcd}
    M[a,b] \arrow[r,"\delta"] \arrow[d, "r"] \arrow[dr, "\tautr{a,b}{c,d}" description] &
    M[b+1,d] \arrow[d, "i"] \\
    M[a,c-1] \arrow[r, "\delta"'] & M[c,d],
  \end{tikzcd}
\end{equation}
and its diagonal map~\(\tautr{a,b}{c,d}\)
generates \(\mathcal{NT}_1([a,b]),[c,d])\cong\Z\).
We have defined a generator~\(\tau_{[a,b]}^{[c,d]}\)
for \(\mathcal{NT}_*([a,b],[c,d])\)
whenever \([a,b] \to [c,d]\),
that is, whenever \(\mathcal{NT}_*([a,b]),[c,d])\neq0\)
by~\eqref{eq:NT_explicit}.  It is convenient to define
\(\tau_{[a,b]}^{[c,d]}=0\)
if \(\mathcal{NT}_*([a,b]),[c,d])=0\).
By the Yoneda Lemma, the natural transformations~\(\tautr{a,b}{c,d}\)
correspond to arrows
\[
\bigl(\tautr{a,b}{c,d}\bigr)^*\colon
\mathcal{R}_{[c,d]} \to \mathcal{R}_{[a,b]}.
\]

\begin{remark}
  \label{rem:exact_NT}
  An \(\mathcal{NT}\)\nb-module
  is called \emph{exact} if the sequences~\eqref{eq:exact_module_NT}
  are exact for all \(a< b \le c\).
  The exact \(\mathcal{NT}\)\nb-modules
  form a stable exact subcategory of the stable Abelian category of
  all \(\mathcal{NT}\)\nb-modules,
  and the filtrated \(\K\)\nb-theory
  of any separable \(\Cst\)\nb-algebra
  over~\(X\)
  is exact as an \(\mathcal{NT}\)\nb-module.
  The results in~\cite{Meyer-Nest:Filtrated_K} imply that any exact
  \(\mathcal{NT}\)\nb-module
  has a projective resolution of length~\(1\).
  Hence it lifts to an object of the bootstrap class~\(\Boot^X\).
  So the image of the filtrated \(\K\)\nb-theory
  functor is equal to the class of exact, countable
  \(\mathcal{NT}\)\nb-modules.
  And the filtrated \(\K\)\nb-theory
  functor, viewed as a functor to the subcategory of exact, countable
  \(\mathcal{NT}\)\nb-modules,
  is the universal \(\Ideal^X_\mathrm{fil}\)\nb-exact
  functor to an exact category.  So~\(\Ideal^X_\mathrm{fil}\)
  has the property that its universal exact functors to an Abelian and
  to an exact category are different.
\end{remark}

Now we study the multiplication in~\(\mathcal{NT}\).
We begin with decomposing the generators in~\(\mathcal{NT}\)
further.  We may rewrite the natural
transformations~\(\tautr{a,b}{c,d}\)
defined above as products of the special natural transformations
\begin{alignat*}{2}
  i = \tautr{a+1,b}{a,b}&\colon [a+1,b] \to [a,b],&\qquad&a+1\le b \le n,\\
  r = \tautr{a,b+1}{a,b}&\colon [a,b+1] \to [a,b],&\qquad&a\le b \le n-1,\\
\delta = \tautr{1,a-1}{a,n}&\colon [1,a-1] \to [a,n],&\qquad&2 \le a \le n.
\end{alignat*}
This is clear in the even case.  In the odd case, we use the
naturality of boundary maps to rewrite the boundary map
\(\delta\colon [a,b] \to [b+1,d]\)
as the product of \(i\colon [a,b] \to [1,b]\),
the boundary map \(\delta\colon [1,b] \to [b+1,n]\),
and \(r\colon [b+1,n] \to [b+1,d]\).
The generating natural transformations defined above form a commuting
diagram as in Figure~\ref{fig:filtrated_K}.%
\begin{figure}
  \caption{Natural transformations on filtrated \(\K\)\nb-theory for general~\(n\)}
  \label{fig:filtrated_K}
  \rotatebox{90}{\begin{tikzcd}[column sep=small,ampersand replacement=\&]
      M[n,n] \arrow[r, "i"] \arrow[d] \&
      M[n-1,n] \arrow[r, "i"] \arrow[d, "r"] \&
      M[n-2,n] \arrow[r, "i"] \arrow[d, "r"] \&
      \dotsb  \arrow[r, "i"] \arrow[d, "r"] \&
      M[2,n] \arrow[r, "i"] \arrow[d, "r"] \&
      M[1,n] \arrow[r] \arrow[d, "r"] \&
      0 \arrow[d] \\
      0 \arrow[r] \&
      M[n-1,n-1] \arrow[r, "i"] \arrow[d] \&
      M[n-2,n-1] \arrow[r, "i"] \arrow[d, "r"] \&
      \dotsb  \arrow[r, "i"] \arrow[d, "r"] \&
      M[2,n-1] \arrow[r, "i"] \arrow[d, "r"] \&
      M[1,n-1] \arrow[r, "\delta"] \arrow[d, "r"] \&
      M[n,n] \arrow[d,"i"] \\
      \& 0 \arrow[r] \&
      M[n-2,n-2] \arrow[r, "i"] \arrow[d] \&
      \dotsb  \arrow[r, "i"] \arrow[d, "r"] \&
      M[2,n-2] \arrow[r, "i"] \arrow[d, "r"] \&
      M[1,n-2] \arrow[r, "\delta"] \arrow[d, "r"] \&
      M[n-1,n] \arrow[d,"i"] \\
      \&\&0\arrow[r]\&\vdots\arrow[r]\&
      \vdots\arrow[r] \arrow[d,"r"]\&
      \vdots\arrow[r] \arrow[d,"r"]\&
      \vdots \arrow[d,"i"]\\
      \&\&\&0\arrow[r]\&
      M[2,2] \arrow[r, "i"] \arrow[d] \&
      M[1,2] \arrow[r, "\delta"] \arrow[d, "r"] \&
      M[3,n] \arrow[d,"i"] \\
      \&\&\&\&0\arrow[r]\&
      M[1,1] \arrow[r, "\delta"] \arrow[d] \&
      M[2,n] \arrow[d,"i"] \\
      \&\&\&\&\&0\arrow[r]\&
      M[1,n]
    \end{tikzcd}}
\end{figure}%
\begin{figure}
  \caption{Complete diagram of natural transformations on filtrated \(\K\)\nb-theory for \(n=3\)}
  \label{fig:filtrated_K_3}
  \rotatebox{90}{\begin{tikzcd}[column sep=small,ampersand replacement=\&]
      M[3,3] \arrow[r, "i"] \arrow[d] \&
      M[2,3] \arrow[r, "i"] \arrow[d, "r"] \&
      M[1,3] \arrow[r] \arrow[d, "r"] \&
      0 \arrow[d] \\
      0 \arrow[r] \&
      M[2,2] \arrow[r, "i"] \arrow[d] \&
      M[1,2] \arrow[r, "\delta"] \arrow[d, "r"] \&
      M[3,3] \arrow[d,"i"] \arrow[r] \&
      0 \arrow[d] \\
      \& 0 \arrow[r] \&
      M[1,1] \arrow[r, "\delta"] \arrow[d] \&
      M[2,3] \arrow[d,"i"] \arrow[r, "r"] \&
      M[2,2] \arrow[d,"i"] \arrow[r] \& 0\arrow[d]\\
      \&\&0\arrow[r]\&
      M[1,3] \arrow[r, "r"] \&
      M[1,2] \arrow[r, "r"] \&
      M[1,1] \\
    \end{tikzcd}}
\end{figure}
The last column and the first row in the diagram are the same.  So the
diagram repeats when we put a reflected copy of it next to it.
Figure~\ref{fig:filtrated_K_3} shows the full diagram for \(n=3\).
We claim that all relations among the generating natural
transformations are given by this extended commuting diagram.  In
particular, a composite of \(i,r,\delta\)
vanishes if and only if it factors through one of the objects~\(0\)
on the boundary of the extended diagram.

Let \(1 \le a \le b \le n\)
and \(1 \le c \le d \le n\).
A product of the generators of type \(i,r\)
from~\([a,b]\)
to~\([c,d]\)
exists if and only if \(c \le a\)
and \(d \le b\).
Figure~\ref{fig:filtrated_K_3} shows that all such products are equal.
Equation~\eqref{eq:NT_explicit} shows that this product is~\(0\)
unless \(a \le d\),
so that \(c \le a \le d \le b\).
Then it is equal to \(\tautr{a,b}{c,d}\)
by the definition in~\eqref{def:tautr_even}.  As a consequence,
\(\tautr{e,g}{c,d} \cdot \tautr{a,b}{e,g} = \tautr{a,b}{c,d}\)
if \(c \le e \le d \le g\)
and \(e \le a \le g \le b\);
this is non-zero if and only if also \(a \le d\)
or, equivalently, \(c \le e \le a \le d \le g \le b\).

Now consider a product of \(i,r,\delta\)
going from~\([a,b]\)
to~\([c,d]\)
and containing exactly one factor of~\(\delta\).
Using the diagram in Figure~\ref{fig:filtrated_K}, we may rearrange
this product in such a way that we first go right and then go down in
the extended diagram as in Figure~\ref{fig:filtrated_K_3}.  If this
goes through the zeros outside the drawn region, the product is~\(0\).
If not, we may combine consecutive~\(i\)
and consecutive~\(r\) to bring the product into the following form:
\[
[a,b] \xrightarrow{i}
[1,b] \xrightarrow{\delta}
[b+1,n] \xrightarrow{r}
[b+1,d] \xrightarrow{i}
[c,d].
\]
The combination \(r\circ\delta\circ i\)
in the beginning is the boundary map
\(\delta\colon [a,b] \to [b+1,d]\).
So we get~\(\tautr{a,b}{c,d}\)
if \(c+1 \le a \le d+1 \le b\)
and~\(0\)
otherwise by \eqref{def:tautr_odd} and~\eqref{eq:NT_explicit}.  Since
we may rewrite all even~\(\tautr{a,b}{c,d}\)
in terms of \(i,r\),
we can now compute \(\tautr{e,g}{c,d}\cdot \tautr{a,b}{e,g}\)
if one of the transformations \(\tautr{e,g}{c,d}\)
and \(\tautr{a,b}{e,g}\)
is even and the other one is odd.  Namely, the product is
\(\tautr{a,b}{c,d}\)
if \(a+1 \le c \le b+1 \le d\),
and~\(0\)
otherwise.  In more detail, the assumption that exactly one of the
transformations \(\tautr{e,g}{c,d}\)
and \(\tautr{a,b}{e,g}\)
is even means that \(e+1 \le c \le g+1 \le d\)
and \(e \le a \le g \le b\),
or \(c \le e \le d \le g\)
and \(a+1 \le e \le b+1 \le g\).
The assumption \(a+1 \le c \le b+1 \le d\)
becomes \(e \le a < c \le g+1 \le b+1 \le d\)
or \(a< c \le e \le b+1 \le d \le g\)
in these two cases, respectively.

Finally, any product with more than two factors~\(\delta\)
vanishes because it may be deformed in the extended diagram in
Figure~\ref{fig:filtrated_K_3} so as to factor through one of the
zeros on the boundary.  We sum up our results about the multiplication
in~\(\mathcal{NT}\):
\begin{equation}
  \label{eq:compose_tau}
  \tautr{e,g}{c,d} \cdot \tautr{a,b}{e,g} = \tautr{a,b}{c,d} \neq0
  \iff
  \left\{
    \begin{array}{c}
      c \le e \le a \le d \le g \le b,\\
      e \le a \le c-1 \le g \le b < d,\\
      a< c \le e \le b+1 \le d \le g,
    \end{array}
  \right.
\end{equation}
and \(\tautr{e,g}{c,d} \cdot \tautr{a,b}{e,g}=0\)
otherwise.  We write
\[
[a,b] \to [e,g] \to [c,d] \iff
\tautr{e,g}{c,d} \cdot \tautr{a,b}{e,g} = \tautr{a,b}{c,d} \neq0.
\]
It can happen that \([a,b] \to [e,g]\),
\([e,g] \to [c,d]\)
and \([a,b] \to [c,d]\),
but not \([a,b] \to [e,g] \to [c,d]\);
that is, \(\tautr{e,g}{c,d} \cdot \tautr{a,b}{e,g} =0\)
although \(\tautr{a,b}{c,d},\tautr{a,b}{e,g}, \tautr{e,g}{c,d}\neq0\).

Given \(1\le a \le b \le n\),
there are one or two proper natural transformations to \(M[a,b]\)
that are shortest in the sense that all others factor through them.
If \(a<b<n\),
these are \(i\colon M[a+1,b] \to M[a,b]\)
and \(r\colon M[a,b+1]\to M[a,b]\).
If \(1< a<b=n\),
then~\(r\)
above is replaced by \(\delta\colon M[1,a-1] \to M[a,n]\).
One of these maps is missing if \(a=b\)
or if \((a,b) = (1,n)\),
that is, on the two outer diagonals in the diagram in
Figure~\ref{fig:filtrated_K_3} (now for general~\(n\)).
We define \(M[a+1,a]\defeq0\)
for \(0\le a \le n\) to make this a special case of the generic case.

Now we build an \(\Ideal^X_\mathrm{fil}\)\nb-projective
resolution of~\(A\)
of length~\(1\).
This has not yet been done in~\cite{Meyer-Nest:Filtrated_K}, where
only the existence of such a resolution is proven.
For \(1\le a \le b \le n\),
let \(M[a,b]_\mathrm{ss}\)
be the quotient of~\(M[a,b]\)
by the images of all proper natural transformations to \(M[a,b]\)
or, equivalently, by the images of the two shortest natural
transformations:
\[
M[a,b]_\mathrm{ss} \defeq
\begin{cases}
  M[a,b]\bigm/ (i(M[a+1,b]) + r(M[a,b+1]))&\text{if }b<n,\\
  M[a,b]\bigm/ (i(M[a+1,b]) + \delta(M[1,a-1]))&\text{if }b=n
\end{cases}
\]
(compare \cite{Meyer-Nest:Filtrated_K}*{Definition~3.7 and
  Lemma~3.8}).  Choose a resolution
\begin{equation}
  \label{eq:Mab_ss_resolution}
  \begin{tikzcd}
    Q_1[a,b] \arrow[r, rightarrowtail, "d_1"] &
    Q_0[a,b] \arrow[r, twoheadrightarrow, "d_0"] &
    M[a,b]_\mathrm{ss}
  \end{tikzcd}
\end{equation}
of~\(M[a,b]_\mathrm{ss}\)
by countable \(\Z/2\)\nb-graded
free Abelian groups.  For \(i=0,1\), let
\[
\hat{Q}_i[a,b] \defeq \mathcal{R}_{[a,b]} \otimes_\Z Q_i[a,b],
\]
where the tensor product is defined as
in~\eqref{def:tensor_with_Cstar}.
Since~\(\mathcal{R}_{[a,b]} \otimes_\Z Q_i[a,b]\)
is a direct sum of copies of suspensions of~\(\mathcal{R}_{[a,b]}\),
the definition of~\(\mathcal{R}_{[a,b]}\)
as a representing object implies
\begin{equation}
  \label{eq:KK_out_of_RY}
  \KK^X_0(\hat{Q}_i[a,b], B) \cong \Hom\bigl(Q_i[a,b], \K_*(B[a,b])\bigr)
\end{equation}
for all \(\Cst\)\nb-algebras~\(B\)
over~\(X\);
here~\(\Hom\)
means grading-preserving group homomorphisms.
This property characterises~\(\hat{Q}_i[a,b]\)
uniquely up to isomorphism in~\(\KKcat^X\).
Equation~\eqref{eq:KK_out_of_RY} implies a similar description of
the \(\Z/2\)\nb-graded Abelian group
\(\KK^X_*(\hat{Q}_i[a,b], B)\),
replacing \(\Hom\)
by group homomorphisms that need not respect the grading.
Given a group homomorphism \(g\colon Q_i[a,b] \to \K_*(B[a,b])\),
let~\(g^\#\)
denote the corresponding element of
\(\KK^X_*(\hat{Q}_i[a,b], B)\).

Since the
\(\Z/2\)\nb-graded
Abelian group \(Q_0[a,b]\)
is free, the homomorphism~\(d_0\)
in~\eqref{eq:Mab_ss_resolution} lifts to a grading-preserving
homomorphism
\begin{equation}
  \label{eq:Q0_to_M}
  f[a,b]\colon Q_0[a,b] \to M[a,b] = \K_*(A[a,b]).
\end{equation}
Let \(f[a,b]^\# \in \KK^X_0(\hat{Q}_0[a,b], A)\)
correspond to~\(f[a,b]\) by~\eqref{eq:KK_out_of_RY}.  Let
\[
P_i \defeq \bigoplus_{1\le a\le b\le n} \hat{Q}_i[a,b]
\]
for \(i=0,1\).  The objects \(\hat{Q}_i[a,b]\)
and~\(P_i\)
for \(i=0,1\)
are \(\Ideal^X_\mathrm{fil}\)-projective
because of~\eqref{eq:KK_out_of_RY}.  There is
a unique element \(f \in \KK^X_0(P_0,A)\)
that restricts to~\(f[a,b]^\#\)
on the summand \(\hat{Q}_i[a,b]\).

\begin{lemma}
  \label{lem:surjective_filtrated_K}
  The map \(FK(f)\colon FK(P_0) \to M \defeq FK(A)\)
  is surjective.  Its kernel is isomorphic to~\(FK(P_1)\)
  as an \(\mathcal{NT}\)-module.
\end{lemma}

\begin{proof}
  Let \(f_*\defeq FK(f)\).
  The \(\mathcal{NT}\)\nb-module \(M_\mathrm{ss}\)
  is defined as the quotient \(M/\mathcal{NT}_\mathrm{nil}\cdot M\)
  for a certain ideal~\(\mathcal{NT}_\mathrm{nil}\)
  in~\(\mathcal{NT}\).
  It follows that the functor \(M\mapsto M_\mathrm{ss}\) is right exact.
  Even more, it is isomorphic to the tensor product functor with the
  right \(\mathcal{NT}\)\nb-module
  \(\mathcal{NT}_\mathrm{ss}\).
  This follows from the extension of
  \(\mathcal{NT}\)\nb-modules
  \(\mathcal{NT}_\mathrm{nil}\into\mathcal{NT}
  \prto\mathcal{NT}_\mathrm{ss}\)
  and the right exactness of tensor product functors.

  Right exactness implies
  \((\coker f_*)[a,b]_\mathrm{ss} = \coker (f_*[a,b]_\mathrm{ss})\)
  for all \(1 \le a \le b \le n\).
  The projective \(\mathcal{NT}\)\nb-module~\(FK(P_0)\)
  has \(FK(P_0)[a,b]_\mathrm{ss} \cong Q_0[a,b]\).
  Hence \((f_*)_\mathrm{ss}\)
  is surjective by construction of~\(f\),
  and the \(\mathcal{NT}\)\nb-module
  \(\coker f_*\)
  satisfies \((\coker f_*)_\mathrm{ss} = 0\).
  This implies \(\coker f_*=0\)
  by \cite{Meyer-Nest:Filtrated_K}*{Proposition~3.10}.  That is,
  \(f_*\)
  is surjective.  Let \(N\defeq \ker FK(f)\).
  So there is an extension \(N \into FK(P_0) \prto M\)
  of \(\mathcal{NT}\)\nb-modules.
  Since \(M\)
  and \(FK(P_0)\)
  are exact \(\mathcal{NT}\)\nb-modules,
  they satisfy
  \(\Tor^1_\mathcal{NT}(\mathcal{NT}_\mathrm{ss}, M)=0\)
  and \(\Tor^1_\mathcal{NT}(\mathcal{NT}_\mathrm{ss}, FK(P_0))=0\)
  by \cite{Meyer-Nest:Filtrated_K}*{Lemma~3.13}.  Hence
  \[
  N[a,b]_\mathrm{ss}
  \to FK(P_0)[a,b]_\mathrm{ss}
  \to M[a,b]_\mathrm{ss}
  \]
  is a short exact sequence for all \(1 \le a \le b \le n\)
  and \(\Tor^1_\mathcal{NT}(\mathcal{NT}_\mathrm{ss}, N)=0\).
  Since \(FK(P_0)[a,b]_\mathrm{ss} \cong Q_0[a,b]\),
  this implies \(N[a,b]_\mathrm{ss} \cong Q_1[a,b]\).
  Now \cite{Meyer-Nest:Filtrated_K}*{Theorem~3.12} shows that~\(N\)
  is a projective \(\mathcal{NT}\)-module.
  In fact, the proof of this theorem shows that \(N \cong FK(P_1)\).
  More precisely, the quotient maps
  \(N[a,b] \prto N[a,b]_\mathrm{ss}\)
  split because \(N[a,b]_\mathrm{ss} \cong Q_1[a,b]\)
  is free.  Let
  \begin{equation}
    \label{eq:section_hab}
    \varphi[a,b]\colon Q_1[a,b] \cong N[a,b]_\mathrm{ss}
    \to N[a,b] \subseteq FK(P_0)[a,b]
  \end{equation}
  be sections.  They induce an \(\mathcal{NT}\)\nb-module
  homomorphism \(FK(P_1) \to N\)
  by the universal property of the ``free''
  \(\mathcal{NT}\)\nb-module~\(FK(P_1)\).
  And the proof of \cite{Meyer-Nest:Filtrated_K}*{Theorem~3.12} shows
  that it is an isomorphism.
\end{proof}

Disregarding the \(\Z/2\)\nb-grading, we may write
\[
FK(P_0)[a,b] = \bigoplus_{[c,d]\to[a,b]} Q_0[c,d],
\]
that is, the sum runs over all \(1\le c \le d \le n\)
with \(a\le c \le b \le d\)
or \(c+1\le a \le d+1 \le b\).
So the map~\(\varphi[a,b]\) in~\eqref{eq:section_hab} has components
\[
\varphi_{[a,b]}^{[c,d]}\colon Q_1[a,b] \to Q_0[c,d]
\]
for \([c,d] \to [a,b]\).
Since~\(\varphi[a,b]\)
is even, the map \(\varphi_{\7{a,b}}^{\7{c,d}}\)
has the same parity as \(\tautr{c,d}{a,b}\),
that is, it is grading-preserving if \(a \le c \le b \le d\)
and grading-reversing if \(c+1 \le a \le d+1 \le b\).
The image of \(\varphi[a,b]\)
for \(1\le a \le b \le n\)
is contained in \(N = \ker FK(f)\),
that is, \(FK(f) \circ \varphi[a,b]=0\)
as a map \(Q_1[a,b] \to \K_*(A[a,b])\).
Unravelling the definition of \(FK(f)\), this becomes
\begin{equation}
  \label{eq:tau_f_phi_relation}
  \sum_{[c,d]\to[a,b]}
  \tau_{[c,d]}^{[a,b]} \circ f[c,d] \circ \varphi_{[a,b]}^{[c,d]}
  = 0 \colon Q_1[a,b] \to \K_*(A[a,b]).
\end{equation}
The objects \(P_0\) and~\(P_1\)
are \(\Ideal^X_\mathrm{fil}\)-projective,
and \(FK\)
is fully faithful on \(\Ideal^X_\mathrm{fil}\)-projective objects.
So the arrow \(FK(P_1) \to FK(P_0)\)
with components \(\varphi_{[a,b]}^{[c,d]}\)
for \([c,d]\to[a,b]\)
lifts uniquely to an arrow \(\varphi\in \KK^X_0(P_1,P_0)\).
More precisely, the map~\(\varphi\)
is given by a matrix of maps
\(\mathcal{R}_{[c,d]} \otimes Q_1[c,d] \to
\mathcal{R}_{[a,b]} \otimes Q_0[a,b]\)
for all \(1\le c\le d\le n\)
and \(1\le a\le b\le n\).
The entries of this matrix are
\(\Bigl(\bigl(\tautr{a,b}{c,d}\bigr)^* \otimes
\varphi_{\7{c,d}}^{\7{a,b}}\Bigr)\), that is,
\begin{multline}
  \label{eq:varphi_components}
  \varphi = \Bigl(\bigl(\tautr{a,b}{c,d}\bigr)^* \otimes
  \varphi_{\7{c,d}}^{\7{a,b}}\Bigr)_{[c,d]\to[a,b]}\colon\\
  \bigoplus_{1\le c\le d\le n} \mathcal{R}_{[c,d]} \otimes Q_1[c,d] \to
  \bigoplus_{1\le a\le b\le n} \mathcal{R}_{[a,b]} \otimes Q_0[a,b].
\end{multline}
Here we use the convention that \(\tautr{a,b}{c,d}=0\)
if not \([a,b]\to [c,d]\).

Lemma~\ref{lem:surjective_filtrated_K} says
that~\eqref{eq:fil_K_resolution} with \(f\)
and~\(\varphi\)
as above is \(\Ideal^X_\mathrm{fil}\)-exact,
that is, \(FK\)
applied to~\eqref{eq:fil_K_resolution} is an exact sequence.  The
\(\Ideal^X_\mathrm{fil}\)-exactness
of~\eqref{eq:fil_K_resolution} says that the functor
\(B\mapsto \K_*(B[a,b])\)
maps it to an exact sequence for each \(1 \le a \le b \le n\).
In fact, this gives projective resolutions.  We write them down
explicitly:

\begin{lemma}
  \label{lem:resolution_KBab}
  Let \(1 \le a \le b \le n\).  Then
  \begin{equation}
    \label{eq:varphi_f_exact_sequence_ab}
    \bigoplus_{[c,d] \to [a,b]} Q_1[c,d]
    \xrightarrow{\bigl(\varphi_{\7{c,d}}^{\7{e,g}}\bigr)_{\7{c,d}\to \7{e,g}\to \7{a,b}}}
    \bigoplus_{[e,g] \to [a,b]} Q_0[e,g]
    \xrightarrow{\bigl(\tau_{\7{e,g}}^{\7{a,b}}\circ f\7{e,g}\bigr)}
    \K_*(A[a,b])  
  \end{equation}
  is a free resolution.  Here
  \(\bigl(\varphi_{\7{c,d}}^{\7{e,g}}\bigr)_{\7{c,d}\to \7{e,g}\to
    \7{a,b}}\)
  means that the matrix entry is \(\varphi_{\7{c,d}}^{\7{e,g}}\)
  if \(\7{c,d}\to \7{e,g}\to \7{a,b}\)
  as described in~\eqref{eq:compose_tau}, and~\(0\) otherwise.
\end{lemma}

The boundary maps in~\eqref{eq:varphi_f_exact_sequence_ab} are
inhomogeneous, that is, the matrix entries of the maps may have even
or odd degree.

\begin{proof}
  Equation~\eqref{eq:NT_explicit} computes the group
  \(\K_*(\mathcal{R}_{[c,d]}([a,b])) \cong
  \mathcal{NT}_*([c,d],[a,b])\):
  it is~\(\Z\)
  in even or odd degree if \([c,d] \to [a,b]\)
  and~\(0\)
  otherwise.  Therefore,
  \(P_i[a,b] \cong \bigoplus_{[c,d] \to [a,b]} Q_i[c,d]\)
  for \(i=0,1\),
  disregarding the grading.  The map
  \(\Bigl(\bigl(\tautr{e,g}{c,d}\bigr)^* \otimes
  \varphi_{\7{c,d}}^{\7{e,g}}\Bigr)\)
  between the summands \(\mathcal{R}_{[c,d]} \otimes Q_1[c,d]\)
  in~\(P_1\)
  and \(\mathcal{R}_{[e,g]} \otimes Q_0[e,g]\)
  in~\(P_0\)
  induces the map
  \(\varphi_{\7{c,d}}^{\7{e,g}}\colon Q_1[c,d] \to Q_0[e,g]\)
  if \([c,d] \to [e,g] \to [a,b]\),
  and~\(0\)
  otherwise, compare~\eqref{eq:compose_tau}.  The map
  \(\K_*(P_0[a,b]) \to \K_*(A[a,b])\) corresponds to a family of maps
  \[
  \K_*(\mathcal{R}_{[c,d]}[a,b]) \otimes Q_0[c,d] \cong
  \K_*(\mathcal{R}_{[c,d]}[a,b] \otimes Q_0[c,d]) \to \K_*(A[a,b])
  \]
  for \(1 \le c \le d \le n\).
  As above, \(\K_*(\mathcal{R}_{[c,d]}[a,b])\neq0\)
  only if \([c,d] \to [a,b]\),
  and then the map \(Q_0[c,d] \to K_*(A[a,b])\)
  induced by \(f\colon P_0 \to A\)
  is \(\tautr{c,d}{a,b} \circ f[c,d]\).
\end{proof}

We have reached the first milestone in the computation of the
obstruction class: the \(\Ideal_\mathrm{fil}^X\)-projective
resolution~\eqref{eq:fil_K_resolution}.  It is explicit enough to
express the obstruction class that comes from filtrated
\(\K\)\nb-theory
in the terms of Theorem~\ref{the:obstruction_class_X}, namely, as
being represented by a family of elements in
\(\Ext^1\bigl(\K_*(A[e+1,n]),\K_*(A[e,n])\bigr)\)
for \(e=1,\dotsc,n-1\).
We compute these Ext-groups with the resolutions
in~\eqref{eq:varphi_f_exact_sequence_ab}.  So the obstruction class
corresponds to a sequence of maps
\[
\delta_e\colon
\bigoplus_{[a,b] \to [e+1,n]} Q_1[a,b] \to \K_*(A[e,n]),\qquad
e=1,\dotsc,n-1.
\]
In turn, each~\(\delta_e\)
is given by maps \(\delta_e^{[a,b]}\colon Q_1[a,b] \to \K_*(A[e,n])\)
for all \(1 \le a \le b \le n\)
with \([a,b] \to [e+1,n]\).
The following theorem computes these maps~\(\delta_e^{[a,b]}\).
It is the main result of this section.
Section~\ref{sec:proof_filtrated_theorem} is dedicated to its proof.

\begin{theorem}
  \label{the:compare_obstruction_filtered}
  Let
  \[
  \delta_e^{[a,b]} \defeq
  \begin{cases}
    \displaystyle\sum_{[c,d] \to [a,n] \to [a,b]}
    \tautr{c,d}{e,n} \circ f[c,d] \circ \varphi_{\7{a,b}}^{\7{c,d}}&
    \text{if }a=e \text{ and }b<n,\\
    0&\text{otherwise.}
  \end{cases}
  \]
  The resulting family of maps \((\delta_e)_{1 \le e <n}\)
  lifts the obstruction class of~\(A\)
  to an element of
  \(\prod_{e=1}^{n-1} \Ext^1\bigl(\K_*(A[e+1,n]),\K_*(A[e,n])\bigr)\).
\end{theorem}

Any \(\Ideal_\mathrm{fil}^X\)-projective
resolution of~\(A\)
is isomorphic to one of the form above because any
\(\Ideal_\mathrm{fil}^X\)-epic
map from an \(\Ideal_\mathrm{fil}^X\)-projective
object to~\(A\)
is isomorphic to a map \(f\in \KK^X_0(P_0,A)\)
as above.  Since \(\Ideal_\mathrm{fil}^X\)-projective
resolutions of~\(A\)
are equivalent to projective resolutions of \(FK(A)\),
the maps~\(\delta_e^{[a,b]}\)
may, in principle, be computed from~\(FK(A)\)
by choosing a projective resolution.  This gives the maps \(f[c,d]\)
and \(\varphi_{[a,b]}^{[c,d]}\).
Such a computation may, of course, be difficult in practice.

\subsection{Proof of the obstruction class formula}
\label{sec:proof_filtrated_theorem}

First we examine the smaller invariant
\(F^X_\K\colon \KKcat^X \to \Abel^X\).
This takes the part of filtrated \(\K\)\nb-theory
consisting of \(\K_*(A[a,n])\)
for \(1 \le a \le n\)
with the maps~\(i\)
between them because the minimal open subset containing~\(a\)
is \(U_a = [a,n]\).
So the diagram~\(F^X_\K(A)\)
is simply the first row in the diagram in
Figure~\ref{fig:filtrated_K}.  We have
\(i_a \C \cong \mathcal{R}_{[a,n]}\)
for \(1\le a \le n\)
because both objects represent the same functor
\(A\mapsto \K_*(A[a,n])\).
So~\(\mathcal{R}_{[a,n]}\)
for \(1 \le a \le n\)
is \(\Ideal^X_\K\)\nb-projective,
and \(F^X_\K(\mathcal{R}_{[a,n]})\) is the diagram
\begin{equation}
  \label{eq:projectives_in_AX}
  \mathcal{P}_{[a,n]} \defeq F^X_\K(\mathcal{R}_{[a,n]}) = \Bigl(
  \underbrace{\Z = \Z = \dotsb = \Z}_{a \text{ times}} \leftarrow 
  \underbrace{0 = \dotsb = 0}_{n-a \text{ times}}
  \Bigr).
\end{equation}
in~\(\Abel^X\).
If \(1 \le a \le b \le n-1\),
then \(F^X_\K(\mathcal{R}_{[a,b]})\) is the diagram
\begin{equation}
  \label{eq:generators_in_AX}
  \mathcal{P}_{[a,b]} \defeq F^X_\K(\mathcal{R}_{[a,b]}) = \Bigl(
  \underbrace{0 = \dotsb = 0}_{a \text{ times}} \leftarrow
  \underbrace{\Z_- = \dotsb = \Z_-}_{b+1-a \text{ times}} \leftarrow
  \underbrace{0 = \dotsb = 0}_{n-1-b \text{ times}}
  \Bigr)
\end{equation}
because of the formula for
\(\K_*(\mathcal{R}_{[a,b]}([c,n]))\)
in~\eqref{eq:NT_explicit}.  Thus the objects~\(\mathcal{R}_{[a,b]}\)
for \(1\le a \le b \le n\) are even if \(b=n\) and odd if \(b<n\).

The formula for the obstruction class in
Theorem~\ref{the:compute_obstruction} uses the parity-reversing part
of \(\varphi\in\KK^X_0(P_1,P_0)\).
This is described by the following lemma:

\begin{lemma}
  \label{lem:fil_K_parity-reversing}
  The component
  \(\bigl(\tautr{c,d}{a,b}\bigr)^* \otimes
  \varphi_{\7{a,b}}^{\7{c,d}}\)
  of~\(\varphi\)
  is parity-reversing if and only if \([c,d] \to [a,n]\to[a,b]\)
  and \(b<n\).
\end{lemma}

\begin{proof}
  If \(\tautr{c,d}{a,b}\neq0\),
  then either \(a \le c \le b \le d\)
  or \(c+1 \le a \le d+1 \le b\).
  The map
  \(\bigl(\tautr{c,d}{a,b}\bigr)^* \otimes
  \varphi_{\7{a,b}}^{\7{c,d}}\)
  always belongs to \(\KK^X_0(\hat{Q}_1[a,b],\hat{Q}_0[c,d])\).
  So \(\varphi_{\7{a,b}}^{\7{c,d}}\colon Q_1[a,b] \to Q_0[c,d]\)
  is parity-preserving in the first case and parity-reversing in the
  second case.  The object~\(\mathcal{R}_{[a,b]}\)
  is even if \(b=n\)
  and odd if \(b<n\).

  First let \(a \le c \le b \le d\).
  If \(b=n\),
  then \(d=n\)
  as well, so that
  \(\mathcal{R}_{[a,b]}\) and~\(\mathcal{R}_{[c,d]}\)
  have the same parity, and~\(\varphi_{\7{a,b}}^{\7{c,d}}\)
  preserves parity.  So we get a parity-preserving component
  of~\(\varphi\).
  For the same reasons, we get a parity-preserving component if
  \(a \le c \le b\le d<n\),
  and a parity-reversing component if
  \(a \le c \le b < d=n\).

  Now let \(c+1 \le a \le d+1 \le b\),
  so that~\(\varphi_{\7{a,b}}^{\7{c,d}}\)
  reverses parity.  If \(b=n\),
  then \(d<n\),
  so that \(\mathcal{R}_{[a,b]}\)
  and~\(\mathcal{R}_{[c,d]}\)
  have opposite parity.  Hence
  \(\bigl(\tautr{c,d}{a,b}\bigr)^* \otimes
  \varphi_{\7{a,b}}^{\7{c,d}}\)
  is parity-preserving altogether.  If \(b<n\),
  then \(d<n\)
  and so \(\mathcal{R}_{[a,b]}\)
  and~\(\mathcal{R}_{[c,d]}\)
  have the same parity.  Thus
  \(\bigl(\tautr{c,d}{a,b}\bigr)^* \otimes
  \varphi_{\7{a,b}}^{\7{c,d}}\)
  is parity-reversing.  Inspection shows that the parity-reversing
  components are exactly those for which \([c,d] \to [a,n] \to [a,b]\)
  as in~\eqref{eq:compose_tau} and \(b<n\).
\end{proof}

Let \(1 \le a \le b < n\).  The exact triangle
\begin{equation}
  \label{eq:projective_triangle_Rab}
  \Sigma \mathcal{R}_{[a,n]} \xrightarrow{i^*}
  \Sigma \mathcal{R}_{[b+1,n]} \xrightarrow{\delta^*}
  \mathcal{R}_{[a,b]} \xrightarrow{r^*}
  \mathcal{R}_{[a,n]}
\end{equation}
in~\(\KKcat^X\)
is \(\Ideal^X_\K\)\nb-exact,
that is, \(F^X_\K(r^*)=0\).
Since \(\Sigma \mathcal{R}_{[a,n]}\)
and \(\Sigma \mathcal{R}_{[b+1,n]}\)
are \(\Ideal^X_\K\)\nb-projective,
it is an \(\Ideal^X_\K\)\nb-projective
resolution of~\(\mathcal{R}_{[a,b]}\)
of length~\(1\).
Since we allow both odd and even arrows in diagrams, we may drop the
suspensions in~\eqref{eq:projective_triangle_Rab}.  For
\(1 \le a \le b \le n\), let
\[
  M(a,b) \defeq
  \begin{cases}
    a&\text{if }b=n,\\
    b+1&\text{if }b<n.
  \end{cases}
\]
Then
\(\bigl(\tau_{[a,b]}^{[M(a,b),n]}\bigr)^*\colon
\mathcal{R}_{[M(a,b),n]} \to \mathcal{R}_{[a,b]}\)
is an \(\Ideal^X_\K\)\nb-epimorphism
both for \(b=n\)
and \(b<n\).
Its cone is \(\mathcal{R}_{[a,n]}\)
if \(b<n\)
and~\(0\)
if \(b=n\).
Now we write down \(\Ideal^X_\K\)\nb-projective
resolutions of~\(P_i\) for \(i=0,1\).  Let
\begin{align}
  P_{i0} &\defeq
           \bigoplus_{1\le a\le b\le n} \mathcal{R}_{[M(a,b),n]} \otimes Q_i[a,b],\\
  P_{i1} &\defeq
           \bigoplus_{1\le a\le b< n} \mathcal{R}_{[a,n]} \otimes Q_i[a,b].
\end{align}
Then the following is an \(\Ideal^X_\K\)\nb-projective resolution:
\begin{equation}
  \label{eq:resolution_Pi}
  0 \to
  P_{i1} \xrightarrow{\bigoplus \tau^*\otimes\id}
  P_{i0} \xrightarrow{\bigoplus \tau^*\otimes\id}
  P_i \to 0.
\end{equation}
Here \(\bigoplus \tau^* \otimes \id\)
means the direct sum of the maps
\(\bigl(\tau_{[z,w]}^{[x,y]}\bigr)^* \otimes \id_{Q_i[a,b]}\)
between the summands for fixed \(a,b\),
with the appropriate \(x,y,z,w\).
So~\eqref{eq:resolution_Pi} is the direct sum of the
resolutions~\eqref{eq:projective_triangle_Rab}, tensored
with~\(Q_i[a,b]\),
over all \(1\le a\le b<n\),
and the trivial resolutions
\(0 \to \mathcal{R}_{[a,n]} \to \mathcal{R}_{[a,n]}\),
tensored with~\(Q_i[a,b]\),
over all \(1 \le a \le b=n\).
The maps in~\eqref{eq:resolution_Pi} are inhomogeneous, that is, some
components are in \(\KK^X_0\) and others in \(\KK^X_1\).

The resolution~\eqref{eq:resolution_Pi} allows us to compute
\(\KK^X_0(P_1,P_0)\)
with the Universal Coefficient Theorem for the invariant~\(F^X_\K\).
First, the long exact sequence for the direct sum of the exact
triangles~\eqref{eq:projective_triangle_Rab} implies a natural
extension of Abelian groups
\begin{multline*}
  \coker \bigl(\KK^X_*(\Sigma P_{10},P_0) \to \KK^X_*(\Sigma P_{11},P_0)\bigr)
  \into \KK^X_*(P_1,P_0)
  \\\prto \ker \bigl(\KK^X_*(P_{10},P_0) \to \KK^X_*(P_{11},P_0)\bigr).
\end{multline*}
We may rewrite the kernel and cokernel here as \(\Hom\)
and \(\Ext\)
in~\(\Abel^X\),
using that~\eqref{eq:resolution_Pi} is an
\(\Ideal^X_\K\)\nb-projective
resolution.  The extension above splits unnaturally, giving the
decomposition of~\(\varphi\)
into its parity-preserving and -reversing parts \(\varphi^+\)
and~\(\varphi^-\), respectively.

\begin{lemma}
  \label{lem:odd_phi}
  The image of the parity-reversing part~\(\varphi^-\)
  of~\(\varphi\) in \(\Ext^1_{\Ideal^X_\K}(P_1,P_0)\) is the map
  \[
  P_{11}
  = \bigoplus_{1 \le a \le b <n} \mathcal{R}_{[a,n]}\otimes Q_1[a,b]
  \to \bigoplus_{1 \le c \le d \le n} \mathcal{R}_{[c,d]}\otimes Q_0[c,d]
  = P_0
  \]
  with matrix coefficients
  \(\bigl(\tautr{c,d}{a,n}\bigr)^* \otimes
  \varphi_{\7{a,b}}^{\7{c,d}}\)
  if \([c,d] \to [a,n] \to [a,b]\) and \(b<n\), and~\(0\) otherwise.
\end{lemma}

\begin{proof}
  In terms of the matrix description of~\(\varphi\),
  each matrix entry
  \(\bigl(\tautr{c,d}{a,b}\bigr)^* \otimes
  \varphi_{\7{a,b}}^{\7{c,d}}\)
  has even or odd parity and thus belongs to either \(\Hom\)
  or \(\Ext\),
  respectively.  By Lemma~\ref{lem:fil_K_parity-reversing}, the entry
  belongs to~\(\varphi^-\)
  if and only if~\(\tautr{c,d}{a,b}\)
  factors through
  \(r^*\colon \mathcal{R}_{[a,b]} \to \mathcal{R}_{[a,n]}\).
  In this case, it factors as
  \(\Bigl(\bigl(\tautr{c,d}{a,n}\bigr)^* \otimes
  \varphi_{\7{a,b}}^{\7{c,d}}\Bigr) \circ (r^*\otimes \id)\).
  Since~\(r^*\)
  is the boundary map in~\eqref{eq:projective_triangle_Rab}, this
  exhibits a map \(P_{11} \to P_0\).
  The map from \(\Ext^1_{\Abel^X}\)
  to \(\KK^X_0\)
  in the UCT is defined by composing with the boundary map in the
  exact triangle that contains the given resolution.  So the map
  \(P_{11} \to P_0\)
  found above is the relevant component of~\(\varphi\).
  The formula in the lemma follows.
\end{proof}

According to the recipe in Theorem~\ref{the:compute_obstruction}, the
obstruction class in \(\Ext^2_{\Ideal^X_\K}(\Sigma A, A)\)
is the composite of the parity-reversing part of~\(\varphi\),
viewed as an element of \(\Ext^1_{\Ideal^X_\K}(P_1,P_0)\),
with \(f\in \KK^X_0(P_0, A)\)
and with the class of the extension~\eqref{eq:fil_K_resolution} in
\(\Ext_{\Ideal^X_\K}(A,P_1)\).
Composing the two extensions gives a length-\(2\) resolution
\begin{equation}
  \label{eq:resolution_1_for_A}
  P_{11} \into  P_{10} \to P_0 \prto A.  
\end{equation}
The component \(\mathcal{R}_{[a,n]}\otimes Q_1[a,b] \to A\)
in the composite map \(P_{11} \to P_0 \to A\)
is the map \(\mathcal{R}_{[a,n]}\otimes Q_1[a,b] \to A\)
that corresponds to
\begin{equation}
  \label{eq:odd_phi_explicit_filtered}
  \sum_{[c,d] \to [a,n] \to [a,b]}
  \tau_{[c,d]}^{[a,n]}\circ f[c,d] \circ \varphi_{[a,b]}^{[c,d]}
  \colon Q_1[a,b] \to \K_*(A[a,n])
\end{equation}
under the isomorphism~\eqref{eq:KK_out_of_RY}.  Here the sum runs only
over those \([c,d]\)
with \([c,d] \to [a,n] \to [a,b]\)
as in Lemma~\ref{lem:odd_phi}.  In contrast, the sum over all
\([c,d]\)
with \([c,d] \to [a,b]\) is~\(0\) by~\eqref{eq:tau_f_phi_relation}.

In a sense, we have now computed the obstruction class.  The
length-\(2\)
resolution in~\eqref{eq:resolution_1_for_A} is, however, different
from the one that is implicitly used in
Theorem~\ref{the:obstruction_class_X} to compute the relevant
\(\Ext^2\)-group
and the obstruction class in it.  To translate the formula for the
obstruction class that we get from filtrated \(\K\)\nb-theory
into the setting of Theorem~\ref{the:obstruction_class_X}, we must
compare the underlying length-\(2\)
resolutions.  First, we replace the resolution
in~\eqref{eq:resolution_1_for_A} by one that is
\(\Ideal^X_\K\)\nb-projective.

The entries \(P_{10}\)
and~\(P_{11}\)
are already \(\Ideal^X_\K\)\nb-projective,
and~\eqref{eq:resolution_Pi} is an \(\Ideal^X_\K\)\nb-projective
resolution of~\(P_0\).
The objects \(P_i\)
and~\(P_{ij}\)
are all sums over \(1\le a \le b \le n\),
with summands of the form \(\mathcal{R}_{[x,y]} \otimes Q_i[a,b]\)
for suitable \(x,y\)
depending on \(a,b\);
the summands in~\(P_{i1}\)
are~\(0\)
for \(b=n\).
Let \((\tau^* \otimes \varphi)\)
denote the map between these sums for \(i=1\)
to those for \(i=0\) with matrix entries
\[
\bigl(\tau_{[z,w]}^{[x,y]}\bigr)^* \otimes \varphi_{[a,b]}^{[c,d]}\colon
\mathcal{R}_{[x,y]} \otimes Q_1[a,b]
\to \mathcal{R}_{[z,w]} \otimes Q_0[c,d].
\]
As usual, \(\tau_{[z,w]}^{[x,y]}=0\) if \(\mathcal{NT}_*([z,w],[x,y]) =0\).

\begin{lemma}
  \label{lem:commuting_P_diagram}
  There is a commuting diagram
  \[
  \begin{tikzcd}[column sep = huge]
    P_{11} \arrow[r, rightarrowtail, "\bigoplus \tau^*\otimes \id"]
    \arrow[d, "(\tau^* \otimes \varphi)"] &
    P_{10}
    \arrow[r, twoheadrightarrow, "\bigoplus \tau^*\otimes \id"]
    \arrow[d, "(\tau^* \otimes \varphi)"] &
    P_1 \arrow[d, "(\tau^* \otimes \varphi)"] \\
    P_{01} \arrow[r, rightarrowtail, "\bigoplus \tau^*\otimes \id"'] &
    P_{00}
    \arrow[r, twoheadrightarrow, "\bigoplus \tau^*\otimes \id"'] &
    P_0
  \end{tikzcd}
  \]
\end{lemma}

\begin{proof}
  We compare maps between direct sums by comparing their matrix
  coefficients.  For the two composite maps \(P_{11} \to P_{00}\),
  these are maps
  \(\mathcal{R}_{[a,n]} \otimes Q_1[a,b] \to \mathcal{R}_{[M(c,d),n]}
  \otimes Q_0[c,d]\)
  for \(1 \le a \le b <n\)
  and \(1 \le c \le d \le n\).
  The composite map through~\(P_{10}\)
  is
  \(\bigl(\tautr{M(c,d),n}{a,n}\bigr)^* \otimes
  \varphi_{[a,b]}^{[c,d]}\)
  if \([a,n] \leftarrow [M(a,b),n] \leftarrow [M(c,d),n]\),
  and~\(0\)
  otherwise; and the composite map through~\(P_{01}\)
  is
  \(\bigl(\tautr{M(c,d),n}{a,n}\bigr)^* \otimes
  \varphi_{[a,b]}^{[c,d]}\)
  if \([a,n] \leftarrow [c,n] \leftarrow [M(c,d),n]\)
  and \(d<n\),
  and~\(0\)
  otherwise; the condition \(d<n\)
  comes in because~\(P_{01}\)
  contains only summands \(\mathcal{R}_{[c,n]} \otimes Q_0[c,d]\)
  with \(1 \le c \le d <n\).
  In both cases, the map vanishes unless \([c,d] \to [a,b]\)
  because of the factor~\(\varphi_{[a,b]}^{[c,d]}\).
  We claim that if \([c,d] \to [a,b]\)
  and \(b<n\),
  then \([a,n] \leftarrow [M(a,b),n] \leftarrow [M(c,d),n]\)
  if and only if \([a,n] \leftarrow [c,n] \leftarrow [M(c,d),n]\)
  and \(d<n\);
  here \(M(a,b)=b+1\)
  because \(b<n\).
  Indeed, if \(d=n\),
  then \(M(c,d)=c\),
  and \([a,n] \leftarrow [b+1,n] \leftarrow [c,n]\)
  means \(a \le b+1 \le c\),
  which contradicts \([c,n] \to [a,b]\).
  If \(d<n\),
  then \(M(c,d)=d+1\).
  Then \([a,n] \leftarrow [c,n] \leftarrow [d+1,n]\)
  and \([a,n] \leftarrow [M(a,b),n] \leftarrow [M(c,d),n]\)
  are equivalent to \(a \le c \le d+1\)
  and \(a \le b+1 \le d+1\),
  respectively.  If \([c,d] \to [a,b]\),
  both conditions say that we are in the case \(a \le c \le b \le d\).
  The computations above show that the two maps \(P_{11} \to P_{00}\)
  are equal.

  Now consider the two maps \(P_{10} \to P_0\).
  Its matrix coefficients are maps
  \[
  \mathcal{R}_{[M(a,b),n]} \otimes Q_1[a,b] \to \mathcal{R}_{[c,d]} \otimes Q_0[c,d],
  \qquad
  1 \le a \le b \le n,\ 1 \le c \le d \le n.
  \]
  As above, the composite maps through~\(P_{00}\)
  and~\(P_1\)
  are
  \(\bigl(\tautr{c,d}{M(a,b),n}\bigr)^* \otimes
  \varphi_{[a,b]}^{[c,d]}\)
  or~\(0\).
  For the maps through \(P_1\)
  and~\(P_{00}\),
  the former case occurs if
  \([M(a,b),n] \leftarrow [a,b] \leftarrow [c,d]\)
  or \([M(a,b),n] \leftarrow [M(c,d),n] \leftarrow [c,d]\),
  respectively.  We may assume \([a,b] \leftarrow [c,d]\)
  and \([M(a,b),n] \leftarrow [c,d]\)
  because otherwise \(\varphi_{[a,b]}^{[c,d]} = 0\)
  or \(\tautr{c,d}{M(a,b),n}=0\).
  Under these assumptions,
  \([M(a,b),n] \leftarrow [a,b] \leftarrow [c,d]\)
  always holds by~\eqref{eq:compose_tau}.  And
  \([c,d] \to [M(a,b),n]\)
  implies \([M(a,b),n] \leftarrow [M(c,d),n] \leftarrow [c,d]\)
  because any natural transformation \(\K_*(A[c,d]) \to \K_*(A[e,n])\)
  for some \(1 \le e \le n\)
  factors through \(\tautr{c,d}{M(c,d),n}\).
  So the two maps \(P_{10} \to P_0\) are equal as well.
\end{proof}

Using also the resolution~\eqref{eq:fil_K_resolution} of~\(A\),
we get the following \(\Ideal^X_\K\)\nb-projective
resolution of~\(A\):
\[
P_{11} \xrightarrow{\begin{pmatrix}
    -(\tau^* \otimes \varphi)\\ \bigoplus \tau^* \otimes \id
  \end{pmatrix}}
P_{01} \oplus P_{10}
\xrightarrow{\begin{pmatrix} \bigoplus \tau^* \otimes \id&
    (\tau^* \otimes \varphi)
  \end{pmatrix}}
P_{00} \xrightarrow{\bigoplus f[a,b]^\# \circ (\tau^* \otimes \id)}
A.
\]

The computation of the obstruction class in
Theorem~\ref{the:classification_X_obstruction} starts with the
following \(\Ideal^X\)\nb-projective
resolution of length~\(1\) in~\(\KKcat^X\):
\begin{equation}
  \label{eq:exact_sequence_AjbA}
  \bigoplus_{b=1}^{n-1} \mathcal{R}_{[b,n]} \otimes A[b+1,n] 
  \into
  \bigoplus_{b=1}^n \mathcal{R}_{[b,n]} \otimes A[b,n]
  \prto
  A.
\end{equation}
Since \(\mathcal{R}_{[b,n]} = i_b(\C)\)
in the notation of Section~\ref{sec:UCT_X}, we have
\(\mathcal{R}_{[b,n]} \otimes A[b,n] \cong i_b(A[b,n])\).
The restriction of the second map in~\eqref{eq:exact_sequence_AjbA} to
this direct summand is the one that corresponds to the identity map
on~\(A[b,n]\)
under the isomorphism in~\eqref{eq:KK_out_ix}.  The first map
in~\eqref{eq:exact_sequence_AjbA}, restricted to the summand
\(\mathcal{R}_{[b,n]} \otimes A[b+1,n]\),
is the difference of the two maps
\begin{alignat*}{2}
  (\tautr{b+1,n}{b,n})^*\otimes \id &&\colon
  \mathcal{R}_{[b,n]}\otimes A[b+1,n]
  &\to \mathcal{R}_{[b+1,n]}\otimes A[b+1,n],\\
  \id\otimes \tautr{b+1,n}{b,n}&&\colon
  \mathcal{R}_{[b,n]} \otimes A[b+1,n]
  &\to \mathcal{R}_{[b,n]} \otimes A[b,n],
\end{alignat*}
where \(\tautr{b+1,n}{b,n}\)
denotes the inclusion of~\(A[b+1,n]\)
into~\(A[b,n]\);
we could have written~\(i\)
for~\(\tautr{b+1,n}{b,n}\)
as in Figure~\ref{fig:filtrated_K}.  It is shown in
Section~\ref{sec:UCT_X} that this sequence is \(\Ideal^X\)\nb-exact.
And it is easy to prove this directly.

The projective resolutions of Abelian groups
in~\eqref{eq:varphi_f_exact_sequence_ab} imply that there is an
\(\Ideal^X_\K\)\nb-projective resolution
\[
\bigoplus_{[c,d] \to [b+1,n]} \mathcal{R}_{[b,n]} \otimes Q_1[c,d]
\into
\bigoplus_{[c,d] \to [b+1,n]} \mathcal{R}_{[b,n]} \otimes Q_0[c,d]
\prto \bigoplus_{b=1}^{n-1} \mathcal{R}_{[b,n]} \otimes A[b+1,n].
\]
Splicing it with the resolution in~\eqref{eq:exact_sequence_AjbA}
gives an \(\Ideal^X_\K\)\nb-exact chain complex
\[
W_1 \into W_0
\to \bigoplus_{b=1}^n \mathcal{R}_{[b,n]} \otimes A[b,n]
\prto A
\]
with
\begin{equation}
  \label{eq:def_Wi}
  W_i \defeq \bigoplus_{b=1}^{n-1} \bigoplus_{[c,d] \to [b+1,n]}
  \mathcal{R}_{[b,n]} \otimes Q_i[c,d],\qquad i=0,1.
\end{equation}

Next we are going to compare the two \(\Ideal^X_\K\)\nb-exact
chain complexes built above.  We are going to build maps
\(\gamma_{ij}\)
and~\(\delta\)
for \(0\le i,j \le 1\)
that make the following diagram commute, and such that~\(\delta\)
gives the obstruction class:
\begin{equation}
  \label{eq:compare_resolutions}
  \begin{tikzcd}[column sep = large]
    P_{11} \arrow[r, rightarrowtail]
    \arrow[d, "\gamma_{11}"] &
    P_{01} \oplus P_{10}
    \arrow[r]
    \arrow[d, "(\gamma_{01}\ \gamma_{10})"] &
    P_{00} \arrow[r, twoheadrightarrow] \arrow[d, "\gamma_{00}"] &
    A \arrow[d, equal] \\
    W_1
    \arrow[r, rightarrowtail] \arrow[d, "\delta"] &
    W_0
    \arrow[r, twoheadrightarrow] &
    \bigoplus_{e=1}^n \mathcal{R}_{[e,n]} \otimes A[e,n] \arrow[r] &
    A\\
    A
  \end{tikzcd}
\end{equation}
We describe maps between direct sums through matrices of maps between
the direct summands.  Recall that~\(W_i\)
is defined in~\eqref{eq:def_Wi} and that
\[
P_{i0} \defeq \bigoplus_{1\le a\le b\le n}
\mathcal{R}_{[M(a,b),n]} \otimes Q_i[a,b],\qquad
P_{i1} \defeq \bigoplus_{1\le a\le b< n} \mathcal{R}_{[a,n]} \otimes Q_i[a,b]
\]
for \(i=0,1\).  The matrix coefficients of~\(\gamma_{00}\) are maps
\[
\gamma_{00}^{e,[a,b]}\colon
\mathcal{R}_{[M(a,b),n]} \otimes Q_0[a,b] \to \mathcal{R}_{[e,n]}
\otimes A[e,n]
\]
for \(1 \le a \le b \le n\)
and \(1 \le e \le n\).
We let \(\gamma_{00}^{e,[a,b]}=0\)
if \(e\neq M(a,b)\).
Let \(e=M(a,b)\).  Then \(\gamma_{00}^{e,[a,b]}\) corresponds to a map
\[
(\gamma_{00}^{e,[a,b]})^\flat\colon
Q_0[a,b] \to \K_*(\mathcal{R}_{[e,n]}[e,n] \otimes A[e,n])
\]
under the isomorphism~\eqref{eq:KK_out_of_RY}.  We have already used
above that \(\mathcal{R}_{[e,n]} \otimes A[e,n] \cong i_{e}(A[e,n])\);
so
\begin{equation}
  \label{eq:Ren_Aen}
  \mathcal{R}_{[e,n]}[e,n] \otimes A[e,n] \cong A[e,n].  
\end{equation}
Using this isomorphism implicitly, we let
\[
(\gamma_{00}^{e,[a,b]})^\flat \defeq
\tautr{a,b}{e,n}\circ f[a,b]\colon
Q_0[a,b] \xrightarrow{f[a,b]}
\K_*(A[a,b]) \xrightarrow{\tautr{a,b}{e,n}}
\K_*(A[e,n]).
\]
As usual, this is~\(0\)
unless \([a,b] \to [e,n]\).
The map \(\gamma_{01}\colon P_{01} \to W_0\)
is given by a matrix of maps
\[
\gamma_{01}^{e,[c,d],[a,b]}\colon
\mathcal{R}_{[a,n]} \otimes Q_0[a,b] \to \mathcal{R}_{[e,n]} \otimes
Q_0[c,d]
\]
for \(1 \le a \le b < n\),
\(1 \le e < n\),
and \(1 \le c \le d \le n\) with \([c,d] \to [e+1,n]\).  We let
\[
\gamma_{01}^{e,[c,d],[a,b]}\defeq
\begin{cases}
  \bigl(\tautr{e,n}{a,n}\bigr)^* \otimes \id_{Q_0[a,b]}
  &\text{if } [a,b]=[c,d],\\
  0&\text{otherwise.}
\end{cases}
\]
If \(a=c\)
and \(b=d<n\),
then \([c,d] \to [e+1,n]\)
if and only if \(a \le e \le b\),
so that \(\tautr{e,n}{a,n}\neq0\) in this formula.

The map \(\gamma_{10}\colon P_{10} \to W_0\) is given by a matrix of maps
\[
\gamma_{10}^{e,[c,d],[a,b]} \colon
\mathcal{R}_{[M(a,b),n]} \otimes Q_1[a,b] \to \mathcal{R}_{[e,n]}
\otimes Q_0[c,d]
\]
for \(1 \le a \le b \le n\),
\(1 \le e < n\),
and \(1 \le c \le d \le n\) with \([c,d] \to [e+1,n]\).  We let
\[
\gamma_{10}^{e,[c,d],[a,b]} \defeq
\begin{cases}
  \bigl(\tautr{e,n}{M(a,b),n}\bigr)^* \otimes \varphi_{[a,b]}^{[c,d]}
  &\text{if }M(a,b) \le e < M(c,d),\\
  0&\text{otherwise.}
\end{cases}
\]
The map \(\gamma_{11}\colon P_{11} \to W_1\)
is given by a matrix of maps
\[
\gamma_{11}^{e,[c,d],[a,b]}\colon
\mathcal{R}_{[a,n]} \otimes Q_1[a,b] \to
\mathcal{R}_{[e,n]} \otimes Q_1[c,d]
\]
for \(1 \le a \le b < n\),
\(1 \le e < n\),
and \(1 \le c \le d \le n\) with \([c,d] \to [e+1,n]\).  We let
\[
\gamma_{11}^{e,[c,d],[a,b]}\defeq
\begin{cases}
  \bigl(\tautr{e,n}{a,n}\bigr)^* \otimes \id_{Q_1[a,b]}
  &\text{if }a=c,\ b=d,\\
  0&\text{otherwise}.
\end{cases}
\]
The map~\(\delta\)
is given by a family of maps
\(\delta_{b,[c,d]}\colon \mathcal{R}_{[b,n]} \otimes Q_1[c,d] \to A\)
for \(1 \le b <n\)
and \([c,d] \to [b+1,n]\).
These correspond to maps
\(\delta_{b,[c,d]}^\flat \colon Q_1[c,d] \to \K_*(A[b,n])\)
by the isomorphism~\eqref{eq:KK_out_of_RY}.  We define~\(\delta\)
so that the maps \(\delta_{b,[c,d]}^\flat\)
are the maps denoted by~\(\delta_b^{[c,d]}\)
in Theorem~\ref{the:compare_obstruction_filtered}.

Now we must prove that the squares in the diagram commute.  We begin
on the right, comparing the two maps \(P_{00} \to A\).
Its restrictions \(\mathcal{R}_{[M(a,b),n]} \otimes Q_0[a,b] \to A\)
correspond to maps \(Q_0[a,b] \to \K_*(A[M(a,b),n])\)
under the isomorphism~\eqref{eq:KK_out_of_RY}.  This map is
\(\tautr{a,b}{M(a,b),n} \circ f[a,b]\)
both for the direct boundary map \(P_{00} \to A\)
and for the map through
\(\bigoplus_{e=1}^n \mathcal{R}_{[e,n]} \otimes A[e,n]\).
So this square commutes.

Next we compare the two maps from \(P_{10}\oplus P_{01}\)
to~\(\bigoplus_{b=1}^n \mathcal{R}_{[b,n]} \otimes A[b,n]\).
We first consider the restriction to~\(P_{01}\),
then to~\(P_{10}\).
The matrix coefficients of the map on~\(P_{01}\)
are maps
\(\mathcal{R}_{[a,n]} \otimes Q_0[a,b] \to \mathcal{R}_{[e,n]} \otimes
A[e,n]\)
for \(1 \le a \le b < n\)
and \(1 \le e \le n\).
Such maps correspond to group homomorphisms
\(Q_0[a,b] \to \K_*(\mathcal{R}_{[e,n]}[a,n] \otimes A[e,n])\).
Recall that \(\mathcal{R}_{[e,n]}[a,n] = \C\)
if \(a \le e\)
and~\(0\)
otherwise.  So we may assume without loss of generality that
\(a \le e\),
and then we get corresponding maps \(Q_0[a,b] \to \K_*(A[e,n])\).
The map~\(\gamma_{00}\)
picks out the summand with \(e = M(a,b) = b+1\),
and \(\bigl(\tautr{M(a,b),n}{a,n}\bigr)^*\)
induces the identity map
\(\Z \cong \K_*\bigl(\mathcal{R}_{[a,n]}([b+1,n])\bigr) \to
\K_*\bigl(\mathcal{R}_{[b+1,n]}([b+1,n])\bigr) \cong \Z\).
Therefore, the map in the square through~\(\gamma_{00}\)
contributes the map
\[
\delta_{e,b+1} \tautr{a,b}{e,n}\circ f[a,b]\colon Q_0[a,b]
\to \K_*(A[e,n]).
\]
When we map through~\(\gamma_{01}\)
instead, then we first map \(\mathcal{R}_{[a,n]} \otimes Q_0[a,b]\)
to the direct sum of \(\mathcal{R}_{[g,n]} \otimes Q_0[a,b]\)
over all \(g\in [a,b]\)
using \(\tau^*\otimes\id\)
and then apply the boundary map on~\(W_0\).
This gives a contribution in \(\K_*(A[e,n])\)
if \(g=e\)
or \(g=e-1\),
and these two contributions cancel each other for \(a<e \le b\).
For \(e=b+1\),
we get the same term as for the map that goes through~\(P_{00}\).
And we get~\(0\)
for \(e=a\)
because
\(\tautr{a+1,n}{a,n} \circ \tautr{a,b}{a+1,n} = \tautr{a,b}{a,n} =
0\).  So the two maps are equal on~\(P_{01}\).

The matrix coefficients of the two maps on~\(P_{10}\)
are maps
\(\mathcal{R}_{[M(a,b),n]} \otimes Q_1[a,b] \to \mathcal{R}_{[e,n]}
\otimes A[e,n]\)
for \(1 \le a \le b \le n\)
and \(1 \le e \le n\).
As above, we may assume \(M(a,b) \le e\)
because otherwise any such map is zero.  And then maps
\(\mathcal{R}_{[M(a,b),n]} \otimes Q_1[a,b] \to \mathcal{R}_{[e,n]}
\otimes A[e,n]\)
correspond to maps \(Q_1[a,b] \to \K_*(A[e,n])\).
We shall examine the difference of the map through~\(P_{00}\)
and the map through~\(W_0\).
We first consider the map through~\(P_{00}\).
It first applies the matrix \(\tau^* \otimes \varphi\),
going to the direct sum of
\(\mathcal{R}_{[M(c,d),n]} \otimes Q_0[c,d]\)
for \(1 \le c \le d \le n\).
The map \(Q_1[a,b] \to \K_*(A[e,n])\)
for the composite map through
\(\mathcal{R}_{[M(c,d),n]} \otimes Q_0[c,d]\)
is
\(\delta_{M(c,d),e} \tautr{c,d}{e,n} \circ f[c,d]\circ
\varphi_{[a,b]}^{[c,d]}\).
So we get the sum of these terms over all \(1\le c \le d \le n\).
When we apply \(\gamma_{10}\colon P_{10}\to W_0\),
then we apply the maps
\(\bigl(\tautr{g,n}{M(a,b),n}\bigr)^* \otimes
\varphi_{[a,b]}^{[c,d]}\)
to the direct summands \(\mathcal{R}_{[g,n]} \otimes Q_0[c,d]\)
of~\(W_0\),
where \(1 \le c \le d \le n\)
and \(1 \le g < n\)
are such that \([c,d] \to [g+1,n]\)
and \(M(a,b) \le g < M(c,d)\).
The condition \([c,d] \to [g+1,n]\)
is equivalent to \(c > g\)
if \(d=n\)
and \(c \le g < d+1\)
if \(d<n\).
So the set of~\(g\)
that are allowed is an interval \([x,y]\)
or empty.  The upper bound is always \(y \defeq M(c,d)-1\).
The lower bound~\(x\)
is \(M(a,b)\)
if \(d=n\)
or the maximum of \(c\)
and~\(M(a,b)\)
if \(d<n\).
By convention, we redefine \(x \defeq M(c,d)\)
if the lower bound is bigger than~\(M(c,d)\).
So~\(g\)
runs through the interval~\([x,y]\)
if \(x\le y\),
and otherwise \(x=y+1 = M(c,d)\)
and the set of possible~\(g\) is empty.

We must compose~\(\gamma_{10}\)
with the boundary map on~\(W_0\).
As above, this only contributes to the map
\(\mathcal{R}_{[M(a,b),n]} \otimes Q_1[a,b] \to \mathcal{R}_{[e,n]}
\otimes A[e,n]\)
if \(g=e\)
or \(g=e-1\).
And the contribution to the corresponding map
\(Q_1[a,b] \to \K_*(A[e,n])\)
is \(-\tautr{c,d}{e,n} \circ f[c,d]\circ \varphi_{[a,b]}^{[c,d]}\)
if \(g=e\)
and \(+\tautr{c,d}{e,n} \circ f[c,d]\circ \varphi_{[a,b]}^{[c,d]}\)
if \(g=e-1\).
The contributions for \(g=e\)
and \(g=e-1\)
cancel if both occur.  Therefore, when we sum over all~\(g\)
in the interval~\([x,y]\)
above, we get
\(-\tautr{c,d}{e,n} \circ f[c,d]\circ \varphi_{[a,b]}^{[c,d]}\)
if \(e=x\),
\(\tautr{c,d}{e,n} \circ f[c,d]\circ \varphi_{[a,b]}^{[c,d]}\)
if \(e=y+1\), and~\(0\) otherwise.  So we get the map
\[
(\delta_{e,y+1}-\delta_{e,x}) \cdot
\tautr{c,d}{e,n} \circ f[c,d]\circ \varphi_{[a,b]}^{[c,d]}\colon
Q_1[a,b] \to \K_*(A[e,n]).
\]
This formula remains correct if no~\(g\)
is allowed because then \(x=y+1\).
Since \(y+1 = M(c,d)\),
the map involving~\(\delta_{e,y+1}\)
is equal to the one that we get from the map through~\(P_{00}\).
So when we take the difference of the two maps in the square, this
term is cancelled.  We remain with
\begin{equation}
  \label{eq:sum_to_vanish}
  \sum_{[c,d]\to [a,b]} \delta_{e,x} \cdot \tautr{c,d}{e,n} \circ f[c,d]\circ \varphi_{[a,b]}^{[c,d]},
\end{equation}
where~\(x\)
depends on~\(a,b,c,d\)
as above.  Recall that we only need the case \(e \le M(a,b)\).
We are going to prove that the sum in~\eqref{eq:sum_to_vanish}
vanishes under this assumption.  First we have to study the lower
bound~\(x\)
for the different order relations among \(a,b,c,d\).
We may assume \([c,d] \to [a,b]\)
because otherwise \(\varphi_{[a,b]}^{[c,d]} = 0\).
So either \(a \le c \le b \le d\)
or \(c+1 \le a \le d+1 \le b\).
It is also important whether \(b=n\)
or \(d=n\).
First assume \(b=n\).
So \(M(a,b) = a\).
If \(a \le c \le b = d = n\),
then \(M(c,d) = c\)
and so \(x= a\).
If \(c+1 \le a \le d+1 \le b = n\),
then \(d<n\).
So \(M(c,d) = d+1\)
and \(x=a\)
as well.  Therefore, \(x=a\) and \([e,n] = [a,b]\)
whenever \(b=n\).
In this case, the sum in~\eqref{eq:sum_to_vanish} vanishes because
of~\eqref{eq:tau_f_phi_relation}.

Now assume \(b<n\),
so \(M(a,b) = b+1\).
If \(a \le c \le b < d = n\),
then \(M(c,d) = c < b+1\).
So \(x= M(c,d)\)
and the summand in~\eqref{eq:sum_to_vanish} vanishes because
\(e=x < M(a,b)\).
And \(\tautr{a,b}{b+1,n} \cdot \tautr{c,d}{a,b} =0\)
as well.  If \(a \le c \le b \le d < n\),
then \(M(c,d) = d+1\)
and \(x=b+1\).
In this case, we have
\(\tautr{c,d}{b+1,n} = \tautr{a,b}{b+1,n} \cdot \tautr{c,d}{a,b} \neq
0\).  So we may rewrite the sum in~\eqref{eq:sum_to_vanish} as
\[
\sum_{[c,d]\to [a,b]} \delta_{e,x} \cdot \tautr{c,d}{e,n} \circ f[c,d]\circ \varphi_{[a,b]}^{[c,d]}
= \sum_{[c,d]\to [a,b]} \tautr{a,b}{b+1,n} \circ \tautr{c,d}{a,b}
\circ f[c,d]\circ \varphi_{[a,b]}^{[c,d]},
\]
and this vanishes by~\eqref{eq:tau_f_phi_relation}.  This proves the
vanishing of~\eqref{eq:sum_to_vanish} in all cases and finishes the
proof that the square of maps
\(P_{01} \oplus P_{10} \to \bigoplus \mathcal{R}_{[e,n]} \otimes
A[e,n]\) commutes.

Next, we consider the maps \(P_{11} \to W_0\).
We look at the matrix coefficient
\(\mathcal{R}_{[a,n]} \otimes Q_1[a,b] \to \mathcal{R}_{[e,n]} \otimes
Q_0[c,d]\)
of the maps \(P_{11} \to W_0\)
through \(W_1\),
\(P_{10}\)
and~\(P_{01}\)
for fixed \(1 \le a \le b < n\),
\(1 \le e < n\),
\([c,d] \to [e+1,n]\).
In~\(W_1\),
we have summands \(\mathcal{R}_{[e,n]} \otimes Q_1[a,b]\)
for those~\(e\)
with \([a,b] \to [e+1,n]\),
which is equivalent to \(a \le e \le b\)
because \(b<n\).
The map~\(\gamma_{11}\)
maps the summands \(\mathcal{R}_{[a,n]} \otimes Q_1[a,b]\)
in~\(P_{11}\)
to each of these summands through
\(\bigl(\tautr{e,n}{a,n}\bigr)^* \otimes \id\).
The boundary map \(W_1 \to W_0\)
is obtained by tensoring the boundary map
in~\eqref{eq:varphi_f_exact_sequence_ab} with
\(\id_{\mathcal{R}[e,n]}\).
So we get the contribution
\(\bigl(\tautr{e,n}{a,n}\bigr)^* \otimes \varphi_{[a,b]}^{[c,d]}\)
to our matrix coefficient if \([c,d] \to [a,b] \to [e,n]\),
and~\(0\)
otherwise.  The boundary map to~\(P_{10}\)
maps the summands \(\mathcal{R}_{[a,n]} \otimes Q_1[a,b]\)
in~\(P_{11}\)
to the summand \(\mathcal{R}_{[b+1,n]} \otimes Q_1[a,b]\)
through \(\bigl(\tautr{b+1,n}{a,n}\bigr)^* \otimes \id\).
When we continue with~\(\gamma_{10}\),
we get the contribution
\(\bigl(\tautr{e,n}{a,n}\bigr)^* \otimes \varphi_{[a,b]}^{[c,d]}\)
to our matrix coefficient if and only if \(M(a,b) \le e < M(c,d)\).
The map through~\(P_{01}\)
gives the contribution
\(-\bigl(\tautr{e,n}{a,n}\bigr)^* \otimes \varphi_{[a,b]}^{[c,d]}\)
to our matrix coefficient if \(d<n\)
and \(a \le c \le e\)
(recall that the summands \(\mathcal{R}_{[c,n]} \otimes Q_0[c,d]\)
of~\(P_{01}\) only run over \(1 \le d < n\)).

We must show that the sum of these terms is~\(0\).
Again we look at different cases regarding the order among
\(a,b,c,d,e\).
We may assume \(b<n\)
and \([c,d] \to [e+1,n]\)
because our matrix coefficient is only defined in this case.  And we
may assume \([e,n] \to [a,n]\)
and \([c,d] \to [a,b]\)
because otherwise
\(\bigl(\tautr{e,n}{a,n}\bigr)^* \otimes \varphi_{[a,b]}^{[c,d]}=0\).
Besides \(b<n\),
these assumptions mean, first, that \(e+1 \le c \le n = d\)
or \(c+1 \le e+1 \le d+1 \le n\)
holds; secondly, \(a \le e\);
and, thirdly, \(a \le c \le b \le d\)
or \(c+1 \le a \le d+1 \le b\).
Assume first that \(d=n\).
Then our assumptions imply \(a \le e < c \le b < d = n\).
In this case, none of the three maps \(P_{11} \to W_0\)
give a non-zero contribution because \([c,d] \to [a,b] \to [e,n]\)
is impossible, \(M(a,b) = b+1 > c = M(c,d)\)
and \(d=n\).
So we may assume \(d<n\)
from now on.  If \(c+1 \le a\),
then it follows that \(c+1 \le a \le e< d+1 \le b<n\).
Again, none of the three maps \(P_{11} \to W_0\)
give a non-zero contribution in this case.  So we may assume
\(a \le c\).
Then either \(a \le c \le e \le b \le d < n\)
or \(a \le c \le b < e \le d < n\).
In the first case, the maps through \(W_1\)
and~\(P_{01}\)
give contributions that cancel each other, and the map
through~\(P_{10}\)
gives no contribution because \(e < M(a,b) = b+1\).
In the second case, the maps through \(P_{10}\)
and~\(P_{01}\)
give contributions that cancel each other, and the map through~\(W_1\)
vanishes because \(b<e\).
Hence we get~\(0\)
in all cases, as needed.  This finishes the proof that the square of
maps \(P_{11} \to W_0\) commutes.

Finally, we compute the composite map
\(\delta\circ \gamma_{11} \colon P_{11} \to A\)
in our commuting diagram.  Consider the restriction to
\(\mathcal{R}_{[a,n]} \otimes Q_1[a,b]\)
for some \(1 \le a \le b <n\).
This map corresponds to a map \(Q_1[a,b] \to \K_*(A[a,n])\)
by~\eqref{eq:KK_out_of_RY}.  The map
\(\delta_{e,[c,d]}^\flat = \delta_e^{[c,d]}\)
vanishes unless \(e=c\),
and the matrix coefficient~\(\gamma_{00}^{c,[c,d],[a,b]}\)
vanishes for \([a,b] \neq [c,d]\)
and is the identity map if \([a,b]=[c,d]\).
So the composite map corresponds simply to the map
\[
\sum_{[c,d] \to [a,n] \to [a,b]} \tautr{c,d}{a,n} \circ f[c,d] \circ \varphi_{\7{a,b}}^{\7{c,d}}
\colon Q_1[a,b] \to \K_*(A[a,n]).
\]
This is exactly the formula for the obstruction class
in~\eqref{eq:odd_phi_explicit_filtered}.  This finishes the proof of
Theorem~\ref{the:compare_obstruction_filtered}.

\subsection{The case of extensions}
\label{sec:extensions}

We now specialise to the case \(n=2\).
Then an object of \(\KKcat^X\)
is equivalent to a \(\Cst\)\nb-algebra
extension
\[
I\overset{i}{\into} A \overset{r}{\prto} A/I,
\]
where \(I = A[2]\),
\(A = A[1,2]\)
and \(A/I = A[1]\)
and the maps are those in~\eqref{eq:exact_module_NT}.  Here we
abbreviate \([1]=[1,1]\)
and \([2]=[2,2]\).
The filtrated \(\K\)\nb-theory
is the six-periodic exact chain complex
\begin{equation}
  \label{eq:K_exact_sequence}
  \begin{tikzcd}
    \K_0(I) \arrow[r, "i_*"] &
    \K_0(A) \arrow[r, "r_*"] &
    \K_0(A/I) \arrow[d, "\delta"] \\
    \K_1(A/I) \arrow[u, "\delta"] &
    \K_1(A) \arrow[l, "r_*"] &
    \K_1(I) \arrow[l, "i_*"]
  \end{tikzcd}
\end{equation}
The morphisms between the filtrated \(\K\)\nb-theory
invariants are grading-preserving chain maps (morphisms of six-term
exact sequences).

The invariant in Theorem~\ref{the:Z_classify} is the \(\KK\)\nb-class
\([i]\in \KK_0(I,A)\).
The invariant in Theorem~\ref{the:classification_X_obstruction} is the
induced map \(i_*\colon \K_*(I) \to \K_*(A)\),
together with the obstruction class.  To compute the latter, let
\(i^- \in \Ext\bigl(\K_{1+*}(I), \K_*(A)\bigr)\)
be the parity-reversing part of~\([i]\)
in the Universal Coefficient Theorem for \(\KK_0(I,A)\).
The obstruction class is the image of~\(i^-\)
in the cokernel of the map
\begin{multline}
  \label{eq:obstruction_cokernel}
  \Ext\bigl(\K_{1+*}(I), \K_*(I)\bigr) \oplus \Ext\bigl(\K_{1+*}(A), \K_*(A)\bigr)
  \to \Ext\bigl(\K_{1+*}(I), \K_*(A)\bigr),\\
  (t_I,t_A) \mapsto i\circ t_I + t_A\circ i.
\end{multline}
It follows from our theory that~\(i_*\)
and the image of~\(i^-\)
in the cokernel of~\eqref{eq:obstruction_cokernel} determine an object
of~\(\Boot^X\)
uniquely up to \(\KK^X\)-equivalence.
The cokernel comes in because there are isomorphisms of
\(\Cst\)\nb-algebra
extensions that act identically on \(\K_*(I)\)
and~\(\K_*(A)\),
but have non-trivial components in
\(\Ext\bigl(\K_{1+*}(I), \K_*(I)\bigr)\)
or \(\Ext\bigl(\K_{1+*}(A), \K_*(A)\bigr)\).
So the isomorphism class of an object in~\(\KKcat^X\)
does not determine~\(i^-\)
uniquely.  Only its image in the cokernel
of~\eqref{eq:obstruction_cokernel} is unique.  And our theory shows
that isomorphism classes of pairs consisting of
\(i_*\in \Hom\bigl(\K_*(I),\K_*(A)\bigr)\)
and an element in the cokernel of~\eqref{eq:obstruction_cokernel} are
in bijection with isomorphism classes of objects in the bootstrap
class \(\Boot^X \subseteq \KKcat^X\).

We are going to compare this classification result with the filtrated
\(\K\)\nb-theory
classification by the long exact sequences
in~\eqref{eq:K_exact_sequence}.  The long exact sequence
in~\eqref{eq:K_exact_sequence} contains~\(i_*\) and the extension
\begin{equation}
  \label{eq:extension_from_six-term}
  \coker\bigl(i\colon \K_*(I) \to \K_*(A)\bigr)
  \into \K_*(A/I)
  \prto \ker\bigl(i\colon \K_{*+1}(I) \to \K_{*+1}(A)\bigr),
\end{equation}
and we may reconstruct the long exact sequence from these two pieces.
Two extensions as in~\eqref{eq:extension_from_six-term} have the same
class in Ext if and only if the long exact sequences associated to
them (for the same~\(i_*\))
are isomorphic with an isomorphism that is the identity on \(\K_*(I)\)
and~\(\K_*(A)\).
Therefore, the filtrated \(\K\)\nb-theory
invariant is equivalent to the pair consisting of
\(i_*\colon \K_*(I) \to\K_*(A)\)
and a class in \(\Ext\bigl(\ker(i_*), \coker(i_*)\bigr)\).
Now the following proposition clarifies the relationship between our
different invariants:

\begin{proposition}
  \label{pro:obstruction_class_for_extensions}
  The cokernel in~\eqref{eq:obstruction_cokernel} is naturally
  isomorphic to the group
  \[
  \Ext^2_{\Abel^X}(\Sigma A,A) \cong
  \Ext\Bigl(\ker\bigl(\K_{*+1}(I) \xrightarrow{i_*} \K_{*+1}(A)\bigr),
  \coker\bigl(\K_*(I) \xrightarrow{i_*} \K_*(A)\bigr)\Bigr).
  \]
  And the obstruction class is the class that corresponds to minus the
  extension in~\eqref{eq:extension_from_six-term}.
\end{proposition}

\begin{proof}
  Let~\(G\)
  be a \(\Z/2\)\nb-graded
  Abelian group.  By the long exact sequence for Hom and Ext,
  the restriction map
  \(\Ext\bigl(\K_{*+1}(A), G\bigr) \to \Ext\bigl(i_*(\K_{*+1}(I)),G\bigr)\)
  is surjective and
  \begin{multline*}
    \dotsb \to
    \Ext\bigl(i_*(\K_{*+1}(I)),G\bigr)
    \to \Ext\bigl(\K_{*+1}(I),G\bigr)
    \\\to \Ext\bigl(\ker(i_*\colon \K_{*+1}(I) \to \K_{*+1}(A)),G\bigr)
    \to 0
  \end{multline*}
  is exact.  Hence the cokernel of the map
  \(i^*\colon \Ext\bigl(\K_{*+1}(A),G\bigr) \to \Ext\bigl(\K_{*+1}(I),G\bigr)\)
  is \(\Ext(\ker(i_*),G)\).
  If we let \(G \defeq \K_*(A)\),
  then we may identify the cokernel of the map
  in~\eqref{eq:obstruction_cokernel} with the cokernel of the map
  \[
  \Ext\bigl(\K_{*+1}(I),\K_*(I)\bigr) \xrightarrow{\text{restrict}}
  \Ext\bigl(\ker(i_*),\K_*(I)\bigr)
  \xrightarrow{i_*} \Ext\bigl(\ker(i_*), \K_*(A)\bigr).
  \]
  Since the first map is surjective, this is equal to the cokernel of
  \[
  i_*\colon \Ext\bigl(\ker(i_*),\K_*(I)\bigr) \to \Ext\bigl(\ker(i_*), \K_*(A)\bigr).
  \]
  A variant of the proof above for the second variable identifies this
  cokernel with the group \(\Ext\bigl(\ker(i_*), \coker(i_*)\bigr)\)
  as claimed.  Given a class
  \(\delta\in\Ext\bigl(\K_{1+*}(I), \K_*(A)\bigr)\),
  the map to \(\Ext\bigl(\ker(i_*), \coker(i_*)\bigr)\)
  simply applies the bifunctoriality of Ext for the quotient map
  \[
  \K_*(A)\to \coker\bigl(i_*\colon \K_*(I) \to \K_*(A)\bigr)
  \]
  and the inclusion map
  \[
  \ker\bigl(i_*\colon \K_{*+1}(I) \to \K_{*+1}(A)\bigr) \to \K_{*+1}(I).
  \]
  It remains to compute the image of the obstruction class in
  Theorem~\ref{the:compare_obstruction_filtered} in
  \(\Ext\bigl(\ker(i_*),\coker(i_*)\bigr)\).
  We recall the data used in
  Theorem~\ref{the:compare_obstruction_filtered} in our special case.
  The semi-simple part of the filtrated \(\K\)\nb-theory consists of
  \begin{align*}
    \K_*(A[2])_\mathrm{ss} &\cong \frac{\K_*(I)}{\delta(\K_{*-1}(A/I))}
                             \cong i_*(\K_*(I)) \subseteq \K_*(A),\\
    \K_*(A[1,2])_\mathrm{ss} &\cong \frac{\K_*(A)}{i_*(\K_*(I))}
                               = \coker(i_*),\\
    \K_*(A[1])_\mathrm{ss} &\cong \frac{\K_*(A/I)}{r_*(\K_*(A))}
                             \cong \delta_*(\K_*(A/I)) = \ker(i_*)\subseteq \K_{*-1}(I).
  \end{align*}
  Our construction is based on free resolutions of these
  \(\Z/2\)\nb-graded
  Abelian groups.  In particular, we use a free resolution
  \(Q_1[1] \into Q_0[1] \prto \ker(i_*)\).
  We lift these resolutions to free resolutions
  \begin{gather*}
    Q_1[2] \oplus Q_1[1] \into Q_0[2] \oplus Q_0[1] \prto \K_*(A[2]),\\
    Q_1[1,2] \oplus Q_1[2] \into Q_0[1,2] \oplus Q_0[2] \prto \K_*(A[1,2]),\\
    Q_1[1] \oplus Q_1[1,2] \into Q_0[1] \oplus Q_0[1,2] \prto \K_*(A[1]),
  \end{gather*}
  which contain the maps \((\tau f[a,b])^\#\)
  and \(\varphi_{[a,b]}^{[c,d]}\).
  We shall need the maps in the third extension and put them into a
  larger diagram, which commutes because of the construction of the
  maps \((\tau f[a,b])^\#\) and \(\varphi_{[a,b]}^{[c,d]}\):
  \begin{equation}
    \label{eq:extension_filtered_K_theory_extension}
    \begin{tikzpicture}[baseline=(current bounding box.west)]
      \matrix(m)[cd,column sep = 6em]{
        Q_1[1] \oplus Q_1[1,2]&
        Q_0[1] \oplus Q_0[1,2]&
        \K_*(A/I)\\
        Q_1[1]&
        Q_0[1]&
        \frac{\K_*(A/I)}{r_*(\K_*(A))}\\
        r_*(\K_*(A))&
        \K_*(A/I)&
        \ker(i_*)\\};
      \draw[cdar,>->](m-1-1) -- node[yshift=2.2ex] {\(\scriptstyle\matrixintikz\)} (m-1-2);
      \draw[cdar, ->>] (m-1-2) -- node {\(\scriptstyle(f\7{1}\ r_* f\7{1,2})\)} (m-1-3);
      \draw[cdar,>->](m-2-1) -- node {\(\scriptstyle\varphi_{\7{1}}^{\7{1}}\)} (m-2-2);
      \draw[cdar, ->>] (m-2-2) -- node {\(\scriptstyle d_0\)} (m-2-3);
      \draw[cdar,>->](m-3-1) -- node {\(\scriptstyle \text{incl.}\)} (m-3-2);
      \draw[cdar, ->>] (m-3-2) -- node {\(\scriptstyle \delta_*\)} (m-3-3);
      \draw[cdar, ->>] (m-1-1) -- node {\(\scriptstyle\mathrm{pr}_1\)} (m-2-1);
      \draw[cdar, ->>] (m-1-2) -- node {\(\scriptstyle\mathrm{pr}_1\)} (m-2-2);
      \draw[cdar, ->>] (m-1-3) -- node {\(\scriptstyle\mathrm{can.}\)} (m-2-3);
      \draw[cdar, ->>] (m-2-1) -- node
      {\(\scriptstyle f\7{1}\circ \varphi_{\7{1}}^{\7{1}}\)} (m-3-1);
      \draw[cdar, ->>] (m-2-2) -- node {\(\scriptstyle f\7{1}\)} (m-3-2);
      \draw[cdar] (m-2-3) -- node[swap] {\(\scriptstyle\cong\)} node {\(\scriptstyle\delta\)} (m-3-3);
    \end{tikzpicture}
  \end{equation}
  The matrix coefficient~\(\delta_e^{[a,b]}\)
  in Theorem~\ref{the:compare_obstruction_filtered} is defined only if
  \(e = a \le b<n\).
  In our case \(n=2\),
  there is only one such matrix coefficient, namely,
  \(\delta_1^{[1]} \colon Q_1[1] \to \K_*(A[1,2]) = \K_*(A)\).
  The sum defining it has only one summand, which is indexed by
  \([1,2] \to [1,2] \to [1]\).  So
  \[
  \delta_1^{[1]} = f[1,2] \circ \varphi_{[1]}^{[1,2]} \colon Q_1[1] \to \K_*(A).
  \]
  We compose this map with the quotient map
  \(\K_*(A) \prto \coker(i_*)\).
  The exact sequence~\eqref{eq:K_exact_sequence} shows that~\(r_*\)
  induces an isomorphism from \(\coker(i_*)\)
  onto \(r_*(\K_*(A)) \subseteq \K_*(A/I)\).
  So we may as well compose with the map
  \(r_*\colon \K_*(A) \to r_*(\K_*(A))\).
  The composite of the two maps in the top row
  in~\eqref{eq:extension_filtered_K_theory_extension} is~\(0\).
  Thus \(r^* \circ \delta_1^{[1]} = - f[1] \circ \varphi_{[1]}^{[1]}\)
  (compare~\eqref{eq:tau_f_phi_relation}).  The commuting
  diagram~\eqref{eq:extension_filtered_K_theory_extension} shows that
  the group extension of~\(\ker(i_*)\)
  in~\eqref{eq:extension_from_six-term} belongs to the map
  \(f[1] \circ \varphi_{[1]}^{[1]}\).
  The obstruction class belongs to the negative of this map.
\end{proof}

\end{document}